\documentclass{amsart}

\usepackage{etex}
\usepackage{amsmath,amssymb,amsthm,amsfonts,mathrsfs}
\usepackage[frame,cmtip,arrow,matrix,line,graph,curve]{xy}
\usepackage{graphpap,color,paralist,pstricks}
\usepackage[mathscr]{eucal}
\usepackage{mathabx}
\usepackage[pdftex,colorlinks,backref=page,citecolor=blue]{hyperref}
\usepackage{tikz}
\usetikzlibrary{calc,decorations.markings}
\usepackage{epic,eepic}
\usepackage{yfonts}
\usepackage{enumitem}
\usepackage{bbm}
\usepackage{tikz-cd}

\newtheorem{theorem}{Theorem}[section]
\newtheorem{lemma}[theorem]{Lemma}

\allowdisplaybreaks[4]

\theoremstyle{definition}

\newtheorem{proposition}[theorem]{Proposition}

\theoremstyle{remark}
\newtheorem{remark}[theorem]{Remark}

\numberwithin{equation}{section}

\begin{document}
\UseRawInputEncoding

\title[\hfilneg \hfil
Blowup dynamics for equivariant critical LL flow]{Blowup dynamics for equivariant critical Landau--Lifshitz flow}
\author{Fangyu Han}
\address{School of Mathematical Sciences, Xiamen University, Xiamen 361005, People's Republic of China}
\email{fangyuh88@163.com}
\thanks{Corresponding author: Z. Tan (tan85@xmu.edu.cn)}
\author{Zhong Tan}
\address{School of Mathematical Sciences, Xiamen University, Xiamen 361005, People's Republic of China}
\address{Shenzhen Research Institute of Xiamen University, Shenzhen 518057, People's Republic of China}
\email{tan85@xmu.edu.cn}
%\date{January 1, 2001 and, in revised form, June 22, 2001.}

\begin{abstract}
The existence of finite time blowup solutions for the two-dimensional Landau--Lifshitz equation is a long-standing problem, which exists in the literature at least since 2001 (E, Mathematics Unlimited--2001 and Beyond, Springer, Berlin, P.410, 2001). A more refined description in the equivariant class is given in (van den Berg and Williams, European J. Appl. Math., 24(6), 912--948, 2013). In this paper, we consider the blowup dynamics of the Landau--Lifshitz equation
$$
\partial_tu=\mathfrak{a}_1u\times\Delta u-\mathfrak{a}_2u\times(u\times\Delta u),\quad x\in\mathbb{R}^2,
$$
where $u\in\mathbb{S}^2$, $\mathfrak{a}_1+i\mathfrak{a}_2\in\mathbb{C}$ with $\mathfrak{a}_2\geq0$ and $\mathfrak{a}_1+\mathfrak{a}_2=1$. We prove the existence of 1-equivariant Krieger--Schlag--Tataru type blowup solutions near the lowest energy steady state. More precisely, we prove that for any $\nu>1$, there exists a 1-equivariant finite-time blowup solution of the form
$$
u(x,t)=\phi(\lambda(t)x)+\zeta(x,t),\quad \lambda(t)=t^{-1/2-\nu},
$$
where $\phi$ is a lowest energy steady state and $\zeta(t)$ is arbitrary small in $\dot{H}^1\cap\dot{H}^2$. The proof is accomplished by renormalizing the blowup profile and a perturbative analysis in the spirit of (Krieger, Schlag and Tataru, Invent. Math., 171(3), 543--615, 2008), (Perelman, Comm. Math. Phys., 330(1), 69--105, 2014) and (Ortoleva and Perelman, Algebra i Analiz, 25(2), 271--294, 2013).
\end{abstract}

\maketitle

\section{Introduction and main result}\label{Sec1}
\subsection{Introduction}
The Landau--Lifshitz flow from the $m$-dimensional Riemannian manifold $(\mathcal{M},g)$ to the two-sphere $\mathbb{S}^2$ is given by
\begin{equation}\label{IM1}
\begin{cases}
\partial_t u=\mathfrak{a}_1u\times\Delta_{\mathcal{M}} u-\mathfrak{a}_2u\times(u\times\Delta_{\mathcal{M}} u), \quad x\in\mathcal{M},\,\, t\in\mathbb{R},\\
u|_{t=0}=u_0\in\mathbb{S}^2, \quad x\in\mathcal{M},
\end{cases}
\end{equation}
where $\mathfrak{a}_1+i\mathfrak{a}_2\in\mathbb{C}$ with $\mathfrak{a}_2\geq0$ and  $\mathfrak{a}_1+\mathfrak{a}_2=1$, $u=(u_1,u_2,u_3)$ is a three-dimensional vector with normalized length that satisfies $u(x,t):\mathcal{M}\times\mathbb{R}\rightarrow\mathbb{S}^2$, $g=|\det(g_{ij})|$ is the Riemann metric, $\Delta_{\mathcal{M}}$ is the Laplace--Beltrami operator defined by $\Delta_{\mathcal{M}}u=\frac{1}{\sqrt{g}} \partial_{x_i}(g^{ij}\sqrt{g}\partial_{x_j}u)$, where $(g^{ij})$ is the inverse of $(g_{ij})$. This is an important model first developed by Landau and Lifshitz \cite{LL35} to model the effects of magnetic fields on ferromagnetic materials and to describe the evolution of continuous spin fields in ferromagnets.

In fact, the Landau--Lifshitz flow is closely related to some other important geometric flows, for instance, the harmonic map heat flow and the Schr\"odinger map flow.

\subsubsection{Harmonic heat flow}
When $\mathfrak{a}_1=0$ and $\mathfrak{a}_2=1$, \eqref{IM1} becomes a parabolic harmonic heat flow:
\begin{equation}
\text{(Harmonic heat flow)}\quad
\begin{cases}
\partial_tu=\Delta_{\mathcal{M}} u+|\nabla u|^2u,\quad x\in\mathcal{M},\, t\in\mathbb{R},\\
u|_{t=0}=u_0,\quad x\in\mathcal{M},
\end{cases}
\end{equation}
where $u(x,t)\in \mathbb{S}^2$ and $|\nabla u|^2=\sum_{i,j} \sum_kg^{ij} \partial_{x_i}u_k \partial_{x_j} u_k$. This is an important model in liquid crystal flow and ferromagnetism (see, e.g.,  \cite{BBCH92}\cite{BHK03}). In addition, it is also related to the harmonic map. The harmonic map $u$ satisfies the Euler--Lagrange equation: $\Delta_{\mathcal{M}} u+|\nabla u|^2u=0$, the theory of which was first established in 1964 by Eells and Sampson \cite{ES64}, who proved that any map can be deformed into a harmonic map in a certain geometric context.

When $\mathcal{M}$ is a Riemann surface, Struwe \cite{Struwe85} proved the existence and uniqueness of weak solutions with at most finitely many singularities. For a further extension of this conclusion and for the higher dimensional case, see \cite{Freire95}\cite{CS89}\cite{Struwe88}. Chang, Ding and Ye \cite{CDY92} constructed the first example of finite-time blowup solutions for the harmonic heat flow.  For the case where the initial value is defined on $\mathbb{D}^2 \subset \mathbb{R}^2$ and the target manifold is $\mathbb{S}^2$, van den Berg, Hulshof and King \cite{BHK03} used formal asymptotic analysis to predict the existence of blowup solutions with quantifiable blowup rate
$$
\lambda_L(t)\approx C\frac{|T-t|^L}{|\ln (T-t)|^{\frac{2L}{2L-1}}},\quad L\in\mathbb{N}^*.
$$
Since the heat flow in two dimensions is energy critical, the formation of singularity by energy concentration is possible. It is well known that concentration implies non-trivial harmonic map of bubbles at a finite number of blowup points, see for instance  \cite{CG89}\cite{CD90}\cite{DT95}\cite{LW98}\cite{Qing95}\cite{QT97} \cite{Struwe85}\cite{Topping04}\cite{Wang96} for more details. For the case of $\mathcal{M}=\mathbb{R}^2$, Gustafson, Nakanish and Tsai \cite{GNT10} proved that the asymptotical stability of the $k$-equivariant harmonic map for $k\geq3$ and gave a class of infinite-time equivariant blowup solutions near the $2$-equivariant harmonic map. Rapha\"el and Schweyer \cite{RS13}\cite{RS14} have selected a family of initial values which are arbitrarily close to the lowest energy harmonic map under the energy critical topology, and proved that the corresponding solutions blowup in finite time with rate  $\lambda_L(t)$, where $L\geq1$ is arbitrary. The case of $L=1$ corresponds to a stable regime. When there is no assumption of symmetry, D\'avila, del Pino and Wei \cite{DPW20} construct a solution in a bounded region in $\mathbb{R}^2$, which blowup exactly at pre-given finite number of points, at each of which the blowup profile is close to the asymptotic singularity expansion of the 1-corotational harmonic map and the blowup rate is $\lambda_L(t)$ with $L=1$. This rate is similar to that expected in 1-corrotational heat flow, see \cite{BHK03}. For the existence and uniqueness results, please refer to \cite{Chang89}\cite{CD90}\cite{ES64}\cite{LW98}\cite{LW08} and the references therein.

\subsubsection{Schr\"odinger map flow}
When $\mathfrak{a}_1=1$ and $\mathfrak{a}_2=0$, \eqref{IM1} becomes the Schr\"odinger map flow, which is a fundamental content in differential geometry, see\cite{CSU00}\cite{Ding02}\cite{DW98}\cite{TU06}. By the action of the complex structure $u\times$, the Schr\"odinger map can be written as
\begin{equation}
\text{(Schr\"odinger map)}\quad
\begin{cases}
u\times\partial_t u=-\Delta u -|\nabla u|^2u,\quad x\in\mathcal{M},\, t\in\mathbb{R},\\
u|_{t=0}=u_0,\quad x\in\mathcal{M},
\end{cases}
\end{equation}
where $u(x,t)\in \mathbb{S}^2$.

The local well-posedness of Schr\"odinger map can be found in \cite{DW01}\cite{McGahagan07}\cite{SSB86}. When the target manifold is $\mathbb{S}^2$, Bejenaru et al. \cite{BIKT11} proved the global well-posedness with small data in the critical space. Their results were generalized by Li \cite{Li21}\cite{Li22} to the case of K\"ahler manifold target. The static solution of the Schr\"odinger flow is a harmonic map. When the energy is less than $4\pi$, the 1-equivariant solutions are global in time and scattering (see \cite{BIKT13}). Gastafson et al. \cite{GKT08}\cite{GKT07}\cite{GNT10} proved that the harmonic map is asymptotically stable with respect to the Schr\"odinger map in the $k$-equivariant class for $k\geq3$, which shows that such solutions do not blowup near the harmonic map. The case of $k=2$ is still an important open problem. However, in the 1-equivariant class, Bejenaru and Tataru \cite{BT14} proved that the harmonic map is stable under a smooth well-localized perturbation, but unstable in the $\dot{H}^1$ topology. Merle, Rapha\"el and Rodnianski \cite{MRR13} proved the existence of a codimension one set of smooth well localized initial data arbitrarily close to the ground state harmonic map, which generates finite time type II blowup solutions. They also gave a sharp description of the corresponding singularity formation. Perelman \cite{Perelman14} proved the existence of another class type II blowup solutions with a different blowup behavior. For more results on the global well-posedness of solutions near the ground state, see \cite{Bejenaru08}\cite{BIK07}\cite{BIKT11}\cite{BT14}\cite{IK07} and the references therein.

\subsubsection{Landau--Lifshitz flow}
Landau--Lifshitz flow \eqref{IM1} was first proposed in the study of classically continuous isotropic Heisenberg ferromagnetic chains. It describes the evolution of magnetic moments in classical ferromagnetic and anti-ferromagnetic chains, which is an important basis for understanding the non-stationary magnetism (see, e.g., \cite{LL35}\cite{Wijn66}).

For the global existence and partial regularity of weak solutions, see for instance \cite{AS92}\cite{CF01}\cite{GH93}\cite{Ko05}\cite{Melcher05}\cite{Wang06}. In particular, when $\mathcal{M}$ is a Riemannian surface, Guo and Hong \cite{GH93} proved uniqueness of weak solutions and regularity except for at most finitely many points. When $\mathcal{M}=\mathbb{R}^2$, Ko \cite{Ko05} constructed a smooth solution away from a two-dimensional locally finite Hausdorff measure set by using the discretization approximation method. In general, for the high-dimensional weak solutions, we expect a better partial regularity results (i.e., no further assumptions of regularity or minimal energy), for example, there is a well-known example constructed by Rivi\`{e}re \cite{Riviere95}: There exists a weak harmonic map from the ball $B^3\subset\mathbb{R}^3$ to $\mathbb{S}^2$, whose singular set is the closure of the ball $\overline{B^3}$, and this conclusion also holds in the higher dimensions. Following the idea of this example, Chen and Struwe \cite{CS89} proved the existence of partially regular solutions for high-dimensional harmonic heat flows, and Melcher \cite{Melcher05} proved that the existence of global weak solutions for the Landau--Lifshitz flow in $\mathbb{R}^3$, where the singular set has finite three-dimensional parabolic Hausdorff measure. Wang \cite{Wang06} generalized Melcher's result to the case of $\mathcal{M}=\mathbb{R}^m$, $m\leq4$. If an assumption on the stability of weak solutions is attached, Moser \cite{Moser02} obtained a better estimate for the singular set.

Although the Landau--Lifshitz flow has been studied extensively, there are few studies on its dynamical behavior. In the $m$-equivariant class ($m\geq3$), Gustafson, Nakanishi and Tsai \cite{GNT10} proved the stability of the harmonic map for Landau--Lifshitz flow. In the 1-equivariant class, Li and Zhao \cite{LZ17} proved that the solutions with energy less than $4\pi$ converge to a constant map in the energy space. van den Berg and Williams \cite{BW13} obtained equivariant blowup solutions by formal expansion and verified them experimentally, but as the author stated in \cite{BW13}: ``mathematically rigourous justification is required''. The blowup dynamics of the 1-equivariant Landau--Lifshitz equation near the equivariant harmonic map is an important open problem. In 2020, Xu and Zhao \cite{XZ20} proved the existence of a codimension one set of smooth well localized initial data arbitrarily close to the ground state harmonic map, which generates finite time type II blowup 1-equivariant solutions. They also gave a sharp description of the corresponding singularity formation. Recently, based on the inner-outer gluing method and the distorted Fourier transform, Wei, Zhang and Zhou \cite{WZZ} constructed a finite-time blowup solution in $\mathbb{R}^2$ without any symmetry.

The purpose of this paper is to consider the Landau--Lifshitz flow \eqref{IM1} with $\mathcal{M}=\mathbb{R}^2$, and to prove the existence of type II blowup solutions in the 1-equivariant class. This blowup solution has a continuous blowup rate, therefore, it different from that constructed in \cite{XZ20} and \cite{WZZ}.

\subsection{Model and main result}
\subsubsection{Setting of the problem}
In this paper, we consider the initial value problem of the Landau--Lifshitz flow from $\mathbb{R}^2$ to $\mathbb{S}^2$:
\begin{equation}\label{SM}
\begin{cases}
\partial_tu=\mathfrak{a}_1u\times\Delta u-\mathfrak{a}_2u\times(u\times\Delta u),\quad x=(x_1,x_2)\in\mathbb{R}^2,\,\, t\in\mathbb{R},\\
u|_{t=0}=u_0,
\end{cases}
\end{equation}
where $u(t,x)=(u_1(x,t),u_2(x,t),u_3(x,t))\in\mathbb{S}^2\subset\mathbb{R}^3$ and  $\mathfrak{a}_1+i\mathfrak{a}_1\in\mathbb{C}$ with $\mathfrak{a}_2\geq0$ and $\mathfrak{a}_1+\mathfrak{a}_2=1$.

Equation \eqref{SM} conserves the energy
\begin{equation}\label{energy}
E(u)=\frac{1}{2}\int_{\mathbb{R}^2}|\nabla u(x,t)|^2dx.
\end{equation}
The two-dimensional problem \eqref{SM} is critical in the sense that \eqref{energy} is invariant with respect to the scaling $u(x,t)\rightarrow u(\lambda x,\lambda^2t)$, where $\lambda\in\mathbb{R}_+=\{x\in\mathbb{R}|x>0\}$.

For a finite energy map $u: \mathbb{R}^2\rightarrow\mathbb{S}^2$, we can define its topological degree as
$$
\deg(u)=\frac{1}{4\pi}\int_{\mathbb{R}^2}u_{x_1}\cdot J_uu_{x_2}dx,
$$
where $J_u$ is a complex structure on $\mathbb{S}^2$ defined by
$$
J_uv=u\times v,\quad v\in\mathbb{R}^3.
$$
According to \eqref{energy}, we get
\begin{equation}\label{Edeg}
E(u)\geq 4\pi|\deg(u)|,
\end{equation}
where the equality is achieved at the harmonic map $\phi_m$ (see \cite{GNT10}):
\begin{equation}\label{HHH13}
\begin{split}
&\phi_m(x)=e^{m\theta R}Q^m(r),\quad Q^m=\left(h_1^m,0,h_3^m\right)\in\mathbb{S}^2,\\
&h_1^m(r)=\frac{2r^m}{r^{2m}+1},\quad h_3^m(r)=\frac{r^{2m}-1}{r^{2m}+1}.
\end{split}
\end{equation}
Here $m\in\mathbb{Z}^+$, $(r,\theta)$ is the polar coordinate in the plane $\mathbb{R}^2: x_1+ix_2=e^{i\theta}r$ and $R$ is the generator of horizontal rotations:
$$
R=\left(
    \begin{array}{ccc}
      0 & -1 & 0 \\
      1 & 0 & 0 \\
      0 & 0 & 0 \\
    \end{array}
  \right),
$$
which can also be equivalently written as
$$
Ru=\mathbf{k}\times u,\quad \mathbf{k}=(0,0,1).
$$
A direct calculation gives
$$
\deg(\phi_m)=m,\quad E(\phi_m)=4\pi m.
$$
Up to the symmetries, $\phi_m$ are the only energy minimizer in their homotogy class.

Since $\phi_1$ is crucial in the rest of this paper, we write $\phi=\phi_1$, $Q=Q_1$, $h_1=h_1^1$ and $h_3=h_3^1$.

\subsubsection{Main result}
Based on reformatting blowup profile and perturbation method, Krieger, Schlag and Tataru \cite{KST08} proved that the energy of solutions for the equivariant critical wave map concentrates in the cuspidal region:
$$
0\leq t\lesssim \frac{1}{\lambda_0(t)},\quad\lambda_0(t)=\frac{1}{t^{1+\nu_0}},\,\,\nu_0>\frac{1}{2},
$$
thus they obtained a class of solutions that blowup at $r=t=0$. We call this type of blowup solutions as Krieger--Schlag--Tataru type blowup solutions. The aim of this paper is to prove that \eqref{SM} also exists 1-equivariant Krieger--Schlag--Tataru type blowup solutions, where the initial data of the form
$$
u_0=\phi+\zeta_0,
$$
where $\zeta_0$ is 1-equivariant and is arbitrarily small in $\dot{H}^1\cap \dot{H}^3$.

The main result of this paper is as follows.
\begin{theorem}[Existence of the Krieger--Schlag--Tataru type blowup solution]\label{MT}
For any $\nu>1$ and $\alpha_0\in\mathbb{R}$, let $\delta>0$ be sufficiently small, then there exists $t_0>0$ such that \eqref{SM} exists a 1-equivariant solution $u\in C((0,t_0],\dot{H}^1\cap\dot{H}^3)$ of the form:
\begin{equation}\label{T1}
u(x,t)=e^{\alpha(t)R}\phi\left(\lambda(t)x\right)+\zeta(x,t),
\end{equation}
where
\begin{equation}\label{T2}
\lambda(t)=t^{-\frac{1}{2}-\nu},\quad \alpha(t)=\alpha_0\ln t,
\end{equation}
\begin{equation}\label{T3}
\|\zeta(t)\|_{\dot{H}^1\cap\dot{H}^2}\leq\delta,\quad \|\zeta(t)\|_{\dot{H}^3}\leq C_{\nu,\alpha_0}\frac{1}{t},\quad \forall t\in(0,t_0].
\end{equation}
Furthermore, $\zeta(t)\rightarrow\zeta^*$ in $\dot{H}^1\cap\dot{H}^2$ as $t\rightarrow0$, where $\zeta^*\in H^{1+2\nu-}$, $\nu-$ means any positive number less than $\nu$.
\end{theorem}

Here are some comments on the result.
\begin{remark}
Singularity formation of finite energy solutions of two-dimensional Landau--Lifshitz flow \eqref{SM} is an open problem, which was proposed by E (see Subsection 2.1 in \cite{E01}), Ding and Wang (see Remark 1.6 in \cite{DW07}), Guo and Ding (see Preface in \cite{GD08}) and Gustafson, Nakanishi and Tsai (see Section 1 in \cite{GNT10}), etc. For a more refined description of this problem in the equivariant class see \cite{BW13}. In this paper, we prove the existence of a continuous blowup solution for the two-dimensional Landau--Lifshitz flow \eqref{SM} in 1-equivariant class. The idea of proof is based on the reformatting blowup profile and perturbation method, which is very similar to the studies of Krieger, Schlag and Tataru \cite{KST08}, Perelman \cite{Perelman14} and Ortoleva and Perelman \cite{OP13}.
\end{remark}

\begin{remark}
Note that Zhao and Xu \cite{XZ20} proved the existence of finite time 1-equivariant type II blowup solutions with codimension one, and Wei, Zhang and Zhou \cite{WZZ} constructed a finite time type II blowup solution without symmetric assumption. Compared with the results of \cite{XZ20} and \cite{WZZ}, we give a class of 1-equivariant type II blowup solutions with continuous blowup rate. Therefore, our solution has a different singularity regime from theirs.
\end{remark}

\begin{remark}
In fact, similar to the discussion in this paper, the result also holds when $\dot{H}^3$ is replaced by $\dot{H}^{1+2s}$ in Theorem \ref{MT}, where $1\leq s<\nu$.
\end{remark}

\begin{remark}
When $\mathfrak{a}_1=1$ and $\mathfrak{a}_2=0$, \eqref{SM} becomes the Schr\"odinger map flow, in this case Theorem \ref{MT} also holds, see \cite{Perelman14} for more details. However, due to the appearance of $\mathfrak{a}_2\neq0$ and $\mathfrak{a}_1$, the equation behaved as parabolic heat flow property, which is characterized particularly by the corresponding complex coefficients in the profiles of the self-similar and remote regions. This makes it difficult to match the self-similar and remote regions with the inner region. Fortunately, in the self-similar region, we found that the coefficients consisting of $\mathfrak{a}_1$ and $\mathfrak{a}_2$ have some elimination regime. Indeed, we observe that there exists a basis $\{f_j^1,f_j^2\}$ of solutions for equation $(\mathcal{L}-\tilde{\mu}_j)f=0$, and $f_j^1$ and $f_j^2$ have asymptotic expansions at infinity, which do not contain $\mathfrak{a}_1$ and $\mathfrak{a}_2$ in the power of $y$ (see Lemma \ref{lem2.6} in Subsection \ref{subsec2.3}). This allows us to construct solutions in the remote region that match the asymptotic expansion of the solutions in the inner region at the origin (see Subsection \ref{subsec2.4}).
\end{remark}

\subsection{Strategy of the proof}
The proof of Theorem \ref{MT} consists of two steps, which are in Sections \ref{Sec2} and \ref{Sec3}, respectively.

In Section \ref{Sec2}, we construct approximate solutions $u^{(N)}$ that have the form \eqref{T1}, \eqref{T2} and \eqref{T3}, and satisfy \eqref{SM} up to an arbitrary high order error $O(t^N)$.

In Section \ref{Sec3}, by solving a time-forward problem with zero initial value at $t=0$ with respect to the remainder (see Proposition \ref{pro3.1}), we solve the equation \eqref{SM} exactly. The control of the remainder is obtained by energy estimates (see Section \ref{Sec3} for details), where the assumption $\nu>1$ ensures that the approximate solutions we construct belong to $\dot{H}^1\cap\dot{H}^3$, so that we can work in the framework of $H^3$ well-posedness theory.

\section{Approximate solutions}\label{Sec2}

\subsection{Preliminaries and main result of present section}
We consider the 1-equivariant solutions of \eqref{SM}, i.e.,
\begin{equation}\label{E2.1}
u(x,t)=e^{\theta R}v(r,t),\quad v=(v_1,v_2,v_3)\in\mathbb{S}^2\subset\mathbb{R}^3.
\end{equation}
Thus, \eqref{SM} restricted to the 1-equivariant class yields
\begin{equation}\label{E2.2}
v_t=\mathfrak{a}_1v\times\left(\Delta v+\frac{R^2}{r^2}v\right)-\mathfrak{a}_2v\times\left[v\times\left(\Delta v+\frac{R^2}{r^2}v\right)\right],
\end{equation}
and the corresponding energy is
$$
E(u)=\pi\int_0^\infty \left(\left|v_r\right|^2+\frac{v_1^2+v_2^2}{r^2}\right)r dr.
$$
Note that $Q=(h_1,0,h_3)$ is a static solution of \eqref{E2.2} satisfying the following identities:
\begin{equation}\label{E2.3}
\begin{split}
&\partial_rh_1=-\frac{h_1h_3}{r},\quad \partial_r h_3=\frac{h_1^2}{r},\\
&\Delta Q+\frac{R^2}{r^2}Q=\kappa(r)Q,\quad \kappa(r)=-\frac{2h_1^2}{r^2}.
\end{split}
\end{equation}

The main purpose of this section is to prove the following proposition.
\begin{proposition}\label{pro2.1}
For any $\delta>0$ sufficiently small and any $N$ sufficiently large, there is a approximate solution $u^{(N)}:\mathbb{R}^2\times\mathbb{R}^*_+ \rightarrow\mathbb{S}^2$ of \eqref{SM}, where $\mathbb{R}^*_+=\{x\in\mathbb{R}|x\geq0\}$. Moreover, $u^{(N)}$ satisfies the following estimates:
\begin{description}
  \item[(i)] $u^{N}$ is a $C^\infty$ 1-equivariant profile of the form
   \begin{equation}
   u^{N}=e^{\alpha(t)R}\left[\phi(\lambda(t)x)+\chi^{N}(\lambda(t)x,t)\right],
   \end{equation}
   where $\chi^{(N)}(y,t)=e^{\theta R}Z^{(N)}(\rho,t)$, $\rho=|y|$, and $Z^{(N)}$ satisfies that for any $0<t\leq T(N,\delta)$ with some $T(N,\delta)>0$,
      \begin{align}
      &\left\|\partial_\rho Z^{(N)}(t)\right\|_{L^2(\rho d\rho)},\, \left\|\rho^{-1} Z^{(N)}(t)\right\|_{L^2(\rho d\rho)},\, \left\|\rho\partial_\rho Z^{(N)}(t)\right\|_{L^\infty}\leq C\delta^{2\nu},\label{E2.5}\\
      &\left\|\rho^{-l}\partial_\rho^k Z^{(N)}(t)\right\|_{L^2(\rho d\rho)}\leq C\delta^{2\nu-1}t^{\frac{1}{2}+\nu},\quad k+l=2,\label{E2.6}\\
      &\left\|\rho^{-l}\partial_\rho^k Z^{(N)}(t)\right\|_{L^2(\rho d\rho)}\leq Ct^{2\nu},\quad k+l=3,\label{E2.7}\\
      &\left\|\rho\partial_\rho Z^{(N)}(t)\right\|_{L^\infty},\, \left\|\rho^{-1} Z^{(N)}(t)\right\|_{L^\infty}\leq C\delta^{2\nu-1}t^\nu,\label{E2.8}\\
      &\left\|\rho^{-l}\partial_\rho^k Z^{(N)}(t)\right\|_{L^\infty} \leq Ct^{2\nu},\, 2\leq k+l\leq 3.\label{E2.9}
      \end{align}
      The constants $C$ here and next do not depend on $N$ and $\delta$.

      In addition, there holds
      \begin{align}
      \left\|\chi^{(N)}(t)\right\|_{\dot{W}^{4,\infty}} + \left\|\langle y\rangle^{-1}\chi^{(N)}(t)\right\|_{\dot{W}^{5,\infty}} \leq Ct^{2\nu},\label{E2.10}\\
      \langle x\rangle^{2(\nu-1)}\nabla^4u^{(N)}(t), \langle x\rangle^{2(\nu-1)}\nabla^2u_t^{(N)}(t) \in L^\infty(\mathbb{R}^2),
      \end{align}
      where $\langle x\rangle=\sqrt{1+x^2}$.

      Furthermore, there exists $\zeta_N^*\in \dot{H}^1\cap\dot{H}^{1+2\nu-}$ such that $e^{\alpha(t)R}\chi^{(N)}(\lambda(t)\cdot,t)\rightarrow \zeta_N^*$ in $\dot{H}^1\cap\dot{H}^2$ as $t\rightarrow0$.
  \item[(ii)] The corresponding error
  \begin{equation}
  r^{(N)}=-u_t^{(N)} +\mathfrak{a}_1u^{(N)}\times\Delta u^{(N)} -\mathfrak{a}_2u^{(N)}\times(u^{(N)} \times\Delta u^{(N)})
  \end{equation}
  satisfies the estimate:
      \begin{equation}\label{E2.11}
      \left\|r^{(N)}(t)\right\|_{H^3} + \left\|\partial_tr^{(N)}(t)\right\|_{H^3} +\left\|\langle x\rangle r^{(N)}(t)\right\|_{L^2} \leq t^N,\quad 0<t\leq T(\delta,N).
      \end{equation}
\end{description}
\end{proposition}

Here are some comments on the proposition.
\begin{remark}
Note that \eqref{E2.5} and \eqref{E2.6} imply
\begin{equation}\label{E2.12}
\left\|u^{(N)}(t)-e^{\alpha(t)R}\phi(\lambda(t)\cdot)\right\|_{\dot{H}^1\cap \dot{H}^2} \leq \delta^{2\nu-1},\quad \forall t\in(0,T(N,\delta)].
\end{equation}
\end{remark}

\begin{remark}
According to the construction, for any $s<\nu$, $\chi^{(N)}(t)\in\dot{H}^{1+2s}$ satisfies the following estimate:
\begin{equation}
\left\|\chi^{(N)}(t)\right\|_{\dot{H}^{1+2s}(\mathbb{R}^2)} \leq C\left(t^{2\nu}+t^{s(1+2\nu)}\delta^{2\nu-2s}\right).
\end{equation}
\end{remark}

\begin{remark}
In fact, the remainder $r^{(N)}$ satisfies that for any $m,l,k$, if $N\geq C_{l,m,k}$, then
\begin{equation}
\left\|\langle x\rangle^l\partial_t^mr^{(N)}(t)\right\|_{H^k} \leq C_{l,m,k}t^{N-C_{l,m,k}}.
\end{equation}
\end{remark}

Next, we prove Proposition \ref{pro2.1}. For convenience, we consider only the case when $\nu$ is an irrational number, and it is natural to extend it to the case when $\nu$ is a rational number.

To construct an arbitrarily good approximate solution, we analyze three regions corresponding to three different spatial scales: the inner region with the scale $r\lambda(t)\lesssim 1$, the self-similar region with scale $r=O(t^{1/2})$ and the remote region with scale $r=O(1)$. The inner region is the region where the blowup concentrates, in which we construct the solutions by perturbing the profile $e^{\alpha(t)R}Q(\lambda(t)r)$. In the self-similar and remote regions, we construct solutions that are close to $\mathbf{k}$. These solutions are essentially described by their corresponding linearized equations. More precisely, in the self-similar region, the profile of the solutions is uniquely determined by the matching conditions in the inner region, while in the remote region, the profile remains essentially a free parameter of the structure and can be matched only by the limit behavior at the origin, see Subsections \ref{subsec2.3} and \ref{subsec2.4} for more details. There are some closely related similar studies for other equations, such as the critical harmonic map heat flow \cite{AH05}\cite{BHK03} and the critical Schr\"odinger map flow \cite{Perelman14} and the critical Schr\"odinger equation \cite{OP13}.

\subsection{Inner region $r\lambda(t)\lesssim1$}
First, we consider the inner region $0\leq r\lambda(t)\leq10t^{-\nu+\varepsilon_1}$, where $0<\varepsilon_1<\nu$ is to be determined. Writing $v(r,t)$ as
$$
v(r,t)=e^{\alpha(t)R}V(\lambda(t)r,t),\quad V=(V_1,V_2,V_3),
$$
and using \eqref{E2.2}, we get
\begin{equation}\label{E2.13}
\begin{split}
&t^{1+2\nu}V_t+\alpha_0t^{2\nu}RV -t^{2\nu}\left(\nu+\frac{1}{2}\right)\rho V_\rho\\
& \quad =\mathfrak{a}_1V\times\left(\Delta V+\frac{R^2}{\rho^2}V\right) -\mathfrak{a}_2V\times\left[V\times\left(\Delta V+\frac{R^2}{\rho^2}V\right)\right],\quad \rho=\lambda(t)r.
\end{split}
\end{equation}
We construct the solution of \eqref{E2.13}, which is a perturbation of the harmonic map $Q(\rho)$:
$$
V=Q+Z.
$$
We further decompose $Z$ as
$$
Z(\rho,t)=z_1(\rho,t)f_1(\rho)+z_2(\rho,t)f_2+\gamma(\rho,t)Q(\rho),
$$
where $\{f_1,f_2\}$ is an orthogonal frame on the tangent space $T_Q\mathbb{S}^2$:
$$
f_1(\rho)=\left(
            \begin{array}{c}
              h_3(\rho) \\
              0 \\
              -h_1(\rho) \\
            \end{array}
          \right),\quad
f_2=\left(
            \begin{array}{c}
              0 \\
              1 \\
              0 \\
            \end{array}
          \right).
$$
Therefore, we obtain that
$$
\gamma=\sqrt{1-|z|^2}-1=O(|z|^2),\quad z_1+iz_2,
$$
and the identities:
\begin{equation}
\begin{split}
&\partial_\rho Q=-\frac{h_1}{\rho}f_1,\quad \partial_\rho f_1=\frac{h_1}{\rho}Q,\quad f_2=Q\times f_1,\\
&\Delta f_1+\frac{R^2}{\rho^2}f_1=-\frac{1}{\rho^2}f_1-\frac{2h_3h_1}{\rho^2}Q.
\end{split}
\end{equation}
Now we write \eqref{E2.13} as an equation with respect to $z$. A direct calculation shows that
\begin{equation}\label{E2.14}
\begin{split}
RV&=-h_3z_2f_1+[h_3z_1+h_1(1+\gamma)]f_2-h_1z_2Q,\\
\rho\partial_\rho V&=[\rho\partial_\rho z_1-h_1(1+\gamma)]f_1+\rho\partial_\rho z_2f_2 +(h_1z_1+\rho\partial_\rho\gamma)Q.
\end{split}
\end{equation}
Next, we calculate the nonlinear term
$$
\mathfrak{a}_1V\times\left(\Delta V+\frac{R^2}{\rho^2}V\right) -\mathfrak{a}_2V\times\left[V\times\left(\Delta V+\frac{R^2}{\rho^2}V\right)\right].
$$
In the basis $\{f_1,f_2,Q\}$, $\Delta V+\frac{R^2}{\rho^2}V$ can be expressed as
\begin{align*}
\Delta V+\frac{R^2}{\rho^2}V= &\left(\Delta z_1-\frac{z_1}{\rho^2}-2\frac{h_1}{\rho}\gamma_\rho\right)f_1 +\left(\Delta z_2-\frac{z_2}{\rho^2}\right)f_2\\
&+\left[\Delta\gamma+\kappa(\rho)(1+\gamma) +2\frac{h_1}{\rho}\partial_\rho z_1 -2\frac{h_1h_3}{\rho^2}z_1\right]Q,
\end{align*}
Thus, we obtain that
\begin{align*}
V\times\left(\Delta V+\frac{R^2}{\rho^2}V\right) =\left[(1+\gamma)Lz_2+F_1(z)\right]f_1-\left[(1+\gamma)Lz_1+F_2(z)\right]f_2+F_3(z)Q,
\end{align*}
and
\begin{align*}
&V\times\left[V\times\left(\Delta V+\frac{R^2}{\rho^2}V\right)\right]\\ =&\left\{z_2F_3(z)+(1+\gamma)\left[(1+\gamma)Lz_1+F_2(z)\right]\right\}f_1\\
&+\left\{(1+\gamma)\left[(1+\gamma)Lz_2+F_1(z)\right]-z_1F_3(z)\right\}f_2\\
&-\left\{z_1\left[(1+\gamma)Lz_1+F_2(z)\right] +z_2\left[(1+\gamma)Lz_2+F_1(z)\right]\right\}Q.
\end{align*}
Therefore,
\begin{equation}\small\label{E2.15}
\begin{split}
&\mathfrak{a}_1V\times\left(\Delta V+\frac{R^2}{\rho^2}V\right) -\mathfrak{a}_2V\times\left[V\times\left(\Delta V+\frac{R^2}{\rho^2}V\right)\right]\\
=&\left(\mathfrak{a}_1\left[(1+\gamma)Lz_2+F_1(z)\right] -\mathfrak{a}_2\big\{z_2F_3(z)+(1+\gamma)\big[(1+\gamma)Lz_1+F_2(z)\big]\big\}\right)f_1\\
&+\left(-\mathfrak{a}_1\left[(1+\gamma)Lz_1+F_2(z)\right] -\mathfrak{a}_2\big\{(1+\gamma)\big[(1+\gamma)Lz_2+F_1(z)\big]-z_1F_3(z)\big\}\right)f_2\\
&+\left(\mathfrak{a}_1F_3(z)+\mathfrak{a}_2\big\{z_1\big[(1+\gamma)Lz_1+F_2(z)\big] +z_2\left[(1+\gamma)Lz_2+F_1(z)\right]\big\}\right)Q,
\end{split}
\end{equation}
where
\begin{equation}\label{E2.16}
\begin{split}
L&=-\Delta+\frac{1-2h_1^2}{\rho^2},\\
F_1(z)&=z_2\left(\Delta\gamma-2\frac{h_1h_3}{\rho^2}z_1+2\frac{h_1}{\rho}\partial_\rho z_1\right),\\
F_2(z)&=z_1\left(\Delta\gamma-2\frac{h_1h_3}{\rho^2}z_1+2\frac{h_1}{\rho}\partial_\rho z_1\right) +\frac{2h_1}{\rho}(1+\gamma)\gamma_\rho,\\
F_3(z)&=z_1\Delta z_2-z_2\Delta z_1+\frac{2h_1}{\rho}z_2\gamma_\rho.
\end{split}
\end{equation}
By projecting \eqref{E2.13} onto the plane $\text{span}\{f_1,f_2\}$ and using \eqref{E2.14}, \eqref{E2.15} and \eqref{E2.16}, we can rewrite \eqref{E2.13} as
\begin{align}\label{E2.17}
it^{1+2\nu}z_t-\alpha_0 t^{2\nu}h_3z &-i\left(\frac{1}{2}+\nu\right)t^{2\nu}\rho z_\rho\\
&=(\mathfrak{a}_1-i\mathfrak{a}_2)Lz +\mathfrak{a}_1F(z)+\mathfrak{a}_1dt^{2\nu}h_1-\mathfrak{a}_2\widetilde{F}(z),\notag
\end{align}
where
\begin{align*}
&d=\alpha_0-i\left(\frac{1}{2}+\nu\right),\\
&F(z)=\gamma Lz+z\left(\Delta\gamma+\frac{2h_1}{\rho}\partial_\rho z_1-\frac{2h_1h_3z_1}{\rho^2}\right) +\frac{2h_1}{\rho}(1+\gamma)\gamma_\rho +dt^{2\nu}\gamma h_1,\\
&\widetilde{F}(z)=zF_3(z) -i|z|^2 Lz+ i(1+\gamma)\left[z\left(\Delta\gamma+\frac{2h_1}{\rho}\partial_\rho z_1-\frac{2h_1h_3z_1}{\rho^2}\right) +\frac{2h_1}{\rho}(1+\gamma)\gamma_\rho\right].
\end{align*}
Note that $F$ and $\widetilde{F}$ are functions at least quadratic with respect to $z$.

Expand the solutions of \eqref{E2.17} as a power series of $t^{2\nu}$:
\begin{equation}\label{E2.18}
z(\rho,t)=\sum_{k\geq1}t^{2\nu k}z^k(\rho).
\end{equation}
Substituting \eqref{E2.18} into \eqref{E2.17}, we obtain a system with respect to $z^k$, $k\geq1$:
\begin{equation}\label{E2.19}
\begin{cases}
Lz^1=&-\frac{\mathfrak{a}_1d}{\mathfrak{a}_1-i\mathfrak{a}_2}h_1,\\
Lz^k=&\mathcal{F}_k,\quad k\geq2,
\end{cases}
\end{equation}
where $\mathcal{F}_k$ depends only on $z^j$, $j=1,2,\cdots,k-1$, and \eqref{E2.19} satisfies the following conditions at $\rho=0$:
\begin{equation}\label{E2.21}
z^k(0)=\partial_\rho z^k(0)=0.
\end{equation}
\begin{lemma}\label{lem2.2}
There exists a unique solution $(z^k)_{k\geq1}$ to the problem \eqref{E2.19}, \eqref{E2.21}, where $z^k\in C^\infty(\mathbb{R}_+)$,  $\forall k\geq1$. Furthermore,
\begin{description}
  \item[(i)] $z^k$ has an odd Taylor expansion at $\rho=0$, and the leading term is of order $2k+1$;
  \item[(ii)] As $\rho\rightarrow\infty$, $z^k$ has the following expansion:
\begin{equation}\label{E2.22}
z^k(\rho)=\sum_{l=0}^{2k}\sum_{j\leq k-\frac{l-1}{2}}c_{j,l}^k\rho^{2j-1}(\ln\rho)^l,
\end{equation}
where $c_{j,l}^k=c_{j,l}^k(d,\mathfrak{a}_1,\mathfrak{a}_2)$ is constant, and the asymptotic expansion \eqref{E2.22} can be differentiated with respect to $\rho$ any number of times.
\end{description}
\end{lemma}
\begin{proof}
Note that the equation $Lf=0$ has the following two explicit exact solutions:
\begin{equation}
h_1(\rho),\quad h_2(\rho)=\frac{\rho^4+4\rho^2\ln\rho-1}{\rho(\rho^2+1)}.
\end{equation}

For the case of $k=1$:
\begin{equation*}
\begin{cases}
&Lz^1=-\frac{\mathfrak{a}_1d}{\mathfrak{a}_1-i\mathfrak{a}_2}h_1,\\
&z^1(0)=\partial_\rho z^1(0)=0.
\end{cases}
\end{equation*}
We get
\begin{align}\label{E2.23}
z^1(\rho)=&-\frac{\mathfrak{a}_1d}{4(\mathfrak{a}_1-i\mathfrak{a}_2)}\int_0^\rho \big[h_1(\rho)h_2(s)-h_1(s)h_2(\rho)\big]h_1(s)sds\notag\\
=&-\frac{\mathfrak{a}_1d\rho}{(\mathfrak{a}_1-i\mathfrak{a}_2)\left(1+\rho^2\right)}\int_0^\rho \frac{s\left(s^4+4s^2\ln s-1\right)}{\left(1+s^2\right)^2}ds\\ &+\frac{\mathfrak{a}_1d\left(\rho^4+4\rho^2\ln\rho-1\right)}{(\mathfrak{a}_1-i\mathfrak{a}_2)\rho(\rho^2+1)} \int_0^\rho\frac{s^3}{\left(1+s^2\right)^2}ds.\notag
\end{align}
Note that $h_1$ is a $C^\infty$ function, it has an odd Taylor expansion at $\rho=0$ and the leading term of the expansion is linear, so we can expand $z^1$ as an odd Taylor expansion with a cubic leading term. This proves $\mathbf{(i)}$ for $k=1$.

The asymptotic behavior of $z^1$ at infinity can be obtained directly from \eqref{E2.23}:
$$
z^1(\rho)=c_{1,0}^1\rho+c_{1,1}^1\rho\ln\rho+\sum_{j\leq0}\sum_{l=0,1,2}c_{j,l}^1\rho^{2j-1}(\ln\rho)^l,
$$
where $c_{1,0}^1=-c_{1,1}^1=-\frac{\mathfrak{a}_1d}{\mathfrak{a}_1-i\mathfrak{a}_2}$. Thus, $\mathbf{(ii)}$ holds for $k=1$.

For the case of $k>1$, we prove it by induction. Suppose that $z^j$, $j\leq k-1$, satisfy $\mathbf{(i)}$ and $\mathbf{(ii)}$. According to \eqref{E2.17}, we get that $\mathcal{F}_k$ is an odd $C^\infty$ function, moreover, its asymptotic expansion at $\rho=0$ is zero at order $2k-1$, and the asymptotic expansion as $\rho\rightarrow\infty$ is
\begin{align*}
\mathcal{F}_k=&\sum_{j=1}^{k-1}\sum_{l=0}^{2k-2j-1}\alpha_{j,l}^k\rho^{2j-1}(\ln\rho)^l +\sum_{l=0}^{2k-2}\alpha_{0,l}^k\rho^{-1}(\ln\rho)^l\\
&+\sum_{l=0}^{2k-1}\alpha_{-1,l}^k\rho^{-3}(\ln\rho)^l +\sum_{j\leq-2}\sum_{l=0}^{2k}\alpha_{j,l}^k\rho^{2j-1}(\ln\rho)^l.
\end{align*}
Thus, we obtain that
$$
z^k(\rho)=\frac{1}{4}\int_0^\rho\left[h_1(\rho)h_2(s)-h_1(s)h_2(\rho)\right]\mathcal{F}_k(s)sds
$$
is a $C^\infty$ function, meanwhile, it can be expanded to an odd Taylor series at $\rho=0$ with leading term of order $2k+1$, and has the following asymptotic expansion as $\rho\rightarrow\infty$:
$$
z^k(\rho)=\sum_{l=0}^{2k}\sum_{j\leq k-\frac{l-1}{2}}c_{j,l}^k(\ln\rho)^l\rho^{2j-1}.
$$
This proves Lemma \ref{lem2.2}.
\end{proof}

By \eqref{E2.18}, we obtain a formal solution of \eqref{E2.2}:
\begin{equation}\label{E2.24}
v(r,t)=e^{\alpha(t)R}V\left(\lambda(t)r,t\right),\quad V(\rho,t)=Q+\sum_{k\geq1}t^{2\nu k}Z^k(\rho),
\end{equation}
where $Z^k=(Z_1^k,Z_2^k,Z_3^k)$. Here $Z_i^k$, $i=1,2$, are smooth odd functions with respect to $\rho$, and their asymptotic expansion at $\rho=0$ is zero at order $2k+1$, meanwhile, $Z_3^k$ is an even function, and its asymptotic expansion at $\rho=0$ is zero at order $2k+2$. As $\rho\rightarrow\infty$, we have
\begin{equation}\label{E2.25}
\begin{split}
Z_i^k(\rho)=&\sum_{l=0}^{2k}\sum_{j\leq k-\frac{l-1}{2}} c_{j,l}^{k,i}(\ln\rho)^l\rho^{2j-1},\quad i=1,2,\\
Z_3^k(\rho)=&\sum_{l=0}^{2k}\sum_{j\leq k+1-\frac{l}{2}} c_{j,l}^{k,3}(\ln\rho)^l\rho^{2j-2},
\end{split}
\end{equation}
where the coefficients $c_{j,l}^{k,i}=c_{j,l}^{k,i}(d,\mathfrak{a}_1,\mathfrak{a}_2)$ satisfy $c_{k+1,0}^{k,3}=0$, $\forall k\geq1$. The asymptotic expansion \eqref{E2.25} can be differentiated with respect to $\rho$ any number of times.

According to \eqref{E2.24} and \eqref{E2.25}, as $\rho\rightarrow\infty$, each component of $V(\rho,t)$ (i.e., $V_i(\rho,t)$, $i=1,2,3$) has the following asymptotic expansion:
\begin{align*}
V_i(\rho,t)=&\sum_{j\leq0}c_{j,0}^{0,1}\rho^{2j-1} +\sum_{k\geq1}t^{2\nu k}\sum_{l=0}^{2k}\sum_{j\leq k-\frac{l-1}{2}} c_{j,l}^{k,i}(\ln\rho)^l\rho^{2j-1},\quad i=1,2\\
V_3(\rho,t)=&1+\sum_{j\leq0}c_{j,0}^{0,3}\rho^{2j-2}+ \sum_{k\geq1}t^{2\nu k}\sum_{l=0}^{2k}\sum_{j\leq k+1-\frac{l}{2}} c_{j,l}^{k,3}(\ln\rho)^l\rho^{2j-2}.
\end{align*}
Taking $\rho\rightarrow\infty$ and $y\equiv rt^{-\frac{1}{2}}\rightarrow 0$, the above expansions can be formally rewritten as the new expansions with respect to $y$:
\begin{equation}\label{E2.26}
\begin{split}
V_i(\lambda(t)r,t)=&\sum_{j\geq0}t^{\nu(2j+1)}\sum_{l=0}^{2j+1}\left(\ln y-\nu\ln t\right)^l V_i^{j,l}(y),\quad i=1,2,\\
V_3(\lambda(t)r,t)=&1+\sum_{j\geq1}t^{2\nu j}\sum_{l=0}^{2j}\left(\ln y-\nu\ln t\right)^lV_3^{j,l}(y),\\
V_{i}^{j,l}(y)=&\sum_{k\geq-j+\frac{l}{2}}c_{k,l}^{k+j,i}y^{2k-1},\quad i=1,2,\\
V_3^{j,l}(y)=&\sum_{k\geq-j+\frac{l}{2}}c_{k+1,l}^{k+j,3}y^{2k},
\end{split}
\end{equation}
where $c_{j,l}^{k,i}$ defined by \eqref{E2.25} for $k\neq0$, and $c_{j,0}^{0,i}$ is derived from the expansion of $Q$ as $\rho\rightarrow\infty$:
$$
h_1(\rho)=\sum_{j\leq0}c_{j,0}^{0,1}\rho^{2j-1},\quad c_{j,0}^{0,2}=0,\quad h_3(\rho)=1+\sum_{j\leq0}c_{j,0}^{0,3}\rho^{2j-2}.
$$
The expression \eqref{E2.26} expanded with respect to $y$ is crucial in the matching between the self-similar region and the inner region below.

For $N\geq2$, we define
$$
z_{\text{in}}^{(N)}=\sum_{k=1}^Nt^{2\nu k}z^k,\quad z_{\text{in}}^{(N)}=z_{\text{in},1}^{(N)}+iz_{\text{in},2}^{(N)}.
$$
Substituting $z_{\text{in}}^{(N)}$ into \eqref{E2.17}, we get the error
\begin{align*}
X_N=-it^{1+2\nu}\partial_t z_{\text{in}}^{(N)}&+\alpha_0 t^{2\nu}h_3z_{\text{in}}^{(N)} +i\left(\frac{1}{2}+\nu\right)t^{2\nu}\rho \partial_\rho z_{\text{in}}^{(N)}\\ &+(\mathfrak{a}_1-i\mathfrak{a}_2)Lz_{\text{in}}^{(N)} +\mathfrak{a}_1F\left(z_{\text{in}}^{(N)}\right) +\mathfrak{a}_1dt^{2\nu}h_1 -\mathfrak{a}_2\widetilde{F}\left(z_{\text{in}}^{(N)}\right).
\end{align*}
According to the definition of $z^k$, $\rho<\langle\rho\rangle$ and $\ln\rho<\ln(2+\rho)$, it is easy to verify that the error $X_N$ satisfies the following estimate:
There exists $T(N)>0$, for any $k,m\in\mathbb{N}$, $0\leq l\leq(2N+1-k)_+$, $0\leq\rho\leq 10 t^{-\nu+\varepsilon_1}$ and $0<t\leq T(N)$, we have
\begin{equation}\label{E2.27}
\left|\rho^{-l}\partial_\rho^k\partial_t^m X_N\right|\leq C_{k,l,m}t^{2\nu N-m}\langle \rho\rangle^{2(N+1)-1-l-k}\ln(2+\rho),
\end{equation}
where $N_+=\max\{N,0\}$.

Let
\begin{align*}
\gamma_{\text{in}}^{(N)}=&\sqrt{1-\left|z_{\text{in}}^{(N)}\right|^2}-1,\\
Z_{\text{in}}^{(N)}=&z_{\text{in},1}^{(N)}f_1 +z_{\text{in},2}^{(N)}f_2 +\gamma_{\text{in}}^{(N)}Q,\\
V_{\text{in}}^{(N)}=&Q+Z_{\text{in}}^{(N)}\in \mathbb{S}^2.
\end{align*}
Then $V_{\text{in}}^{(N)}$ satisfies
\begin{equation}\label{E2.28}
\begin{split}
&t^{1+2\nu}\partial_t V_{\text{in}}^{(N)} +\alpha_0t^{2\nu}t^{2\nu}RV_{\text{in}}^{(N)} -t^{2\nu}\left(\nu+\frac{1}{2}\right)\rho \partial_\rho V_{\text{in}}^{(N)}\\
=&\mathfrak{a}_1V_{\text{in}}^{(N)}\times\left(\Delta V_{\text{in}}^{(N)}+\frac{R^2}{\rho^2}V_{\text{in}}^{(N)}\right)\\ &-\mathfrak{a}_2V_{\text{in}}^{(N)}\times\left[V_{\text{in}}^{(N)}\times\left(\Delta V_{\text{in}}^{(N)}+\frac{R^2}{\rho^2}V_{\text{in}}^{(N)}\right)\right] +\mathcal{R}_{\text{in}}^{(N)},
\end{split}
\end{equation}
where
$$
\mathcal{R}_{\text{in}}^{(N)} =\operatorname{Im}\left(X_{N}\right)f_1 -\operatorname{Re}\left(X_N\right) f_2 +\frac{\operatorname{Im}\left(\bar{X}_N z^{(N)}\right)}{1+\gamma^{(N)}}Q
$$
has the same estimate as the error $X_N$. According to the analysis, for any $0\leq\rho\leq10 t^{-\nu+\varepsilon_1}$ and $0<t\leq T(N)$, we have
\begin{equation}\label{E2.29}
\left|\rho^{-l}\partial_\rho^k Z_{\text{in}}^{(N)}\right|\leq C_{k,l}t^{2\nu} \langle\rho\rangle^{1-l-k}\ln(2+\rho),\quad k\in\mathbb{N},\quad l\leq(3-k)_+.
\end{equation}
Thus, we obtain the following estimates.
\begin{lemma}\label{lem2.3}
There exists $T(N)>0$ such that for any $0<t\leq T(N)$, the following holds.
\begin{description}
  \item[(i)] $Z_{\text{in}}^{(N)}(\rho,t)$ satisfies
\begin{align}
&\left\|\partial_\rho Z_{\text{in}}^{(N)}(t)\right\|_{L^2\left(\rho d\rho, 0\leq\rho\leq10t^{-\nu+\varepsilon_1}\right)} \leq Ct^\nu,\label{E2.30}\\
&\left\|\rho^{-1}\partial_\rho Z_{\text{in}}^{(N)}(t)\right\|_{L^2\left(\rho d\rho, 0\leq\rho\leq10t^{-\nu+\varepsilon_1}\right)} \leq Ct^\nu,\label{E2.31}\\
&\left\|Z_{\text{in}}^{(N)}(t)\right\|_{L^\infty\left(0\leq\rho\leq10t^{-\nu+\varepsilon_1}\right)} +\left\|\rho\partial_\rho Z_{\text{in}}^{(N)}(t)\right\|_{L^\infty\left( 0\leq\rho\leq10t^{-\nu+\varepsilon_1}\right)} \leq Ct^\nu,\label{E2.32}\\
&\left\|\rho^{-l}\partial_\rho^k Z_{\text{in}}^{(N)}(t)\right\|_{L^2\left(\rho d\rho, 0\leq\rho\leq10t^{-\nu+\varepsilon_1}\right)} \leq Ct^{2\nu}\left(1+|\ln t|\right), \quad k+l=2,\label{E2.33}\\
&\left\|\rho^{-l}\partial_\rho^k Z_{\text{in}}^{(N)}(t)\right\|_{L^2\left(\rho d\rho, 0\leq\rho\leq10t^{-\nu+\varepsilon_1}\right)} \leq Ct^{2\nu}, \quad k+l\geq3,\, l\leq(3-k)_+,\label{E2.34}\\
&\left\|\partial_\rho Z_{\text{in}}^{(N)}(t)\right\|_{L^\infty\left(0\leq\rho\leq10t^{-\nu+\varepsilon_1}\right)} \label{E2.35}\\
&\quad\quad\quad\quad\quad +\left\|\rho^{-1}Z_{\text{in}}^{(N)}(t)\right\|_{L^\infty\left( 0\leq\rho\leq10t^{-\nu+\varepsilon_1}\right)} \leq Ct^{2\nu}\left(1+|\ln t|\right),\notag\\
&\left\|\rho^{-l}\partial_\rho^k Z_{\text{in}}^{(N)}(t)\right\|_{L^\infty\left(0\leq\rho\leq10t^{-\nu+\varepsilon_1}\right)} \leq Ct^{2\nu},\quad 2\leq l+k,\, l\leq(3-k)_+. \label{E2.36}
\end{align}
  \item[(ii)] The error $\mathcal{R}_{\text{in}}^{(N)}$ has the estimate: If $N>\varepsilon_1^{-1}$, then
\begin{equation}\label{E2.37}
\begin{split}
\left\|\rho^{-l}\partial_\rho^k\mathcal{R}_{\text{in}}^{(N)}(t)\right\|_{L^2\left(\rho d\rho, 0\leq\rho\leq10t^{-\nu+\varepsilon_1}\right)}&\leq t^{N\varepsilon_1},\quad 0\leq l+k\leq3,\\
\left\|\rho^{-l}\partial_\rho^k\partial_t\mathcal{R}_{\text{in}}^{(N)}(t)\right\|_{L^2\left(\rho d\rho, 0\leq\rho\leq10t^{-\nu+\varepsilon_1}\right)}&\leq t^{N\varepsilon_1},\quad 0\leq l+k\leq1.
\end{split}
\end{equation}
\end{description}
\end{lemma}

\subsection{Self-similar region $rt^{-\frac{1}{2}}\lesssim1$}\label{subsec2.3}
Next, we consider the self-similar region $10t^{\varepsilon_1}\leq rt^{-\frac{1}{2}}\leq10t^{-\varepsilon_2}$, where $0<\varepsilon_2<\frac{1}{2}$ is to be determined. In this region, we want the solution to be close to $\mathbf{k}$. Using the stereographic projection:
$$
(v_1,v_2,v_3)=v\rightarrow w=\frac{v_1+iv_2}{1+v_3}\in\mathbb{C}\cup\{\infty\},
$$
equation \eqref{E2.2} is equivalently transformed into
\begin{equation}\label{E2.38}
iw_t=(\mathfrak{a}_1-i\mathfrak{a}_2)\left[-\Delta w +\frac{1}{r^2}w +G(w,\bar{w},w_r)\right],
\end{equation}
where
$$
G(w,\bar{w},w_r)=\frac{2\bar{w}}{1+|w|^2}\left(w_r^2-\frac{1}{r^2}w^2\right).
$$
%更一般地, 若$w(r,t)$满足
%\begin{equation}\label{E2.39}
%iw_t=(\mathfrak{a}_1-i\mathfrak{a}_2)\left[-\Delta u +\frac{1}{r^2}w +G(w,\bar{w},w_r)\right]+A,
%\end{equation}
%则
%$$
%v=\left(\frac{2\operatorname{Re}(w)}{1+|w|^2}, \frac{2\operatorname{Im}(w)}{1+|w|^2}, \frac{1-|w|^2}{1+|w|^2}\right)\in\mathbb{S}^2
%$$
%满足
%\begin{equation}\label{E2.40}
%v_t=\mathfrak{a}_1v\times\left(\Delta v+\frac{R^2}{r^2}v\right) -\mathfrak{a}_2v\times\left[v\times\left(\Delta v+\frac{R^2}{r^2}v\right)\right]+\mathcal{A},
%\end{equation}
%其中$\mathcal{A}=(\mathcal{A}_1,\mathcal{A}_2,\mathcal{A}_3)$为
%$$
%\mathcal{A}_1+i\mathcal{A}_2 =-2i\frac{A+w^2\bar{A}}{\left(1+|w|^2\right)^2},\quad \mathcal{A}_3=\frac{4\operatorname{Im}\left(w\bar{A}\right)}{\left(1+|w|^2\right)^2}.
%$$
Let
\begin{equation}\label{wW}
w(r,t)=e^{i\alpha(t)}W(y,t),\quad y=rt^{-\frac{1}{2}}.
\end{equation}
Then \eqref{E2.38} becomes
\begin{equation}\label{E2.41}
itW_t-\alpha_0W=(\mathfrak{a}_1-i\mathfrak{a}_2)\left[\mathcal{L}W +G\left(W,\bar{W},W_y\right)\right],
\end{equation}
where
$$
\mathcal{L}=-\Delta+\frac{1}{y^2}+\frac{i}{2(\mathfrak{a}_1-i\mathfrak{a}_2)}y\partial_y.
$$
Thus, as $y\rightarrow0$, it follows from \eqref{E2.26} that $W$ has an expansion of the form:
\begin{align}\label{E2.42}
W(y,t)=\sum_{j\geq0}\sum_{l=0}^{2j+1}\sum_{\tilde{i}\geq-j+\frac{l}{2}}& \alpha(j,\tilde{i},l) t^{\nu(2j+1)} \left(\ln y-\nu\ln t\right)^l y^{2\tilde{i}-1},
\end{align}
where the coefficients $\alpha(j,\tilde{i},l)$ can be precisely expressed as terms of $c_{j',l'}^{k,i'}$, $1\leq k\leq j+\tilde{i}$, $j'\leq \tilde{i}$, $0\leq l'\leq l$, here $c_{j,l}^{k,i}$ are defined by \eqref{E2.25}. This inspires us to assume that $W$ has the following form:
\begin{equation}\label{E2.43}
W(y,t)=\sum_{j\geq0}\sum_{l=0}^{2j+1}t^{\nu(2j+1)}\left(\ln y-\nu\ln t\right)^l W_{j,l}(y).
\end{equation}
Substituting \eqref{E2.43} into \eqref{E2.41}, we get a system with respect to $W_{j,l}$:
\begin{equation}\label{E2.44}
\begin{cases}
(\mathfrak{a}_1-i\mathfrak{a}_2)\mathcal{L}W_{0,1}=\mu_0W_{0,1},\\
(\mathfrak{a}_1-i\mathfrak{a}_2)\mathcal{L}W_{0,0}=\mu_0W_{0,0} -i\left(\frac{1}{2}+\nu\right)W_{0,1} +2(\mathfrak{a}_1-i\mathfrak{a}_2)\frac{1}{y}\partial_y W_{0,1},
\end{cases}
\end{equation}
\begin{equation}\label{E2.45}
\begin{cases}
(\mathfrak{a}_1-i\mathfrak{a}_2)\mathcal{L}W_{j,2j+1} =\mu_jW_{j,2j+1} +(\mathfrak{a}_1-i\mathfrak{a}_2)\mathcal{G}_{j,2j+1},\\
(\mathfrak{a}_1-i\mathfrak{a}_2)\mathcal{L}W_{j,2j} =\mu_jW_{j,2j} +(\mathfrak{a}_1-i\mathfrak{a}_2)\mathcal{G}_{j,2j} -\frac{1}{2}i(2j+1)W_{j,2j+1}\\
\quad\quad\quad\quad\quad\quad  -i\nu(2j+1)W_{j,2j+1} +2(2j+1)(\mathfrak{a}_1-i\mathfrak{a}_2) \frac{1}{y}\partial_y W_{j,2j+1},\\
(\mathfrak{a}_1-i\mathfrak{a}_2)\mathcal{L} W_{j,l} =\mu_jW_{j,l} +(\mathfrak{a}_1-i\mathfrak{a}_2)\mathcal{G}_{j,l} -\frac{1}{2}i(l+1)W_{j,l+1}\\
\quad\quad\quad\quad\quad -i\nu(l+1)W_{j,l+1} +2(l+1)(\mathfrak{a}_1-i\mathfrak{a}_2)\frac{1}{y}\partial_yW_{j,l+1}\\
\quad\quad\quad\quad\quad + (l+1)(l+2)(\mathfrak{a}_1-i\mathfrak{a}_2)\frac{1}{y^2}W_{j,l+2}, \quad 0\leq l\leq 2j-1,
\end{cases}
\end{equation}
where $0\leq l\leq 2j+1$, $j\geq0$, $\mu_j=-\alpha_0+i\nu(2j+1)$, and $\mathcal{G}_{j,l}$ come from the nonlinear term $G(W,\bar{W},W_y)$, which depends only on $W_{k,n}$, $k\leq j-1$:
\begin{equation*}
\begin{split}
& G(W,\bar{W},W_y)=-\sum_{j\geq1}\sum_{l=0}^{2j+1}t^{(2j+1)\nu}(\ln y-\nu\ln t)^l\mathcal{G}_{j,l}(y),\\
& \mathcal{G}_{j,l}(y)=\mathcal{G}_{j,l}\left(y;W_{k,n},\, 0\leq n\leq 2k+1,\, 0\leq k\leq j-1\right).
\end{split}
\end{equation*}

Thus, we obtain
\begin{lemma}\label{lem2.4}
Given coefficients $a_j$ and $b_j$, $j\geq0$, there exists a unique solution $W_{j,l}\in C^\infty(\mathbb{R}_+^*)$, $0\leq l\leq 2j+1$, $j\geq0$, to the system \eqref{E2.44}, \eqref{E2.45} such that $W_{j,l}$ has the following asymptotic expansion as $y\rightarrow0$:
\begin{equation}\label{E2.46}
W_{j,l}(y)=\sum_{i\geq-j+\frac{l}{2}}d_{i}^{j,l}y^{2i-1},
\end{equation}
where
\begin{equation}\label{E2.47}
d_1^{j,1}=a_j,\quad d_1^{j,0}=b_j.
\end{equation}
Furthermore, the asymptotic expansion \eqref{E2.46} can be differentiated with respect to $y$ any number of times.
\end{lemma}
\begin{proof}
Let
$$
\tilde{\mu}_j=\frac{\mu_j}{(\mathfrak{a}_1-i\mathfrak{a}_2)}.
$$
Note that the solutions of $(\mathcal{L}-\tilde{\mu}_j)f=0$ has a basis $\{e_j^1,e_j^2\}$ satisfying
\begin{description}
  \item[(i)] $e_j^1$ is an odd $C^\infty$ function, and $e_j^1(y)=y+O(y^3)$ as $y\rightarrow0$;
  \item[(ii)] $e_j^2\in C^\infty(\mathbb{R}_+^*)$ and it can be expressed as follows:
$$
e_j^2(y)=\frac{1}{y}+\kappa_je_j^1(y)\ln y+\tilde{e}_j^2(y),\quad \kappa_j=-\frac{i}{4(\mathfrak{a}_1-i\mathfrak{a}_2)}-\frac{1}{2}\tilde{\mu}_j,
$$
where $\tilde{e}_j^2$ is an odd $C^\infty$ function, and $\tilde{e}_j^2(y)=O\left(y^3\right)$ as $y\rightarrow0$.
\end{description}

We consider the system \eqref{E2.44}, according to $(\mathcal{L}-\tilde{\mu}_0)W_{0,1}=0$ and \eqref{E2.46} and \eqref{E2.47}, we get
$$
W_{0,1}=a_0e_0^1.
$$
Consider equation of $W_{0,0}$:
\begin{equation}\label{L-mu0}
\left(\mathcal{L}-\tilde{\mu}_0\right)W_{0,0} =-\frac{i}{(\mathfrak{a}_1-i\mathfrak{a}_2)}\left(\frac{1}{2}+\nu\right)W_{0,1} +2\frac{1}{y}\partial_yW_{0,1}.
\end{equation}
Notice that the right hand side of \eqref{L-mu0} has the following form: $2a_0\frac{1}{y}+$ an odd $C^\infty$ function, where the odd function is $O(y)$ as $y\rightarrow0$. Thus, \eqref{L-mu0} has a unique solution $W_{0,0}^0$, which has the following form:
$$
W_{0,0}^0(y)=d_0\frac{1}{y}+\tilde{W}_{0,0}^0(y),
$$
where $d_0=\frac{a_0}{k_0}$, $\tilde{W}_{0,0}^0$ is an odd $C^\infty$ function, and $\tilde{W}_{0,0}^0(y)=O(y^3)$ as $y\rightarrow0$. Combining \eqref{E2.46} and \eqref{E2.47}, we obtain
$$
W_{0,0}=W^0_{0,0}+b_0e^1_0.
$$

For the case of $j\geq1$, we have
\begin{equation}\label{E2.48}
\left(\mathcal{L}-\tilde{\mu}_j\right)W_{j,l}=\mathcal{F}_{j,l},\quad 0\leq l\leq2j+1,
\end{equation}
where
\begin{equation}\label{E2.49}
\begin{split}
\mathcal{F}_{j,2j+1}=&\mathcal{G}_{j,2j+1},\\
\mathcal{F}_{j,2j}=&\mathcal{G}_{j,2j} -\frac{i}{2(\mathfrak{a}_1-i\mathfrak{a}_2)}(2j+1)W_{j,2j+1}\\
&-\frac{i}{\mathfrak{a}_1-i\mathfrak{a}_2}\nu(2j+1)W_{j,2j+1} +2(2j+1) \frac{1}{y}\partial_y W_{j,2j+1},\\
\mathcal{F}_{j,l}=&\mathcal{G}_{j,l} +\frac{i}{2(\mathfrak{a}_1-i\mathfrak{a}_2)}(l+1)W_{j,l+1} -\frac{i}{(\mathfrak{a}_1-i\mathfrak{a}_2)}\nu(l+1)W_{j,l+1}\\
& +2(l+1)\frac{1}{y}\partial_yW_{j,l+1}+ (l+1)(l+2)\frac{1}{y^2}W_{j,l+2}, \quad 0\leq l\leq 2j-1.
\end{split}
\end{equation}
The resolution of \eqref{E2.48} is based on the following ODE lemma.
\begin{lemma}[Lemma \ref{lem2.5} in \cite{Perelman14}]\label{lem2.5}
Let $F$ be a $C^\infty(\mathbb{R}_+^*)$ function of the form:
$$
F(y)=\sum_{j=k}^0F_jy^{2j-1}+\tilde{F}(y),
$$
where $\tilde{F}$ is an odd $C^\infty$ function, $k\leq-1$. Then there exists a unique constant $A$ such that the equation
$$
(\mathcal{L}-\tilde{\mu}_j)u=F+A\frac{1}{y^3}
$$
exists a solution $u\in C^\infty(\mathbb{R}^*_+)$, which has the following asymptotic behavior as $y\rightarrow0$:
$$
u(y)=\sum_{j\geq k+1}u_jy^{2j-1},\quad u_1=0.
$$
\end{lemma}

More precisely, suppose that $W_{i,n}$ has the asymptotic behavior described in \eqref{E2.46} and \eqref{E2.47}, where $0\leq n\leq2i+1$, $i\leq j-1$, then it is not difficult to verify that $\mathcal{G}_{j,l}$ has the following expansion as $y\rightarrow0$:
\begin{equation}\label{E2.50}
\begin{split}
&\mathcal{G}_{j,2j+1}(y)=\sum_{i\geq1}g^i_{j,2j+1}y^{2i-1},\\
&\mathcal{G}_{j,2j}(y)=\sum_{i\geq0}g^i_{j,2j}y^{2i-1},\\
&\mathcal{G}_{j,l}=\sum_{i\geq-j+\frac{l}{2}-1}g^i_{j,l}y^{2i-1},\quad l\leq2j-1.\\
\end{split}
\end{equation}
Consider the equation of $W_{j,2j+1}$: $(\mathcal{L}-\tilde{\mu}_j)W_{j,2j+1} =\mathcal{G}_{j,2j+1}$, we get
\begin{equation}\label{E2.51}
W_{j,2j+1}=W^0_{j,2j+1}+c_0e^1_j,
\end{equation}
where $W^0_{j,2j+1}$ is the unique odd $C^\infty$ solution of $(\mathcal{L}-\tilde{\mu}_j)f =\mathcal{G}_{j,2j+1}$, and $W^0_{j,2j+1}(y)=O(y^3)$ as $y\rightarrow0$, the constant $c_0$ is to be determined.

Note that $\mathcal{F}_{j,2j}$ has the following form:
$$
\left(g_{j,2j}^0+2(2j+1)c_0\right)\frac{1}{y}+ \text{ an odd } C^\infty \text{ function}.
$$
Thus, we get
\begin{equation}\label{E2.52}
W_{j,2j}=W^0_{j,2j}+c_1e^1_j,
\end{equation}
where $W^0_{j,2j}$ is the unique solution of $(\mathcal{L}-\tilde{\mu}_j)f =\mathcal{G}_{j,2j}$ satisfying
\begin{equation}\label{E2.53}
W^0_{j,2j}=d_1\frac{1}{y}+O\left(y^3\right),\quad d_1=\frac{g_{j,2j}^0+2(2j+1)c_0}{2k_j}.
\end{equation}
as $y\rightarrow0$, $c_1$ is a constant similar to $c_0$.

For $\mathcal{F}_{j,2j-1}$, by \eqref{E2.49}, \eqref{E2.50}, \eqref{E2.51}, \eqref{E2.52} and \eqref{E2.53}, we get
$$
\mathcal{F}_{j,2j-1}= \left(g_{j,2j-1}^{-1}-4jd_1\right)\frac{1}{y^3}+const\cdot\frac{1}{y}+ \text{an odd } C^\infty \text{ function},
$$
The constant here can be calculated exactly:
$$
const =g_{j,2j}^{0} -\frac{i}{\mathfrak{a}_1-i\mathfrak{a}_2} j(1-2\nu)d1 +2(2j+1)\tilde{c}_0,\quad \tilde{c}_0=O(1).
$$
However, the refinement of this constant has no effect on the proof of our main result, and we will continue to use the notation $const$, which may have different values, but in principle it can also be calculated exactly.

Notice that $\frac{1}{y^3}$ does not appear in $(\mathcal{L}-\tilde{\mu}_j)W_{j,2j-1}$, where $W_{j,2j-1}$ is defined as \eqref{E2.46}. Thus, the equation $(\mathcal{L}-\tilde{\mu}_j)W_{j,2j-1}=\mathcal{F}_{j,2j-1}$ has a solution of the form \eqref{E2.46} if and only if
$$
g_{j,2j-1}^{-1}-4jd_1=0.
$$
By \eqref{E2.53}, we get
$$
c_0=\frac{k_jg_{j,2j-1}^{-1}-2jg^0_{j,2j}}{4j(2j+1)}.
$$
By the choice of $c_0$, we get
$$
W_{j,2j-1}=W^0_{j,2j-1}+c_2e^1_j,
$$
where $W^0_{j,2j-1}$ is the unique solution of $(\mathcal{L}-\tilde{\mu}_j)f =\mathcal{F}_{j,2j-1}$ that satisfies
$$
W^0_{j,2j-1}=const\cdot\frac{1}{y}+O\left(y^3\right)
$$
as $y\rightarrow\infty$. Continuing the above process, we obtain $W_{j,2j-2},\cdots,W_{j,0}$, which has the form $W_{j,2j+1-k}=W^0_{j,2j+1-k}+c_ke_j^1$, $k\leq2j+1$, where $W^0_{j,2j+1-k}$ is the unique solution of $(\mathcal{L}-\tilde{\mu}_j)f=\mathcal{F}_{j,2j+1-k}$ that has expansion of the form \eqref{E2.46} with zero coefficient $d_1^{j,l}$ as $y\rightarrow0$, and the constants $c_k$, $k\leq2j-1$, are uniquely determined by the solvability conditions on the equation of $W_{j,2j-k-1}$ (see Lemma \ref{lem2.5}). Finally, $c_{2j+1}$ and $c_{2j+2}$ are given by \eqref{E2.47}, i.e., $c_{2j+1}=a_j$ and $c_{2j+2}=b_j$.
\end{proof}

Let $W_{j,l}^{ss}(y)$ be the solution of the system \eqref{E2.44}--\eqref{E2.45} (see Lemma \ref{lem2.4}), where $0\leq l\leq2j+1$, $j\geq0$, $a_j=\alpha(j,1,1)$, $b_j=\alpha(j,1,0)$. Here $\alpha(i,j,k)$ is defined by \eqref{E2.42}. Since the expansion \eqref{E2.42} is a solution of \eqref{E2.41}, the uniqueness of Lemma \ref{lem2.4} guarantees that
\begin{equation}\label{E2.54}
W^{ss}_{j,l}(y)=\sum_{i\geq-j+\frac{l}{2}}\alpha(j,i,l)y^{2i-1},\quad\text{as}\,\, y\rightarrow0.
\end{equation}

Next, we study the asymptotic behavior of $W_{j,l}^{ss}$ at infinity, where $0\leq l\leq2j+1$ and $j\geq0$. We get
\begin{lemma}\label{lem2.6}
Given coefficients $a_{j,l}$ and $b_{j,l}$, where $0\leq l\leq2j+1$, $j\geq0$, the system \eqref{E2.44}--\eqref{E2.45} has a unique solution of the form:
\begin{align}
W_{0,l}=&W_{0,l}^0+W_{0,l}^1,\quad l=0,1,\label{E2.55}\\
W_{j,l}=&W_{j,l}^0+W_{j,l}^1+W_{j,l}^2,\quad 0\leq l\leq2j+1,\,\, j\geq1,\label{E2.56}
\end{align}
where $(W_{j,l}^{\tilde{i}})_{0\leq l\leq2j+1\atop j\geq1}$, $\tilde{i}=0,1$, are two solutions of the system \eqref{E2.44} and \eqref{E2.45} that has the following asymptotic behavior as $y\rightarrow\infty$:
\begin{equation}\label{E2.57}
\begin{split}
&\sum_{l=0}^{2j+1}(\ln y-\nu\ln t)^lW_{j,l}^{\tilde{i}}(y)= \sum_{l=0}^{2j+1}\left(\ln y+(-1)^{\tilde{i}}\frac{\ln t}{2}\right)^l \hat{W}^{\tilde{i}}_{j,l}(y),\quad \tilde{i}=0,1,\\
&\hat{W}^0_{j,l}(y)=y^{2i\alpha_0+2\nu(2j+1)} \sum_{k\geq0}\hat{w}_k^{j,l,0}y^{-2k},\\
&\hat{W}^1_{j,l}(y)=e^{\frac{iy^2}{4(\mathfrak{a}_1-i\mathfrak{a}_2)}} y^{-2i\alpha_0-2\nu(2j+1)-2} \sum_{k\geq0}\hat{w}_k^{j,l,-1}y^{-2k},
\end{split}
\end{equation}
where
\begin{equation}\label{E2.58}
\hat{w}_0^{j,l,0}=a_{j,l},\quad \hat{w}_0^{j,l,-1}=b_{j,l}.
\end{equation}
Finally, the interaction part $W^2_{j,l}$ can be written as
\begin{equation}\label{E2.59}
W^2_{j,l}(y)=\sum_{-j-1\leq m\leq j}e^{-\frac{imy^2}{4(\mathfrak{a}_1-i\mathfrak{a}_2)}} W_{j,l,m}(y),
\end{equation}
where $W_{j,l,m}$ has the following asymptotic behavior as $y\rightarrow\infty$:
\begin{equation}\small\label{E2.60}
\begin{split}
&W_{j,l,m}(y)=\sum_{k\geq m+2}\sum_{m-j\leq \tilde{i}\leq j-m\atop j-m-\tilde{i}\in2\mathbb{Z}} \sum_{s=0}^{2j+1-l} w_{k,\tilde{i},s}^{j,l,m} y^{2\nu(2i+1) -2k}(\ln y)^s,\quad m\geq1,\\
&W_{j,l,m}(y)=\sum_{k\geq -m}\sum_{-j-m-2\leq \tilde{i}\leq j+m\atop j-m-\tilde{i}\in2\mathbb{Z}} \sum_{s=0}^{2j+1-l} w_{k,\tilde{i},s}^{j,l,m} y^{2\nu(2i+1) -2k}(\ln y)^s,\quad m\leq-2,\\
&W_{j,l,0}(y)=\sum_{k\geq 1}\sum_{-j\leq \tilde{i}\leq j-2\atop j-\tilde{i}\in2\mathbb{Z}} \sum_{s=0}^{2j+1-l} w_{k,\tilde{i},s}^{j,l,0} y^{2\nu(2i+1)-2k}(\ln y)^s,\\
&W_{j,l,-1}(y)=\sum_{k\geq 1}\sum_{-j+1\leq \tilde{i}\leq j-1\atop j-\tilde{i}\in2\mathbb{Z}} \sum_{s=0}^{2j+1-l} w_{k,\tilde{i},s}^{j,l,-1} y^{2\nu(2i+1)-2k}(\ln y)^s.
\end{split}
\end{equation}
Furthermore, the asymptotic expansions \eqref{E2.57} and \eqref{E2.60} can be differentiated with respect to $y$ any number of times. In addition, any solution of the system \eqref{E2.44}, \eqref{E2.45} has the form \eqref{E2.55}, \eqref{E2.56}, \eqref{E2.57}, \eqref{E2.59} and \eqref{E2.60}.
\end{lemma}
\begin{proof}
Note that the solutions $(\mathcal{L}-\tilde{\mu}_j)f=0$ has a basis $\{f_j^1,f_j^2\}$, which have the following asymptotic behavior at infinity:
\begin{align*}
&f_j^1(y)=y^{-2i\tilde{\mu}_j(\mathfrak{a}_1-i\mathfrak{a}_2)} \sum_{k\geq0}f_{j,1}^ky^{-2k},\\
&f_j^2(y)=e^{\frac{iy^2}{4(\mathfrak{a}_1-i\mathfrak{a}_2)}} y^{2i\tilde{\mu}_j(\mathfrak{a}_1-i\mathfrak{a}_2)-2} \sum_{k\geq0}f_{j,2}^ky^{-2k},
\end{align*}
where $f_{j,1}^0=f_{j,2}^0=1$. Thus, the solutions of the homogeneous equations
\begin{equation}\label{E2.61}
\begin{cases}
\mathcal{L}g_{2j+1}=\tilde{\mu}_jg_{2j+1},\\
\mathcal{L}g_{2j}=\tilde{\mu}_jg_{2j} -\frac{i}{2(\mathfrak{a}_1-i\mathfrak{a}_2)}(2j+1)g_{2j+1}\\
\quad\quad\quad-\frac{i}{\mathfrak{a}_1-i\mathfrak{a}_2}\nu(2j+1)g_{2j+1} +2(2j+1) \frac{1}{y}\partial_y g_{2j+1},\\
\mathcal{L}g_{l}=\tilde{\mu}_jg_{l} -\frac{i}{2(\mathfrak{a}_1-i\mathfrak{a}_2)}(l+1)g_{l+1}-\frac{i}{\mathfrak{a}_1-i\mathfrak{a}_2}\nu(l+1)g_{l+1}\\
\quad\quad\quad+2(l+1)\frac{1}{y}\partial_yg_{l+1} +(l+1)(l+2)\frac{1}{y^2}g_{l+2}, \quad 0\leq l\leq 2j-1.
\end{cases}
\end{equation}
have a basis $\{\mathbf{g}_j^{\tilde{i},m}\}_{m=0,\cdots,2j+1}^{\tilde{i}=1,2}$, where
$$
\mathbf{g}_{j}^{\tilde{i},m}=\left(g^{\tilde{i},m}_{j,0},\cdots,g^{\tilde{i},m}_{j,2j+1}\right),\quad 0\leq m\leq2j+1,\,\, \tilde{i}=1,2.
$$
Here each component is defined as
\begin{equation}\label{E2.62}
\begin{split}
\sum_{l=0}^{2j+1}(\ln y-\nu\ln t)^l g^{\tilde{i},m}_{j,l}(y)
=\sum_{l=0}^{2j+1}\left(\ln y+\frac{(-1)^{\tilde{i}-1}}{2}\ln t\right)^l\xi^{\tilde{i},m}_{j,l}(y),
\end{split}
\end{equation}
where $(\xi_{j,l}^{\tilde{i},m})_{l=0,\cdots,2j+1}$ is the unique solution of the following system:
\begin{equation}\label{E2.63}
\begin{cases}
\mathcal{L}\xi_{2j+1}=\tilde{\mu}_j\xi_{2j+1},\\
\mathcal{L}\xi_{2j}=\tilde{\mu}_j\xi_{2j}-\frac{i}{\mathfrak{a}_1-i\mathfrak{a}_2}(2j+1)\left(\tilde{i} -1\right)\xi_{2j+1} +2(2j+1)\frac{1}{y}\partial_y\xi_{2j+1},\\
\mathcal{L}\xi_{l}=\tilde{\mu}_j\xi_{l}-\frac{i}{\mathfrak{a}_1-i\mathfrak{a}_2}(l+1)\left(\tilde{i} -1\right)\xi_{l+1} +2(l+1)\frac{1}{y}\partial_y\xi_{l+1}\\
\quad\quad\quad +(l+1)(l+2)\frac{1}{y^2}\xi_{l+2},\quad 0\leq l\leq 2j-1,
\end{cases}
\end{equation}
that satisfies
\begin{equation}\label{E2.64}
\begin{split}
&\xi_{j,l}^{\tilde{i},m}(y)=0,\quad l>2j+1-m,\\
&\xi_{j,2j+1-m}^{\tilde{i},m}(y)=f_j^{\tilde{i}}(y),\\
&\xi_{j,l}^{1,m}(y)=y^{2i\alpha_0+2\nu(2j+1)} \sum_{k\geq2j+1-l-m}\xi_{l,k}^1y^{-2k},\quad y\rightarrow+\infty,\\
&\xi_{j,l}^{2,m}(y)=e^{\frac{iy^2}{4(\mathfrak{a}_1-i\mathfrak{a}_2)}} y^{-2i\alpha_0-2\nu(2j+1)-2} \sum_{k\geq2j+1-l-m} \xi_{l,k}^2 y^{-2k},\quad y\rightarrow+\infty.
\end{split}
\end{equation}

For $W_{0,l}$, $l=0,1$, we have
\begin{align*}
&\mathcal{L}W_{0,1}=\tilde{\mu}_0W_{0,1},\\
&\mathcal{L}W_{0,0}=\tilde{\mu}_0W_{0,0} -\frac{i}{\mathfrak{a}_1-i\mathfrak{a}_2}\left(\frac{1}{2} +\nu\right)W_{0,1}+2\frac{1}{y}\partial_y W_{0,1}.
\end{align*}
Thus,
$$
W_{0,l}(y)=\sum_{\tilde{i}=1,2\atop m=0,1} A_{\tilde{i},m}g_{0,l}^{\tilde{i},m}(y),\quad l=0,1,
$$
where $A_{\tilde{i},m}$ are constants, $\tilde{i}=1,2$, $m=0,1$. According to \eqref{E2.62} and \eqref{E2.64}, we obtain that $W_{0,l}$, $l=0,1$ have the form \eqref{E2.55} and \eqref{E2.57}, where $\hat{w}_0^{j,l,0}=A_{1,1-l}$ and $\hat{w}_0^{j,l,-1}=A_{2,1-l}$, $l=0,1$. This combined with \eqref{E2.58} yields $A_{1,m}=a_{0,l-m}$ and $A_{2,m}=b_{0,1-m}$, $m=0,1$.

Next, we consider the case of $j\geq1$. Suppose that $W_{\tilde{i},n}$, $0\leq n\leq2\tilde{i}+1$, $\tilde{i}\leq j-1$ have has the asymptotic behavior pre-described in \eqref{E2.56}, \eqref{E2.57}, \eqref{E2.59} and \eqref{E2.60}, then it can be verified that $(\mathfrak{a}_1-i\mathfrak{a}_2)\mathcal{G}_{j,l}$ has the following form:
\begin{equation}\label{E2.65}
(\mathfrak{a}_1-i\mathfrak{a}_2)\mathcal{G}_{j,l}(y)=\sum_{-j-1\leq m\leq j}e^{-\frac{imy^2}{4(\mathfrak{a}_1-i\mathfrak{a}_2)}} \mathcal{G}^m_{j,l}(y),
\end{equation}
where $\mathcal{G}_{j,l}^m$, $m=0,-1$, satisfy
\begin{equation}\label{E2.66}
\begin{split}
&\mathcal{G}_{j,l}^m(y)=\mathcal{G}_{j,l}^{m,0}(y) +\mathcal{G}_{j,l}^{m,1}(y),\quad m=0,-1,\\
&\mathcal{G}_{j,l}^{0,0}(y)=\mathcal{G}_{j,l}\left(y; W_{\tilde{i},n}^0, 0\leq n\leq2\tilde{i}+1, 0\leq \tilde{i}\leq j-1\right),\\
&e^{\frac{iy^2}{4(\mathfrak{a}_1-i\mathfrak{a}_2)}}\mathcal{G}_{j,l}^{-1,0}(y) =\mathcal{G}_{j,l}\left(y; W_{\tilde{i},n}^1, 0\leq n\leq2\tilde{i}+1, 0\leq \tilde{i}\leq j-1\right),
\end{split}
\end{equation}
and have the following asymptotic behavior as $y\rightarrow\infty$:
\begin{equation}\label{E2.67}
\begin{split}
&\mathcal{G}_{j,l}^{0,0}(y)=\sum_{k\geq1}\sum_{s=0}^{2j+1-l}T_{k,j,s}^{j,l,0}  y^{2\nu(2j+1)-2k}(\ln y)^s,\\
&\mathcal{G}_{j,l}^{0,1}(y)=\sum_{k\geq2}\sum_{-j\leq \tilde{i}\leq j-2 \atop j-\tilde{i}\in2\mathbb{Z}}\sum_{s=0}^{2j+1-l} T_{k,j,s}^{j,l,0}y^{2\nu(2i+1) -2k}(\ln y)^s,
\end{split}
\end{equation}

\begin{equation}\label{E2.68}
\begin{split}
&\mathcal{G}_{j,l}^{-1,0}(y)=\sum_{k\geq2}\sum_{s=0}^{2j+1-l}T_{k,-j-1,s}^{j,l,-1} y^{-2\nu(2j+1)-2k}(\ln y)^s,\\
&\mathcal{G}_{j,l}^{-1,1}(y)=\sum_{k\geq1}\sum_{-j+1\leq \tilde{i}\leq j-1 \atop j-\tilde{i}+1\in2\mathbb{Z}}\sum_{s=0}^{2j+1-l} T_{k,i,s}^{j,l,-1}y^{2\nu(2i+1)-2k}(\ln y)^s.
\end{split}
\end{equation}
Finally, $\mathcal{G}_{j,l}^m$, $m\neq0, -1$, have the following asymptotic behavior as $y\rightarrow\infty$:
\begin{equation}\label{E2.69}
\begin{split}
&\mathcal{G}_{j,l}^{m}(y)=\sum_{k\geq m+1} \sum_{m-j\leq \tilde{i}\leq j-m\atop j-m-\tilde{i}\in\mathbb{Z}} \sum_{s=0}^{2j+1-l}T_{k,\tilde{i},s}^{j,l,m} y^{2\nu(2i+1)-2k}(\ln y)^s,\quad m\geq1,\\
&\mathcal{G}_{j,l}^{m}(y)=\sum_{k\geq|m|-1}\sum_{-j-m-2\leq \tilde{i}\leq j+m \atop j-\tilde{i}+m\in2\mathbb{Z}}\sum_{s=0}^{2j+1-l} T_{k,\tilde{i},s}^{j,l,m} y^{2\nu(2i+1)-2k}(\ln y)^s,\quad m\leq-2.
\end{split}
\end{equation}
Thus, integrating \eqref{E2.45} yields
\begin{equation}\label{E2.70}
\begin{split}
&W_{j,l}=\tilde{W}_{j,l}+\sum_{\tilde{i}=1,2\atop m=0,\cdots, 2j+1}A_{\tilde{i},m}\mathbf{g}_{j,l}^{\tilde{i},m},\\
&\tilde{W}_{j,l}(y)=\sum_{-j-1\leq m\leq j}e^{-\frac{imy^2}{4(\mathfrak{a}_1-i\mathfrak{a}_2)}} \tilde{W}^m_{j,l}(y),
\end{split}
\end{equation}
where $e^{-\frac{imy^2}{4(\mathfrak{a}_1-i\mathfrak{a}_2)}} \tilde{W}^m_{j,l}(y)$ is the unique solution of \eqref{E2.45} with $\mathcal{G}_{j,l}$ replaced by $e^{-\frac{imy^2}{4(\mathfrak{a}_1-i\mathfrak{a}_2)}} \mathcal{G}^m_{j,l}(y)$ that has the following asymptotic behavior as $y\rightarrow+\infty$:
\begin{equation}\label{E2.71}
\begin{split}
&\tilde{W}_{j,l}^m=\sum_{k\geq m+2}\sum_{m-j\leq \tilde{i}\leq j-m \atop j-\tilde{i}-m\in2\mathbb{Z}}\sum_{s=0}^{2j+1-l}\tilde{w}_{k,\tilde{i},s}^{j,l,m} y^{2\nu(2i+1)-2k}(\ln y)^s,\quad m\geq1,\\
&\tilde{W}_{j,l}^m(y)=\sum_{k\geq -m}\sum_{-j-m-2\leq \tilde{i}\leq j+m \atop j-m-\tilde{i}\in2\mathbb{Z}}\sum_{s=0}^{2j+1-l} \tilde{w}_{k,\tilde{i},s}^{j,l,m} y^{2\nu(2i+1)-2k}(\ln y)^s,\quad m\leq-2.
\end{split}
\end{equation}
Finally, for the case of $m=0,-1$, we have
\begin{equation}\label{E2.72}
\begin{split}
&\tilde{W}_{j,l}^0=\tilde{W}_{j,l}^{0,0}(y) +\tilde{W}_{j,l}^{0,1}(y),\\
&\tilde{W}_{j,l}^{-1}=\tilde{W}_{j,l}^{-1,0}(y) +\tilde{W}_{j,l}^{-1,1}(y),
\end{split}
\end{equation}
where $\tilde{W}_{j,l}^{0,\tilde{i}}$ and $e^{\frac{iy^2}{4(\mathfrak{a}_1-i\mathfrak{a}_2)}}\tilde{W}_{j,l}^{-1,\tilde{i}}$ are solutions of \eqref{E2.45} with $\mathcal{G}_{j,l}$ replaced by $\mathcal{G}_{j,l}^{0,\tilde{i}}$ and $e^{\frac{iy^2}{4(\mathfrak{a}_1-i\mathfrak{a}_2)}}\mathcal{G}_{j,l}^{-1,\tilde{i}}$, respectively. They have the following asymptotic behavior as $y\rightarrow\infty$:
\begin{equation}\label{E2.73}
\begin{split}
&\tilde{W}_{j,l}^{0,0}(y)=\sum_{k\geq1}\sum_{s=0}^{2j+1-l} \tilde{w}_{k,j,s}^{j,l,0} y^{2\nu(2j+1)-2k} (\ln y)^s,\\
&\tilde{W}_{j,l}^{0,1}(y)=\sum_{k\geq1} \sum_{-j\leq \tilde{i}\leq j-2\atop j-\tilde{i}\in2\mathbb{Z}} \sum_{s=0}^{2j+1-l} \tilde{w}_{k,\tilde{i},s}^{j,l,0} y^{2\nu(2i+1)-2k} (\ln y)^s,\\
&\tilde{W}_{j,l}^{-1,0}(y)= \sum_{k\geq1}\sum_{s=0}^{2j+1-l} \tilde{w}_{k,-j-1,s}^{\tilde{i},l,-1} y^{-2\nu(2j+1)-2k} (\ln y)^s,\\
&\tilde{W}_{j,l}^{-1,1}(y)=\sum_{k\geq1} \sum_{-j+1\leq \tilde{i}\leq j-1\atop j-\tilde{i}\in2\mathbb{Z}} \sum_{s=0}^{2j+1-l} \tilde{w}_{k,\tilde{i},s}^{j,l,-1} y^{2\nu(2i+1)-2k} (\ln y)^s.
\end{split}
\end{equation}
Note that $W_{j,l}^0=\tilde{W}_{j,l}^{0,0} +\sum_{m=0}^{2j+1}A_{1,m}\mathbf{g}_{j,l}^{1,m}$ and $W_{j,l}^1=e^{-\frac{imy^2}{4(\mathfrak{a}_1-i\mathfrak{a}_2)}}\tilde{W}^{-1,0}_{j,l} +\sum_{m=0}^{2j+1}A_{2,m}\mathbf{g}_{j,l}^{2,m}$ are solutions of \eqref{E2.45} with $\mathcal{G}_{j,l}$ replaced by $\mathcal{G}_{j,l}^{0,0}=\mathcal{G}_{j,l}(W_{\tilde{i},n}^0, \tilde{i}\leq j-1)$ and  $e^{\frac{iy^2}{4(\mathfrak{a}_1-i\mathfrak{a}_2)}}\mathcal{G}_{j,l}^{-1,0} =\mathcal{G}_{j,l}(W_{\tilde{i},n}^1, \tilde{i}\leq j-1)$, respectively. Thus, $W_{j,l}^{\tilde{i}}$, $\tilde{i}=0,1$, $0\leq l\leq2j+1$, have the form \eqref{E2.58}, where $\tilde{w}_0^{j,l,\tilde{i}}=A_{2j+1-l}^{1-\tilde{i}}$, $\tilde{i}=0,-1$, $l=0,\cdots,2j+1$. This combined with \eqref{E2.58} gives $A_{1,m}=a_{j,2j+1-m}$ and $A_{2,m}=b_{j,2j+1-m}$, where $m=0,\cdots,2j+1$.
\end{proof}

Let $W_{\text{in}}^{(N)}(y,t)$ be the stereographic projection of
$$
V_{\text{in}}^{(N)}\left(t^{-\nu}y,t\right) =\left(V_{\text{in},1}^{(N)}\left(t^{-\nu}y,t\right), V_{\text{in},2}^{(N)}\left(t^{-\nu}y,t\right), V_{\text{in},3}^{(N)}\left(t^{-\nu}y,t\right)\right),
$$
i.e.,
$$
W_{\text{in},1}^{(N)}(y,t) =\frac{V_{\text{in},1}^{(N)}\left(t^{-\nu}y,t\right) +iV_{\text{in},2}^{(N)}\left(t^{-\nu}y,t\right)}{1 +V_{\text{in},3}^{(N)}\left(t^{-\nu}y,t\right)} \in \mathbb{C}\cup\{\infty\}.
$$
Recalling the coordinate transformation \eqref{wW}, for $N\geq2$, we define
\begin{align*}
&W_{ss}^{(N)}(y,t)= \sum_{j=0}^N\sum_{l=0}^{2j+1} t^{\nu(2j+1)}\left(\ln\rho\right)^lW_{j,l}^{ss}(y),\\
&A_{ss}^{(N)}= -it\partial_tW_{ss}^{(N)} +\alpha_0W_{ss}^{(N)} +(\mathfrak{a}_1-i\mathfrak{a}_2)\left[\mathcal{L} W_{ss}^{(N)} +G\left(W_{ss}^{(N)}, \bar{W}_{ss}^{(N)}, \partial_yW_{ss}^{(N)}\right)\right],\\
&V_{ss}^{(N)}(\rho,t)= \left( \frac{2\operatorname{Re}\left(W_{ss}^{(N)}\right)}{1+\left|W_{ss}^{(N)}\right|^2}, \frac{2\operatorname{Im}\left(W_{ss}^{(N)}\right)}{1+\left|W_{ss}^{(N)}\right|^2}, \frac{1-\left|W_{ss}^{(N)}\right|^2}{1+\left|W_{ss}^{(N)}\right|^2} \right),\quad \rho=t^{-\nu}y,\\
&Z_{ss}^{(N)}(\rho,t)=V_{ss}^{{N}}(\rho,t)-Q(\rho).
\end{align*}

Fixing $\varepsilon_1=\frac{\nu}{2}$, according to the previous analysis, we obtain
\begin{lemma}\label{lem2.7}
For $0<t\leq T(N)$, there exists a positive constant $C=C(\mathfrak{a}_1, \mathfrak{a}_2,\nu)$, such that
\begin{description}
  \item[(i)] For any $k,l$ and $\frac{1}{10}t^{\varepsilon_1} \leq y\leq 10t^{\varepsilon_1}$,
\begin{equation}\label{E2.74}
\left|y^{-l}\partial_y^k\partial_t^i\left(W_{ss}^{(N)}-W_{\text{in}}^{(N)}\right)\right| \leq C_{k,l,i}t^{\nu\left(N+1-\frac{l+k}{2}\right)-i},\quad i=0,1.
\end{equation}
  \item[(ii)] The profile $Z_{ss}^{(N)}$ satisfies
\begin{align}
&\left\|\partial_\rho Z_{ss}^{(N)}(t)\right\|_{L^2\left(\rho d\rho, \frac{1}{10}t^{-\nu+\varepsilon_1}\leq\rho\leq10t^{-\nu+\varepsilon_2}\right)} \leq Ct^\eta, \label{E2.75}\\
&\left\|\rho^{-1}Z_{ss}^{(N)}(t)\right\|_{L^2\left(\rho d\rho, \frac{1}{10}t^{-\nu+\varepsilon_1}\leq\rho\leq10t^{-\nu+\varepsilon_2}\right)} \leq Ct^\eta, \label{E2.76}\\
&\left\|Z_{ss}^{(N)}(t)\right\|_{L^\infty\left( \frac{1}{10}t^{-\nu+\varepsilon_1}\leq\rho\leq10t^{-\nu+\varepsilon_2}\right)} \leq Ct^\eta, \label{E2.77}\\
&\left\|\rho \partial_\rho Z_{ss}^{(N)}(t)\right\|_{L^\infty\left( \frac{1}{10}t^{-\nu+\varepsilon_1}\leq\rho\leq10t^{-\nu+\varepsilon_2}\right)} \leq Ct^\eta, \label{E2.78}\\
&\left\|\rho^{-l}\partial_\rho^kZ_{ss}^{(N)}(t)\right\|_{L^2\left(\rho d\rho, \frac{1}{10}t^{-\nu+\varepsilon_1}\leq\rho\leq10t^{-\nu+\varepsilon_2}\right)} \leq Ct^{\nu+\frac{1}{2}+\eta},\quad k+l=2, \label{E2.79}\\
&\left\|\rho^{-l}\partial_\rho^kZ_{\text{in}}^{(N)}(t)\right\|_{L^2\left(\rho d\rho, \frac{1}{10}t^{-\nu+\varepsilon_1}\leq\rho\leq10t^{-\nu+\varepsilon_2}\right)} \leq Ct^{2\nu},\quad k+l\geq3, \label{E2.80}\\
&\left\|\rho^{-l}\partial_\rho^kZ_{ss}^{(N)}(t)\right\|_{L^\infty\left( \frac{1}{10}t^{-\nu+\varepsilon_1}\leq\rho\leq10t^{-\nu+\varepsilon_2}\right)} \leq Ct^{\nu+\eta},\quad k+l=1, \label{E2.81}\\
&\left\|\rho^{-l}\partial_\rho^kZ_{ss}^{(N)}(t)\right\|_{L^\infty\left( \frac{1}{10}t^{-\nu+\varepsilon_1}\leq\rho\leq10t^{-\nu+\varepsilon_2}\right)} \leq Ct^{2\nu},\quad 2\leq k+l. \label{E2.82}
\end{align}
Here (and below) $\eta$ denotes constants that depend on $\nu$ and $\varepsilon_2$, which may vary from line to line.
  \item[(iii)] The error $A_{ss}^{(N)}$ satisfies the estimate:
\begin{align}\label{E2.83}
&\left\|y^{-l}\partial_y^k\partial_t^iA_{ss}^{(N)}(t)\right\|_{L^2 \left(ydy,\frac{1}{10}t^{\varepsilon_1}\leq y\leq10 t^{-\varepsilon_2}\right)} \leq Ct^{\nu N(1-2\varepsilon_2)-i},\\
&\quad\quad\quad\quad \quad\quad\quad\quad \quad\quad\quad\quad \quad\quad\quad\quad \quad\quad\quad\quad 0\leq l+k\leq4,\, i=0,1.\notag
\end{align}
\end{description}
\end{lemma}

\subsection{Remote region $r\sim1$}\label{subsec2.4}
Next, we consider the remote region $t^{-\varepsilon_2}\leq rt^{-\frac{1}{2}}$. We use the formal solution constructed in Subsection \ref{subsec2.3}:
$$
\sum_{j\geq0}\sum_{l=0}^{2j+1}t^{\nu(2j+1)}(\ln y-\nu\ln t)^lW_{j,l}^{ss}(y).
$$
According to Lemma \ref{lem2.6}, this solution has the form \eqref{E2.55}, \eqref{E2.56}, \eqref{E2.57}, \eqref{E2.59} and \eqref{E2.60}, where $\hat{w}_k^{j,l,i}$ and $w_{k,i,s}^{j,l,m}$ are some coefficients to be determined. Note that by taking the limits $y\rightarrow\infty$ and $r\rightarrow0$, the main order terms of the expansion
$$
\sum_{j\geq0}\sum_{l=0}^{2j+1} t^{i\alpha_0+\nu(2j+1)}(\ln y-\nu\ln t)^lW_{j,l}^{ss}(t^{-\frac{1}{2}}r)
$$
is
\begin{align}\label{E2.84}
&\sum_{j\geq0}\sum_{l=0}^{2j+1} t^{i\alpha_0+\nu(2j+1)}(\ln y-\nu\ln t)^lW_{j,l}^{ss}(t^{-\frac{1}{2}}r)\notag\\
\sim&\sum_{k\geq0}\frac{t^k}{r^{2k}}\sum_{j\geq0}\sum_{l=0}^{2j+1} \hat{w}_k^{j,l,0}(\ln r)^lr^{2i\alpha_0+2\nu(2j+1)}\\
&+\frac{1}{t}e^{\frac{ir^2}{4(\mathfrak{a}_1-i\mathfrak{a}_2)t}} \sum_{k\geq0}\frac{t^k}{r^{2k}}\sum_{j\geq0}\sum_{l=0}^{2j+1} \hat{w}_k^{j,l,-1}\left(\frac{r}{t}\right)^{-2i\alpha_0-2\nu(2j+1)-2} \left(\ln\left(\frac{r}{t}\right)\right)^l.\notag
\end{align}
This inspires us to construct the solutions of \eqref{E2.38} in the region $t^{-\varepsilon_2}\leq rt^{-\frac{1}{2}}$ by perturbing the time-independent profile:
$$
\sum_{j\geq0}\sum_{l=0}^{2j+1}\beta_0(j,l)(\ln r)^l r^{2i\alpha_0+2\nu(2j+1)},
$$
where $\beta_0(j,l)=\hat{w}_0^{j,l,0}$.

Let $\theta\in C_0^\infty(\mathbb{R})$ be a cut-off function that satisfies
\begin{equation*}
\theta(\xi)=
\begin{cases}
1,\quad \text{ if }\, |\xi|\leq1,\\
0,\quad \text{ if }\, |\xi|\geq2.
\end{cases}
\end{equation*}
For $N\geq2$ and $\delta>0$, we define
$$
f_0(r)\triangleq f_0^{N}(r)=\theta\left(\frac{r}{\delta}\right) \sum_{j=0}^N\sum_{l=0}^{2j+1}\beta_0(j,l) (\ln r)^lr^{2i\alpha_0+2\nu(2j+1)}.
$$
Note that $e^{i\theta}f_0\in H^{1+2\nu-}$ and
\begin{equation}\label{E2.85}
\left\|e^{i\theta}f_0\right\|_{\dot{H}^s}\leq C\delta^{1+2\nu-s},\quad \forall0\leq s<1+2\nu.
\end{equation}

Let $w(r,t)=f_0(r)+\chi(r,t)$, then $\chi$ solves
\begin{equation}\label{E2.86}
\begin{split}
i\chi_t
=(\mathfrak{a}_1-i\mathfrak{a}_2)\left[-\Delta\chi +\frac{1}{r^2}\chi+\mathcal{V}_0\partial_r\chi +\mathcal{V}_1\chi+\mathcal{V}_2\bar{\chi}+\mathcal{N}+\mathcal{D}_0\right],
\end{split}
\end{equation}
where
\begin{align*}
\mathcal{V}_0=&\frac{4\bar{f}_0\partial_rf_0}{1+|f_0|^2},\\
\mathcal{V}_1=&-\frac{2|f_0|^2 \left(2+|f_0|^2\right)}{r^2\left(1+|f_0|^2\right)^2} -\frac{2\bar{f}_0^2(\partial_rf_0)^2}{\left(1+|f_0|^2\right)^2},\\
\mathcal{V}_2=&\frac{2\left(r^2(\partial_rf_0)^2-f_0^2\right)}{r^2 \left(1+|f_0|^2\right)^2},\\
\mathcal{D}_0=&\left(-\Delta+\frac{1}{r^2}\right)f_0 +G\left(f_0,\bar{f}_0,\partial_rf_0\right).
\end{align*}
Here $\mathcal{N}$ contains the terms at least quadratic in $\chi$, which has the following form:
\begin{equation}\label{E2.87}
\begin{split}
\mathcal{N}=&N_0\left(\chi,\bar{\chi}\right)+\chi_r N_1\left(\chi,\bar{\chi}\right) +\chi_r^2N_2\left(\chi,\bar{\chi}\right),\\
N_0(\chi,\bar{\chi})=&G\left(f_0+\chi,\bar{f}_0 +\bar{\chi},\partial_rf_0\right)- G\left(f_0,\bar{f}_0,\partial_rf_0\right) -\mathcal{V}_1\chi-\mathcal{V}_2\bar{\chi},\\
N_1(\chi,\bar{\chi})=&\frac{4\partial_rf_0 \left(\bar{f}_0+\bar{\chi}\right)}{1+|f_0+\chi|^2}-\mathcal{V}_0,\\
N_2(\chi,\bar{\chi})=&\frac{2(\bar{f}_0+\bar{\chi})}{1+|f_0+\chi|^2}.
\end{split}
\end{equation}
By \eqref{E2.55}, \eqref{E2.56}, \eqref{E2.57}, \eqref{E2.59} and \eqref{E2.60}, we construct $\chi$ of the form:
\begin{align}\label{E2.88}
\chi(r,t)=\sum_{q\geq0 \atop k\geq1}t^{2\nu q+k}\sum_{m\in\natural}\sum_{s=0}^qe^{-im\Phi}(\ln r-\ln t)^s g_{k,q,m,s}(r),
\end{align}
where
\begin{align*}
&\Phi(r,t)=2\alpha_0 \ln t+\frac{r^2}{4(\mathfrak{a}_1-i\mathfrak{a}_2)t}+\varphi(r),\\
&\natural=\left\{m:-\min\{k,q\}\leq m\leq\min\left\{(k-2)_+,q\right\},\,\, q-m\in2\mathbb{Z} \right\}.
\end{align*}
Here $\varphi(r)$ is to be determined.

By \eqref{E2.88}, we get
\begin{align}
\chi_t=&\sum_{q\geq0 \atop k\geq2}t^{2\nu q+k-2} \sum_{m\in\natural}\sum_{s=0}^qe^{-im\Phi}(\ln r-\ln t)^s\notag\\
&\cdot\bigg[(2\nu q+k-1-2 im\alpha_0) g_{k-1,q,m,s}+\frac{imr^2}{4(\mathfrak{a}_1-i\mathfrak{a}_2)}g_{k,q,m,s}\notag\\
&\quad\quad \quad\quad-(s+1)g_{k-1,q,m,s+1}\bigg],\notag\\
\chi_r=&\sum_{q\geq0 \atop k\geq2}t^{2\nu q+k-2} \sum_{m\in\natural}\sum_{s=0}^qe^{-im\Phi}(\ln r-\ln t)^s\bigg[\frac{-imr}{2(\mathfrak{a}_1-i\mathfrak{a}_2)}g_{k-1,q,m,s}\label{chirA}\\
&+(\partial_r-im\varphi_r)g_{k-2,q,m,s} +\frac{s+1}{r}g_{k-2,q,m,s+1}\bigg],\notag\\
\chi_{rr}=&\sum_{q\geq0 \atop k\geq2}t^{2\nu q+k-2} \sum_{m\in\natural}\sum_{s=0}^qe^{-im\Phi}(\ln r-\ln t)^s\bigg\{(-im\varphi_{rr}-m^2\varphi_r^2\notag\\
& -2im\varphi_r\partial_r+\partial_{rr}) g_{k-2,q,m,s} -\frac{m^2r^2}{4(\mathfrak{a}_1-i\mathfrak{a}_2)^2}g_{k,q,m,s}\notag\\
& -\frac{2m^2r\varphi_r +imr\partial_r+imr}{2(\mathfrak{a}_1-i\mathfrak{a}_2)} g_{k-1,q,m,s}\notag\\
& \left(-\frac{im(s+1)}{r}\varphi_r -\frac{s+1}{r^2}+\frac{2(s+1)}{r}\partial_r\right)g_{k-2,q,m,s+1}\notag\\
&+\frac{(s+1)(s+2)}{r^2}g_{k-2,q,m,s+2} -\frac{im(s+1)}{2(\mathfrak{a}_1-i\mathfrak{a}_2)}g_{k-1,q,m,s+1}\bigg\}.\notag
\end{align}
Thus, substituting the hypothesis \eqref{E2.88} into $(\mathfrak{a}_1-i\mathfrak{a}_2)N_i$, $i=0,1,2$, and $-i\chi_t+(\mathfrak{a}_1-i\mathfrak{a}_2) \left[-\Delta\chi +\frac{1}{r^2}\chi+\mathcal{V}_0\partial_r\chi +\mathcal{V}_1\chi+\mathcal{V}_2\bar{\chi}\right]$, we get
\begin{align}
& -i\chi_t+(\mathfrak{a}_1-i\mathfrak{a}_2)\left[ -\Delta\chi +\frac{1}{r^2}\chi+\mathcal{V}_0\partial_r\chi +\mathcal{V}_1\chi+\mathcal{V}_2\bar{\chi}\right]\notag\\
&\quad\quad =\sum_{q\geq0 \atop k\geq2}t^{2\nu q+k-2} \sum_{m\in\natural}\sum_{s=0}^qe^{-im\Phi}(\ln r-\ln t)^s \Psi_{k,q,m,s}^{lin},\notag\\
&(\mathfrak{a}_1-i\mathfrak{a}_2)N_0(\chi,\bar{\chi})=\sum_{q\geq0 \atop k\geq4}t^{2\nu q+k-2}\sum_{m\in\natural} \sum_{s=0}^qe^{-im\Phi}(\ln r-\ln t)^s \Psi_{k,q,m,s}^{nl,0},\\
&(\mathfrak{a}_1-i\mathfrak{a}_2)\chi_r N_1(\chi,\bar{\chi})=\sum_{q\geq0 \atop k\geq3}t^{2\nu q+k-2}\sum_{m\in\natural} \sum_{s=0}^qe^{-im\Phi}(\ln r-\ln t)^s \Psi_{k,q,m,s}^{nl,1},\notag\\
&(\mathfrak{a}_1-i\mathfrak{a}_2)(\chi_r)^2N_2(\chi,\bar{\chi})=\sum_{q\geq0 \atop k\geq2}t^{2\nu q+k-2}\sum_{m\in\natural} \sum_{s=0}^qe^{-im\Phi}(\ln r-\ln t)^s \Psi_{k,q,m,s}^{nl,2},\notag
\end{align}
where
\begin{equation}\label{E2.89}
\Psi_{k,q,m,s}^{lin} =\frac{m(1+m)}{4(\mathfrak{a}_1-i\mathfrak{a}_2)}r^2 g_{k,q,m,s}+\Psi_{k,q,m,s}^{lin,1} +\Psi_{k,q,m,s}^{lin,2},
\end{equation}
where $\Psi_{k,q,m,s}^{lin,1}$ and $\Psi_{k,q,m,s}^{lin,2}$ depend only on $g_{k-1,q,m,s'}$, $s'=s,s+1$ and $g_{k-2,q,m,s'}$, $s'=s,s+1,s+2$, respectively. In fact, they can be written as
\begin{align}
\Psi_{k,q,m,s}^{lin,1}=&\bigg[-i(2\nu q+k-1-2im\alpha_0) +m^2r\varphi_r+\frac{im}{2} +\frac{imr}{2}\partial_r\label{E2.90}\\
&-\left(\mathcal{V}_0-\frac{1}{r}\right)\frac{imr}{2}\bigg]g_{k-1,q,m,s} +i(m+1)(s+1) g_{k-1,q,m,s+1},\notag\\
\Psi_{k,q,m,s}^{lin,2}=&\bigg[-(\mathfrak{a}_1-i\mathfrak{a}_2)(-im\varphi_{rr}
-m^2\varphi_r^2 -2im\varphi_r\partial_r +\partial_{rr})\label{E2.91}\\
&+(\mathfrak{a}_1-i\mathfrak{a}_2)\left(\mathcal{V}_0 -\frac{1}{r}\right)(-im\varphi_r+\partial_r) +\mathcal{V}_1 +\frac{1}{r^2}\bigg] g_{k-2,q,m,s}\notag\\
&-(\mathfrak{a}_1-i\mathfrak{a}_2) \bigg[-\frac{2im (s+1)}{r}\varphi_r -\frac{s+1}{r^2}+\frac{2(s+1)}{r}\partial_r\notag\\
&\quad\quad\quad \quad\quad\quad\quad\quad-\left(\mathcal{V}_0-\frac{1}{r}\right) \frac{s+1}{r}\bigg]g_{k-2,q,m,s+1}\notag\\
&-(\mathfrak{a}_1-i\mathfrak{a}_2)\frac{(s+1)(s+2)}{r^2}g_{k-2,q,m,s+2}\notag\\ &+(\mathfrak{a}_1-i\mathfrak{a}_2)\mathcal{V}_2\bar{g}_{k-2,q,-m,s}.\notag
\end{align}
Here and below we assume that $g_{k,q,m,s}=0$ if $(k,q,m,s)\notin\Omega$, where
$$
\Omega=\left\{(k,q,m,s)|k\geq1,\, q\geq0,\, 0\leq s\leq q,\, q-m\in2\mathbb{Z},\, -\min\{k,q\}\leq m\leq \min\{k-1,q\}\right\}.
$$
Note that $\Psi_{k,q,m,s}^{nl,i}$, $i=0,1$, depend only on $g_{k',q',m',s'}$, where $k'\leq k-2$, i.e.,
\begin{align*}
&\Psi_{k,q,m,s}^{nl,0}=\Psi_{k,q,m,s}^{nl,0}\left(r;g_{k',q',m',s'},\, k'\leq k-3\right),\\
&\Psi_{k,q,m,s}^{nl,1}=\Psi_{k,q,m,s}^{nl,1}\left(r;g_{k',q',m',s'},\, k'\leq k-2\right).
\end{align*}
By \eqref{chirA}, $\Psi_{k,q,m,s}^{nl,2}$ has the follow structure:
\begin{equation}\label{E2.92}
\begin{split}
&\Psi_{2,q,m,s}^{nl,2}=-\delta_{m,-2}\frac{r^2\bar{f}_0}{2(\mathfrak{a}_1-i\mathfrak{a}_2) \left(1+\left|f_0\right|^2\right)} \sum_{q_1+q_2=q \atop s_1+s_2=s}g_{1,q_1,-1,s_1}g_{1,q_2,-1,s_2},\\
&\Psi_{k,q,m,s}^{nl,2}=\Psi_{k,q,m,s}^{nl,2,0}+\tilde{\Psi}_{k,q,m,s}^{nl,2},\quad k\geq3,\\
&\Psi_{k,q,m,s}^{nl,2,0}=\frac{(m+1)r^2\bar{f}_0}{(\mathfrak{a}_1-i\mathfrak{a}_2) \left(1+\left|f_0\right|^2\right)}\sum_{q_1+q_2=q \atop s_1+s_2=s} g_{1,q_1,-1,s_1}g_{1,q_2,m+1,s_2},
\end{split}
\end{equation}
where $\tilde{\Psi}^{nl,2}_{k,q,m,s}$ depends only on $g_{k',q',m',s'}$ $(k'\leq k-2)$, i.e.,
$$
\tilde{\Psi}^{nl,2}_{k,q,m,s}=\tilde{\Psi}^{nl,2}_{k,q,m,s}\left(r;\, g_{k',q',m',s'},\, k'\leq k-2\right).
$$

Note that
$$
\Psi^{nl,2,0}_{k,q,-1,s}=0,\quad \forall k,q,s.
$$
Thus, \eqref{E2.86} is equivalent to
\begin{equation}\label{E2.93}
\begin{cases}
\Psi_{2,0,0,0}^{lin}+(\mathfrak{a}_1-i\mathfrak{a}_2)\mathcal{D}_0=0,\\
\Psi_{k,q,m,s}^{lin} +\Psi_{k,q,m,s}^{nl}=0,\quad (k,q,m,s)\in\Omega,\,\, (k,q,m,s)\neq(2,0,0,0),
\end{cases}
\end{equation}
where $\Psi_{k,q,m,s}^{nl}=\Psi_{k,q,m,s}^{nl,0}+ \Psi_{k,q,m,s}^{nl,1} +\Psi_{k,q,m,s}^{nl,2}$.

Now we rewrite \eqref{E2.93} as the following system for $k\geq1$:
\begin{equation}\label{E2.94}
\begin{cases}
\Psi_{2,0,0,0}^{lin}+(\mathfrak{a}_1-i\mathfrak{a}_2)\mathcal{D}_0=0,\\
\Psi_{2,2j,0,s}^{lin} =0,\quad (j,s)\neq(0,0),\\
\Psi_{2,2j+1,-1,s}^{lin} =0,
\end{cases}
\end{equation}
and
\begin{equation}\label{E2.95}
\begin{cases}
\Psi_{k+1,q,m,s}^{lin} +\Psi_{k+1,q,m,s}^{nl}=0,\quad m=0,-1,\,\, k\geq2,\\
\Psi_{k,q,m,s}^{lin} +\Psi_{k,q,m,s}^{nl}=0,\quad m\neq0,-1,\,\, k\geq2.
\end{cases}
\end{equation}

For \eqref{E2.94}, select $\varphi$ as
\begin{equation}\label{E2.96}
\varphi(r) =-i\int_0^r\frac{\bar{f}_0(s)\partial_sf_0(s)-f_0(s) \partial_s\bar{f}_0(s)}{1+\left|f_0(s)\right|^2} ds.
\end{equation}
By \eqref{E2.90}, we rewrite \eqref{E2.94} of the form:
\begin{equation}\label{E2.97}
\begin{cases}
g_{1,0,0,0}=-i(\mathfrak{a}_1-i\mathfrak{a}_2)\mathcal{D}_0,\\
(4\nu j+1)g_{1,2j,0,s} -(s+1)g_{1,2j,0,s+1}=0,\quad (j,s)\neq(0,0),\\
\left[2\nu(2j+1)+2i\alpha_0+2 -r\frac{d}{dr}\ln\left(1+\left|f_0\right|^2\right)\right] g_{1,2j+1,-1,s}\\
\quad\quad \quad\quad \quad\quad \quad\quad \quad\quad \quad\quad \quad\quad \quad\quad \quad\quad
+r\partial_r g_{1,2j+1,-1,s}=0.
\end{cases}
\end{equation}
By \eqref{E2.84}, we get a solution of \eqref{E2.97}:
\begin{equation}\label{E2.98}
\begin{cases}
&g_{1,0,0,0}=-i(\mathfrak{a}_1-i\mathfrak{a}_2)\mathcal{D}_0,\\
&g_{1,2j,0,s}=0,\quad(j,s)\neq(0,0),\\
&g_{1,2j+1,-1,s}=\beta_1(j,s)\left(1+\left|f_0\right|^2\right) r^{-2i\alpha_0-2\nu(2j+1)-2},\\
&\quad\quad \quad\quad \quad\quad \quad\quad \quad\quad \quad\quad \quad\quad \quad\quad 0\leq s\leq 2j+1,\, 0\leq j\leq N,\\
&g_{1,2j+1,-1,s}=0,\quad j>N,
\end{cases}
\end{equation}
where $\beta_1(j,l)=\hat{w}_0^{j,l,-1}$.

In order to solve \eqref{E2.95}, we first introduce some new notation: For $m\in\mathbb{Z}$, let $\mathcal{A}_m$ be the space consisting of all continuous functions $a:\mathbb{R}_+\rightarrow\mathbb{C}$ satisfying
\begin{description}
  \item[(i)] $a\in C^\infty(\mathbb{R}_+^*)$ and $\operatorname{supp}(a)\subset\{r\leq2\delta\}$;
  \item[(ii)] For $0\leq r<\delta$, $a$ has the following absolutely convergent expansion:
  $$
  a(r)=\sum_{n\geq K(m)\atop n-m-1\in 2\mathbb{Z}}\sum_{l=0}^{n}\alpha_n(\ln r)^l r^{2\nu n},
  $$
  where $K(m)$ is defined by
  \begin{equation*}
  K(m)=
  \begin{cases}
  m+1,\quad m\geq0,\\
  |m|-1, \quad m\leq-1.
  \end{cases}
  \end{equation*}
\end{description}
In addition, for $k\geq1$, let $\mathcal{B}_k$ be the space consisting of all continuous functions $b:\mathbb{R}_+\rightarrow \mathbb{C}$ satisfying
\begin{description}
  \item[(i)] $b\in C^\infty(\mathbb{R}_+^*)$;
  \item[(ii)] For $0\leq r<\delta$, $b$ has the following absolutely convergent expansion:
  $$
  b(r)=\sum_{n=0}^\infty\sum_{l=0}^{2n}\beta_{n,l}(\ln r)^l r^{4\nu n};
  $$
  \item[(iii)] For $r\geq2\delta$, $b$ is a polynomial with degree $k-1$.
\end{description}
Finally, assume $\mathcal{B}_k^0=\left\{b\in\mathcal{B}_k|b(0)=0\right\}$.

Thus, for any $m$ and $k$, we have $r\partial_r\mathcal{A}_m\subset\mathcal{A}_m$, $r\partial_r\mathcal{B}_k\subset\mathcal{B}_k$ and $\mathcal{B}_k\mathcal{A}_m\subset\mathcal{A}_m$. In addition, note that
\begin{align*}
&f_0\in \mathcal{A}_0,\quad \varphi\in\mathcal{B}_1^0,\quad g_{1,0,0,0}\in r^{2i\alpha_0-2}\mathcal{A}_0,\\
&g_{1,2j+1,-1,s}\in r^{-2i\alpha_0-2\nu(2j+1)-2}\mathcal{B}_1,\quad 0\leq s\leq 2j+1.
\end{align*}
Furthermore, for any $(k,q,m,s)\in\Omega$, if $g_{k,q,m,s}\in r^{2(1+2m)i\alpha_0-2\nu q-2k}\mathcal{A}_m$ ($m\neq-1$) and $g_{k,q,-1,s}\in r^{-2i\alpha_0-2\nu q-2k}\mathcal{B}_k$, then
\begin{equation}\label{E2.99}
\begin{split}
&\Psi_{k,q,m,s}^{lin,i},\, \Psi_{k,q,m,s}^{nl,i},\, \tilde{\Psi}_{k,q,m,s}^{nl,2}\in r^{2(1+2m)i\alpha_0-2\nu q-2(k-1)}\mathcal{A}_m,\quad m\neq-1,\\
&\Psi_{k,q,-1,s}^{lin,2},\, \Psi_{k,q,-1,s}^{nl,i},\, \tilde{\Psi}_{k,q,-1,s}^{nl,2}\in r^{2i\alpha_0-2\nu q-2(k-1)}\mathcal{B}_{k-2},
\end{split}
\end{equation}
where $i=1,2$, $j=0,1,2$.

For \eqref{E2.95}, according to \eqref{E2.89}, \eqref{E2.90}, \eqref{E2.90}, \eqref{E2.91}, \eqref{E2.92} and \eqref{E2.96}, we can rewrite it as
\begin{equation}\label{E2.100}
\begin{cases}
&(2\nu q+k)g_{k,q,0,s}-(s+1)g_{k,q,0,s+1}=C_{k,q,0,s}+D_{k,q,s},\\
&r\partial_r g_{k,q,-1,s}+ \left(2\nu q+k+1- \frac{r\left(\bar{f}_0\partial_rf_0+ f_0\partial_r\bar{f}_0\right)}{1+\left|f_0\right|^2}\right) g_{k,q,-1,s}=C_{k,q,-1,s},\\
&\frac{m(1+m)}{4(\mathfrak{a}_1-i\mathfrak{a}_2)}r^2g_{k,q,m,s}= B_{k,q,m,s},\quad m\neq0,-1,
\end{cases}
\end{equation}
where $B_{k,q,m,s}$ and $C_{k,q,m,s}$ depend only on $g_{k',q',m',s'}$, $k'\leq k-1$, i.e.,
\begin{align*}
&B_{k,q,m,s}=B_{k,q,m,s}\left(r;\, g_{k',q',m',s'},\, k'\leq k-1\right),\quad m\neq0,-1,\\
&C_{k,q,m,s}=C_{k,q,m,s}\left(r;\, g_{k',q',m',s'},\, k'\leq k-1\right),\quad m=0,-1,
\end{align*}
More precisely, they have the form:
\begin{equation}\label{E2.101}
\begin{split}
&B_{k,q,m,s}=-\Psi_{k,q,m,s}^{lin,1} -\Psi_{k,q,m,s}^{lin,2} -\Psi_{k,q,m,s}^{nl},\quad m\neq0,-1,\\
&C_{k,q,m,s}=-i\Psi_{k+1,q,m,s}^{lin,2}-i\tilde{\Psi}_{k+1,q,m,s}^{nl},\quad m=0,-1.
\end{split}
\end{equation}
Finally, $D_{k,q,s}$ depend only on $g_{k,q,l,s}$:
\begin{equation}\label{E2.102}
D_{k,q,s}=-i\Psi_{k+1,q,0,s}^{nl,2,0}
=\frac{-i r^2\bar{f}_0}{(\mathfrak{a}_1-i\mathfrak{a}_2) \left(1+\left|f_0\right|^2\right)}\sum_{q_1+q_2=q\atop s_1+s_2=s}g_{1,q_1,-1,s_1}g_{k,q_2,1,s_2},
\end{equation}
where $D_{2,q,s}=0$.
\begin{remark}\label{re2.8}
It can be verified that if
\begin{align*}
&g_{k,q,m,s}=0,\quad \forall q>(2N+1)(2k-2),\,\, m\neq0,1,\\
&g_{k,q,m,s}=0,\quad \forall q>(2N+1)(2k-1),\,\, m=0,1,
\end{align*}
then
\begin{align*}
&B_{k,q,m,s}=0,\quad \forall q>(2N+1)(2k-2),\,\, m\neq0,1,\\
&C_{k,q,m,s}=0,\quad \forall q>(2N+1)(2k-1),\,\, m=0,1,\\
&D_{k,q,s}=0,\quad \forall q>(2N+1)(2k-1).
\end{align*}
\end{remark}

Now we prove the following result.
\begin{lemma}\label{lem2.9}
The system \eqref{E2.100} exists a unique solution $(g_{k,q,m,s})_{(k,q,m,s)\in\Omega\atop k\geq2}$, which satisfies
\begin{equation}\label{E2.103}
\begin{split}
&g_{k,q,m,s}\in r^{2(2m+1)i\alpha_0-2\nu q-2k}\mathcal{A}_m,\quad m\neq-1,\\
&g_{k,q,-1,s}\in r^{-2i\alpha_0-2\nu q-2k}\mathcal{B}_k.
\end{split}
\end{equation}
Moreover,
\begin{equation}\label{E2.104}
\begin{split}
&g_{k,q,m,s}=0,\quad \forall q>(2N+1)(2k-2),\,\, m\neq0,-1,\\
&g_{k,q,m,s}=0,\quad \forall q>(2N+1)(2k-1),\,\, m=0,-1.
\end{split}
\end{equation}
\end{lemma}
\begin{proof}
For the case of $k=2$, by \eqref{E2.100}, \eqref{E2.101} and \eqref{E2.92}, we get
\begin{align}
&(4\nu j+2)g_{2,2j,0,s}-(s+1)g_{2,2j,0,s+1}=C_{2,2j,0,s},\quad 0\leq s\leq2j,\, 0\leq j,\label{E2.107}\\
&r\partial_r g_{2,2j+1,-1,s}+\left(2\nu(2j+1)+3- \frac{r\left(\bar{f}_0\partial_rf_0+f_0\partial_r\bar{f}_0\right)}{1+\left|f_0\right|^2}\right) g_{2,2j+1,-1,s}\label{E2.106}\\
&\quad =C_{2,2j+1,-1,s},\quad 0\leq s\leq2j+1,\, 0\leq j,\notag\\
&\frac{1}{2(\mathfrak{a}_1-i\mathfrak{a}_2)}r^2g_{2,2j,-2,s}=B_{2,2j,-2,s},\quad 0\leq s\leq2j,\, 1\leq j.\label{E2.105}
\end{align}
Note that $B_{2,q,m,s}$ and $C_{2,q,m,s}$ depend only on $g_{1,q',m',s'}$, so they can be regarded as known here. By \eqref{E2.99}, \eqref{E2.101} and Remark \ref{re2.8}, they satisfy
\begin{align*}
&B_{2,q,-2,s}\in r^{-6i\alpha_0-2\nu q-2}\mathcal{A}_{-2},\\
&C_{2,q,0,s}\in r^{2i\alpha_0-2\nu q-4}\mathcal{A}_0,\\
&C_{2,q,-1,s}\in r^{-2i\alpha_0-2\nu q-4}\mathcal{B}_1,\\
&B_{2,q,-2,s}=0,\quad q>2(2N+1),\\
&C_{2,q,m,s}=0,\quad q>3(2N+1),\,\, m=0,-1.
\end{align*}
Thus, by \eqref{E2.107}, \eqref{E2.106} and \eqref{E2.105}, we get
\begin{align}\label{E2.108}
&g_{2,2j,-2,s}=\frac{2(\mathfrak{a}_1-i\mathfrak{a}_2)}{r^2} B_{2,2j,-2,s}\in r^{-6i\alpha_0-2\nu q-4}\mathcal{A}_{-2},\notag\\
&\quad\quad\quad \quad\quad\quad \quad\quad\quad \quad\quad\quad \quad\quad\quad \quad\quad\quad 0\leq s\leq2j,\,\, 1\leq j,\notag\\
&g_{2,2j,0,2j}=\frac{1}{4j\nu+2} C_{2,2j,0,2j}\in r^{2i\alpha_0-4\nu j-4}\mathcal{A}_0,\quad 0\leq j,\notag\\
&g_{2,2j,0,s}=\frac{1}{4j\nu+2} C_{2,2j,0,s}+\frac{s+1}{4j\nu+2}g_{2,2j,0,s+1} \in r^{2i\alpha_0-4\nu j-4}\mathcal{A}_0,\\
&\quad\quad\quad \quad\quad\quad \quad\quad\quad \quad\quad\quad \quad\quad\quad \quad\quad\quad \quad\quad\quad \quad\quad\quad 0\leq s\leq 2j,\notag\\
&g_{2,2j,-2,s}=0,\quad j>2N+1,\notag\\
&g_{2,2j,0,s}=0,\quad j\geq 3N+2,\notag
\end{align}
For \eqref{E2.106}, we set
$$
g_{2,2j+1,-1,s}= r^{-2i\alpha_0-3-2\nu(2j+1)}\left(1+\left|f_0\right|^2\right) \hat{g}_{2,2j+1,-1,s},
$$
then $\hat{g}_{2,2j+1,-1,s}$ satisfies
\begin{equation}\label{E2.109}
\partial_r \hat{g}_{2,2j+1,-1,s}=r^{-2}\hat{C}_{2,2j+1,-1,s},
\end{equation}
where
$$
\hat{C}_{2,2j+1,-1,s}= r^{2i\alpha_0+4+2\nu(2j+1)}\frac{1}{1+\left|f_0\right|^2} C_{2,2j+1,-1,s}.
$$
Since $C_{2,2j+1,-1,s}\in r^{-2i\alpha_0+4+2\nu(2j+1)-4}\mathcal{B}_1$, we obtain that
\begin{description}
  \item[(i)] For $0\leq r<\delta$, $\hat{C}_{2,2j+1,-1,s}$ has an absolutely convergent expansion of the following form:
      $$
      \hat{C}_{2,2j+1,-1,s}=\sum_{n=0}^\infty\sum_{l=0}^{2n} \beta_{n,l}r^{4\nu n}(\ln r)^l,
      $$
  \item[(ii)] For $r\geq2\delta$, $\hat{C}_{2,2j+1,-1,s}$ is a constant.

  Thus, \eqref{E2.109} has a unique solution $\hat{g}_{2,2j+1,-1,s} \in \frac{1}{r}\mathcal{B}_2$, which can be written as
  \begin{align*}
  \hat{C}_{2,2j+1,-1,s}(r)=&\int_{0}^r \left[\frac{1}{\rho^2}\left(\hat{C}_{2,2j+1,-1,s}(\rho) -\beta_{0,0}\right) -\frac{1}{r}\beta_{0,0}\right]d\rho,\\
   &\quad\quad\quad \quad\quad\quad \quad\quad\quad \quad\quad\quad0\leq s\leq 2j+1,\,\,0\leq j.
  \end{align*}
  Finally, since $C_{2,q,-1,s}=0$ when $q>3(2N+1)$, we have
  $$
  g_{2,2j+1,-1,s}=0,\quad j>3N+1.
  $$
\end{description}

By the induction, suppose that for $k=2,\cdots,l-1$ $(l\geq3)$, \eqref{E2.100} exists a solution $\left(g_{k,q,m,s}\right)_{(k,q,m,s)\in\Omega\atop 2\leq k\leq l-1}$, which satisfies \eqref{E2.103} and \eqref{E2.104}. For the case of $k=l$, according to the last equation of \eqref{E2.100}, we get
$$
\frac{ m(m+1)}{4(\mathfrak{a}_1-i\mathfrak{a}_2)}r^2g_{l,q,m,s}= B_{l,q,m,s}, \quad m\neq0,-1,
$$
where $B_{l,q,m,s}$ are known, and by \eqref{E2.99}, \eqref{E2.101} and Remark \ref{re2.8}, they satisfy
\begin{align*}
&B_{l,q,m,s}\in r^{2(2m+1)i\alpha_0-2\nu q-2(l-1)}\mathcal{A}_m,\\
&B_{l,q,m,s}=0,\quad q>2(2N+1)(2l-2).
\end{align*}
Thus, we obtain that if $m\neq0,-1$, then
\begin{equation}\label{E2.110}
\begin{split}
&g_{l,q,m,s}=\frac{4(\mathfrak{a}_1-i\mathfrak{a}_2)}{m(m+1)r^2}B_{l,q,m,s}\in r^{2(2m+1)i\alpha_0-2\nu q-2l} \mathcal{A}_m,\\
&g_{l,q,m,s}=0,\quad q>2(2N+1)(2l-2).
\end{split}
\end{equation}
Next, we consider the equation of $g_{l,2j,0,s}$:
\begin{equation}\label{E2.111}
(4\nu j+l)g_{l,2j,0,s} -(s+1)g_{l,2j,0,s+1} =C_{l,2j,0,s}+D_{l,2j,s},\quad 0\leq s\leq 2j,\,\, 0\leq j.
\end{equation}
Note that the term on the right hand side $C_{l,2j,0,s}+D_{l,2j,s}$ depends only on $g_{l,q_1,1,s_1}$ and $g_{k,q_2,m_2,s_2}$, $k\leq l-1$. Moreover, by \eqref{E2.99}, \eqref{E2.101}, \eqref{E2.110} and Remark \ref{re2.8}, it satisfies
\begin{align*}
&C_{l,2j,0,s}+D_{l,2j,s}\in r^{2i\alpha_0-4\nu j-2l}\mathcal{A}_0,\\
&C_{l,2j,0,s}+D_{l,2j,s}=0,\quad j>(2N+1)(2l-1).
\end{align*}
Therefore, the solutions of \eqref{E2.111} satisfy
\begin{align*}
&g_{l,2j,0,s}\in r^{2i\alpha_0-4\nu j-2l}\mathcal{A}_0, \quad 0\leq s\leq2j,\,\, 0\leq j,\\
&g_{l,2j,0,s}=0,\quad j>(2N+1)(2l-1).
\end{align*}
Finally, for $g_{l,2j+1,-1,s}$, $0\leq s\leq2j+1$, $0\leq j$, by \eqref{E2.100}, we have
\begin{equation}\label{E2.112}
\begin{split}
&\left(2\nu(2j+1) +l+1- \frac{r\left(\bar{f}_0\partial_rf_0+f_0\partial{f}_0\right)}{1+\left|f_0\right|^2}\right) g_{l,2j+1,-1,s}+r\partial_r g_{l,2j+1,m,s}\\
&\quad\quad\quad \quad\quad\quad \quad\quad\quad \quad\quad\quad\quad\quad\quad\quad\quad\quad \quad\quad\quad \quad\quad\quad \quad =C_{l,2j+1,-1,s},
\end{split}
\end{equation}
where $C_{l,2j+1,-1,s}\in r^{-2i\alpha_0-2\nu(2j+1)-2l}\mathcal{B}_{l-1}$ satisfies
\begin{equation}\label{E2.113}
C_{l,2j+1,-1,s}=0,\quad 2j>(2N+1)(2l-1).
\end{equation}
Equation \eqref{E2.112} exists a unique solution $g_{l,2j+1,-1,s}\in r^{-2i\alpha_0-2\nu(2j+1)-2l}\mathcal{B}_l$. More precisely, it can be written as
\begin{align*}
g_{l,2j+1,-1,s}=&r^{-2i\alpha_0-2\nu(2j+1)-l-1}\left(1+\left|f_0\right|^2\right) \hat{g}_{l.2j+1,-1,s},\\
\hat{g}_{l,2j+1,-1,s}=&\int_0^r\rho^{-l}\left[\hat{C}_{l.2j+1,-1,s} -\sum_{0\leq n\leq \frac{l-1}{4\nu}} \sum_{p=0}^{2n}\beta_{n,p}\rho^{4\nu n}(\ln\rho)^p\right]d\rho\\
&-\int_r^\infty\rho^{-l}\sum_{0\leq n\leq \frac{l-1}{4\nu}} \sum_{p=0}^{2n}\beta_{n,p}\rho^{4\nu n}(\ln\rho)^pd\rho,
\end{align*}
where
\begin{align*}
&\hat{C}_{l,2j+1,-1,s}=r^{2i\alpha_0+2\nu(2j+1)+2l} \frac{1}{1+\left|f_0\right|^2} C_{l,2j+1,-1,s},\\
&\hat{C}_{l,2j+1,-1,s}=\sum_{n=0}^\infty\sum_{p=0}^{2n}\beta_{n,p}r^n(\ln r)^p,\quad r<\delta.
\end{align*}
By \eqref{E2.113}, we get
$$
g_{l,2j+1,-1,s}=0,\quad 2j+1>(2N+1)(2l-1).
$$
\end{proof}

We define
\begin{align*}
&w^{(N)}_{\text{rem}}(r,t)= f_0(r)+\sum_{(k,q,m,s)\in\Omega\atop k\leq N}t^{k+2\nu q}e^{-im\Phi} (\ln r-\ln t)^s g_{k,q,m,s}(r),\\
&A^{(N)}_{\text{rem}}(r,t)=-i\partial_tw^{(N)}_{\text{rem}} -\Delta w^{(N)}_{\text{rem}} +\frac{1}{r^2}w^{(N)}_{\text{rem}} +G\left(w^{(N)}_{\text{rem}}, \bar{w}^{(N)}_{\text{rem}}, \partial_rw^{(N)}_{\text{rem}}\right),\\
&W^{(N)}_{\text{rem}}(r,t)=w^{(N)}_{\text{rem}}\left(rt^{-\frac{1}{2}}, (\mathfrak{a}_1-i\mathfrak{a}_2)t\right).
\end{align*}
According to the previous analysis, we obtain
\begin{lemma}\label{lem2.10}
There exist $T(N,\delta)>0$ and $C=C(\mathfrak{a}_1,\mathfrak{a}_2,\nu)>0$ such that for any $0<t\leq T(N,\delta)$, the following hold.
\begin{description}
  \item[(i)] For any $0\leq l$, $k\leq4$, $i=0,1$ and $\frac{1}{10}t^{-\varepsilon_2}\leq y\leq 10t^{-\varepsilon_2}$, if $N$ sufficiently large (and depends on $\varepsilon_2$), then
      \begin{equation}\label{E2.114}
      \left|y^{-l}\partial_y^k\partial_t^i \left(W^{(N)}_{ss}-W^{(N)}_{\text{rem}}\right)\right|\leq t^{\nu(1-2\varepsilon_2)N} +t^{\varepsilon_2N}.
      \end{equation}
  \item[(ii)] The profile $w^{(N)}_{\text{rem}}(r,t)$ satisfies
  \begin{align}
  &\left\|r^{-l}\partial_r^k \left(w^{(N)}_{\text{rem}}(t)-f_0\right)\right\|_{L^2(rdr,r\geq \frac{1}{10}t^{\frac{1}{2}-\varepsilon_2} )}\leq Ct^\eta,\quad 0\leq k+l\leq3, \label{E2.115}\\
  &\left\|r\partial_r w^{(N)}_{\text{rem}}(t)\right\|_{L^\infty(r\geq \frac{1}{10}t^{\frac{1}{2}-\varepsilon_2} )}\leq C\delta^{2\nu},\label{E2.116}\\
  &\left\|r^{-l}\partial_r^k w^{(N)}_{\text{rem}}(t)\right\|_{L^\infty(r\geq \frac{1}{10}t^{\frac{1}{2}-\varepsilon_2} )}\leq C\left(\delta^{2\nu-k-l} +t^{\nu-\frac{k+l}{2}+\eta}\right),\label{E2.117}\\
  &\quad \quad \quad \quad \quad \quad \quad \quad \quad \quad \quad \quad \quad \quad \quad \quad \quad\quad\quad \quad \quad \quad 0\leq k+l\leq4,\notag\\
  &\left\|r^{-l-1}\partial_r^k w^{(N)}_{\text{rem}}(t)\right\|_{L^\infty(r\geq \frac{1}{10}t^{\frac{1}{2}-\varepsilon_2} )}\leq C\left(\delta^{2\nu-6} +t^{\nu-3+\eta}\right),\quad k+l=5.\label{E2.118}
  \end{align}
  \item[(iii)] If $N$ sufficiently large, then the error $A^{(N)}_{\text{rem}}(r,t)$ satisfies
  \begin{equation}\label{E2.119}
  \left\|r^{-l}\partial_r^k\partial_t^i A^{(N)}_{\text{rem}}(t)\right\|_{L^2(rdr,r\geq \frac{1}{10}t^{\frac{1}{2}-\varepsilon_2} )}\leq t^{\varepsilon_2N},\quad 0\leq k+l\leq3,\, i=0,1.
  \end{equation}
\end{description}
\end{lemma}

\subsection{Proof of Proposition \ref{pro2.1}}
Now we start to prove Proposition \ref{pro2.1}. Fix $\varepsilon_2$ satisfying $0<\varepsilon_2<\frac{1}{2}$. For $N\geq2$, we define
\begin{align*}
\hat{W}_{\text{ex}}^{(N)}(\rho,t) =&\theta\left(t^{\nu-\varepsilon_1}\rho\right) W^{(N)}_{\text{in}}\left(t^\nu \rho,t\right)\\ &+\left[1-\theta\left(t^{\nu-\varepsilon_1}\rho\right)\right] \theta\left(t^{\nu+\varepsilon_2}\rho\right) W^{(N)}_{ss}\left(t^\nu\rho,t\right)\\
&+\left(1-\theta\left(t^{\nu+\varepsilon_2}\rho\right)\right)  w_{\text{rem}}^{(N)}\left(t^{\nu+\frac{1}{2}}\rho,t\right),\\
V_{\text{ex}}^{(N)}(\rho,t)=&\left(\frac{2\operatorname{Re}\left( \hat{W}_{\text{ex}}^{(N)}\right)}{1+\left|\hat{W}_{\text{ex}}^{(N)}\right|^2}, \frac{2\operatorname{Im}\left( \hat{W}_{\text{ex}}^{(N)}\right)}{1+\left|\hat{W}_{\text{ex}}^{(N)}\right|^2}, \frac{1-\left|\hat{W}_{\text{ex}}^{(N)}\right|^2}{1+\left|\hat{W}_{\text{ex}}^{(N)}\right|^2}\right).
\end{align*}
Then $V_{\text{ex}}^{(N)}(\rho,t)$ is well-defined for $\rho$ sufficiently large. In addition, for $\rho<t^{-\nu+\varepsilon_1}$, we take $V_{\text{ex}}^{(N)}(\rho,t)$ to be $V_{\text{in}}^{(N)}(\rho,t)$. That is, we assume
\begin{align*}
&V^{(N)}(\rho,t)=V^{(N)}_{\text{in}}(\rho,t),\quad \rho\leq\frac{1}{2}t^{-\nu+\varepsilon_1},\\
&V^{(N)}(\rho,t)=V^{(N)}_{\text{ex}}(\rho,t),\quad \rho\geq\frac{1}{2}t^{-\nu+\varepsilon_1},\\
&u^{(N)}(x,t)=e^{\left(\alpha(t)+\theta\right)R}V^{(N)}\left(\lambda(t)|x|,t\right).
\end{align*}
Thus, we obtain a $1$-equivariant $C^\infty$ profile $u^{(N)}: \mathbb{R}^2\times\mathbb{R}^*_+\rightarrow\mathbb{S}^2$. According to $\mathbf{(i)}$ in Lemma \ref{lem2.3}, $\mathbf{(ii)}$ in Lemma \ref{lem2.7} and $\mathbf{(ii)}$ in Lemma \ref{lem2.10}, we obtain that for any $N\geq2$, $u^{(N)}$ satisfies $\mathbf{(i)}$ in Proposition \ref{pro2.1}, where $\zeta_N^*$ is given by
\begin{align*}
\zeta_N^*(x)=e^{\theta R}\hat{\zeta}_N^*\left(|x|\right),
\end{align*}
here $\hat{\zeta}_N^*$ is defined by
\begin{align*}
\hat{\zeta}_N^*=\left(\frac{2\operatorname{Re}\left(f_0\right)}{1+\left|f_0\right|^2}, \frac{2\operatorname{Im}\left(f_0\right)}{1+\left|f_0\right|^2}, \frac{1-\left|f_0\right|^2}{1+\left|f_0\right|^2}\right).
\end{align*}
According to $\mathbf{(ii)}$ in Lemma \ref{lem2.3}, $\mathbf{(i)}$ in Lemma \ref{lem2.7} and $\mathbf{(i)}$ and $\mathbf{(iii)}$ in Lemma \ref{lem2.10}, we obtain that for any $N$ sufficiently large, the error $r^{(N)}=-u_t^{(N)}+\mathfrak{a}_1 u^{(N)}\times\Delta u^{(N)}- \mathfrak{a}_2u^{(N)}\times(u^{(N)}\times\Delta u^{(N)})$ satisfies
$$
\left\|r^{(N)}(t)\right\|_{H^3} +\left\|\partial_tr^{(N)}(t)\right\|_{H^1} +\left\|\langle x\rangle r^{(N)}(t)\right\|_{L^2}\leq t^{\eta N},\quad t\leq T(N,\delta),
$$
where $\eta=\eta(\nu,\varepsilon_2)>0$. Rewriting $N=\frac{N}{\eta}$, we obtain a family of approximate solutions $u^{(N)}(t)$ satisfying Proposition \ref{pro2.1}.

\section{The proof of Theorem \ref{MT}}\label{Sec3}
In this section, we prove the theorem \ref{MT} using the compactness method, which relies on the following auxiliary proposition.
\subsection{Auxiliary proposition}
Let $u^{(N)}$ and $T=T(N,\delta)$ be as stated in Proposition \ref{pro2.1}, we consider the following Cauchy problem:
\begin{equation}\label{E3.1}
\begin{cases}
&u_t=\mathfrak{a}_1u\times\Delta u -\mathfrak{a}_2u\times(u\times\Delta u),\quad t\geq t_1,\\
&u|_{t=t_1}=u^{(N)}(t_1),
\end{cases}
\end{equation}
where $0<t_1<T$.

We have the following result.
\begin{proposition}\label{pro3.1}
For any $N$ sufficiently large, there exists $0<t_0< T$ such that for any $t_1\in(0,t_0)$, \eqref{E3.1} exists a solution $u(t)$ satisfying
\begin{description}
  \item[(i)] $u-u^{(N)}\in C\left([t_1,t_0],H^3\right)$ and
  \begin{equation}\label{E3.2}
  \left\|u-u^{(N)}\right\|_{H^3}\leq t^{\frac{N}{2}},\quad \forall t_1\leq t\leq t_0.
  \end{equation}
  \item[(ii)] $\langle x\rangle\left(u(t)-u^{(N)}(t)\right)\in L^2$ and
  \begin{equation}\label{E3.3}
  \left\|\langle x\rangle\left(u(t)-u^{(N)}(t)\right)\right\|_{L^2}\leq t^{\frac{N}{2}},\quad \forall t_1\leq t\leq t_0.
  \end{equation}
\end{description}
\end{proposition}
\begin{proof}
Our proof is to use the bootstrap argument. Let
\begin{align*}
&u^{(N)}(x,t)=e^{\alpha(t)R}U^{(N)}\left(\lambda(t)x,t\right),\\
&r^{(N)}(x,t)=\lambda^2(t)e^{\alpha(t)R}R^{(N)}\left(\lambda(t)x,t\right),\\
&u(x,t)=e^{\alpha(t)R}U\left(\lambda(t)x,t\right),\\
&U(y,t)=U^{(N)}(y,t)+S(y,t),\\
&U^{(N)}(y,t)=\phi(y)+\chi^{(N)}(y,t).
\end{align*}
Then, $S(t)$ solves
\begin{align}\label{E3.4}
&t^{1+2\nu}S_t+\alpha_0 t^{2\nu}RS-t^{2\nu}\left(\nu+\frac{1}{2}\right)y\cdot\nabla \notag\\
=& \mathfrak{a}_1\left(S\times\Delta U^{(N)}+U^{(N)}\times\Delta S +S\times\Delta S\right)\notag\\
&-\mathfrak{a}_2U\times\left(U\times\Delta U\right) +\mathfrak{a}_2U^{(N)}\times\left(U^{(N)}\times\Delta U^{(N)}\right)+R^{N}\\
=&\mathfrak{a}_1\left(S\times\Delta U^{(N)}+U^{(N)}\times\Delta S +S\times\Delta S\right)+R^{N}\notag\\
&-\mathfrak{a}_2U\times\left(U\times\Delta S\right) -\mathfrak{a}_2S\times\left(U\times\Delta U^{(N)}\right) -\mathfrak{a}_2U^{(N)}\times\left(S\times\Delta U^{(N)}\right).\notag
\end{align}

Suppose that
\begin{equation}\label{E3.5}
\|S\|_{L^\infty\left(\mathbb{R}^2\right)}\leq\delta_1,
\end{equation}
where $\delta_1$ sufficiently small. Note that $S$ is 1-equivariant and satisfies
\begin{equation}\label{E3.6}
2(\phi,S)+2(\chi^{(N)},S)+|S|^2=0,
\end{equation}
where $\left\|\chi^{(N)}\right\|_{L^\infty\left(\mathbb{R}^2\right)}\leq C\delta^{2\nu}$ (see \eqref{E2.5}). Thus, the bootstrap hypothesis \eqref{E3.5} implies
\begin{equation}\label{E3.7}
\left\|S\right\|_{L^\infty\left(\mathbb{R}^2\right)}\leq C\left\|\nabla S\right\|_{L^2\left(\mathbb{R}^2\right)}.
\end{equation}

\subsubsection{Energy control}
We first derive the bootstrap control for the following energy norm:
$$
J_1(t)=\int_{\mathbb{R}^2}\left(|\nabla S|^2+\kappa(\rho)|S|^2\right)dy,\quad \rho=|y|.
$$
By \eqref{E3.4}, we get
\begin{align}\label{E3.8}
&t^{1+2\nu}\frac{d}{dt}\int|\nabla S|^2dy\notag\\
=&-2 \mathfrak{a}_1\int(S\times\Delta U^{(N)},\Delta S)dy +2\int(\nabla R^{(N)},\nabla S)dy\notag\\
&+2 \mathfrak{a}_2\int\left(U\times\left(U\times\Delta S\right),\Delta S\right)dy +2 \mathfrak{a}_2\int\left(S\times\left(U\times\Delta U^{(N)}\right),\Delta S\right)dy\\
&+2 \mathfrak{a}_2\int\left(U^{(N)}\times\left(S\times\Delta U^{(N)}\right),\Delta S\right)dy,\notag
\end{align}
and
\begin{align}\label{E3.9}
&t^{1+2\nu}\frac{d}{dt}\int\kappa(\rho)|S|^2dy\notag\\
=&-\left(\frac{1}{2}+\nu\right)t^{2\nu} \int\left(2\kappa+\rho\kappa'\right)(S,S)dy\notag\\
&+2\mathfrak{a}_1\int\kappa\left(U^{(N)}\times\Delta S,S\right)dy +2\int\kappa\left(R^{N},S\right)dy\notag\\
&-2 \mathfrak{a}_2\int\kappa\left(U\times\left(U\times\Delta S\right),\Delta S\right)dy\\
&-2 \mathfrak{a}_2\int\kappa\left(S\times\left(U\times\Delta U^{(N)}\right),\Delta S\right)dy\notag\\
&-2 \mathfrak{a}_2\int\kappa\left(U^{(N)}\times\left(S\times\Delta U^{(N)}\right),\Delta S\right)dy.\notag
\end{align}
Since $U^{(N)}=\phi+\chi^{(N)}$, where $\phi$ satisfies $\Delta\phi=\kappa\phi$, we get
$$
\left(S\times\Delta\phi,\Delta S\right)-\kappa\left(\phi\times\Delta S,S\right)=0.
$$
This combined with \eqref{E3.8} and \eqref{E3.9} gives
$$
t^{1+2\nu}\frac{d}{dt}J_1(t) =\sum_{i=1}^{10}\mathcal{E}_i,
$$
where
\begin{align*}
\mathcal{E}_1=&-2\mathfrak{a}_1\int\left(S\times\Delta\chi^{(N)},\Delta S\right)dy,\\
\mathcal{E}_2=&2\mathfrak{a}_1\int\kappa\left(\chi^{(N)}\times \Delta S,S\right)dy,\\
\mathcal{E}_3=&-\left(\frac{1}{2}+\nu\right)t^{2\nu}\int \left(2\kappa+\rho\kappa'\right) (S,S)dy,\\
\mathcal{E}_4=&2\int\left[\left(\nabla R^{(N)},\nabla S\right) +\kappa\left(R^{(N)},S\right)\right]dy\\
\mathcal{E}_5=&2 \mathfrak{a}_2\int\left(U\times\left(U\times\Delta S\right),\Delta S\right)dy =-2 \mathfrak{a}_2\int\left|\left(U\times\Delta S\right)\right|^2dy,\\
\mathcal{E}_6=&2 \mathfrak{a}_2\int\left(S\times\left(U\times\Delta U^{(N)}\right),\Delta S\right)dy =-2 \mathfrak{a}_2\int\left(U\times\Delta U^{(N)},S\times\Delta S\right)dy,\\
\mathcal{E}_7=&2 \mathfrak{a}_2\int\left(U^{(N)}\times\left(S\times\Delta U^{(N)}\right),\Delta S\right)dy=-2 \mathfrak{a}_2\int\left(S\times\Delta U^{(N)},U^{(N)}\times\Delta S\right)dy,\\
\mathcal{E}_8=&-2 \mathfrak{a}_2\int\kappa\left(U\times\left(U\times\Delta S\right),\Delta S\right)dy =2 \mathfrak{a}_2\int\kappa\left|\left(U\times\Delta S\right)\right|^2dy,\\
\mathcal{E}_9=&-2 \mathfrak{a}_2\int\kappa\left(S\times\left(U\times\Delta U^{(N)}\right),\Delta S\right)dy=2 \mathfrak{a}_2\int\kappa\left(U\times\Delta U^{(N)},S\times\Delta S\right)dy,\\
\mathcal{E}_{10}=&-2 \mathfrak{a}_2\int\kappa\left(U^{(N)}\times\left(S\times\Delta U^{(N)}\right),\Delta S\right)dy =2 \mathfrak{a}_2\int\kappa\left(S\times\Delta U^{(N)},U^{(N)}\times\Delta S\right)dy,
\end{align*}
By Proposition \ref{pro2.1}, we obtain
\begin{align}\label{E1-8}
&\left|\mathcal{E}_j\right|\leq Ct^{2\nu}\left\|S\right\|_{H^1}^2,\quad j=1,2,3,\notag\\
&\left|\mathcal{E}_4\right|\leq Ct^{N+\nu+\frac{1}{2}}\left\|\nabla S\right\|_{L^2},\\
&\mathcal{E}_5+\mathcal{E}_8\leq0.\notag
\end{align}
For $\mathcal{E}_i$, $i=6,7,9,10$, we decompose $U^{(N)}$ and $S$ in the basis $\{f_1,f_2,Q\}$:
\begin{equation}\small\label{decUS}
\begin{split}
&U^{(N)}(y,t)=e^{\theta R}\left[\left(1+z_3^{(N)}(\rho,t)\right)Q(\rho) +z_1^{(N)}(\rho,t)f_1(\rho)+z_2^{(N)}(\rho,t)f_2(\rho)\right],\\
&S(y,t)=e^{\theta R}\left[\zeta_3(\rho,t)Q(\rho) +\zeta_1(\rho,t)f_1(\rho)+\zeta_2(\rho,t)f_2(\rho)\right].
\end{split}
\end{equation}
Thus,
\begin{align*}
\mathcal{E}_6=&-2 \mathfrak{a}_2\int\left(U\times\Delta U^{(N)},S\times\Delta S\right)dy\\
=&-2 \mathfrak{a}_2\int\left(\Upsilon_1\times\Upsilon_2, \zeta\times\Upsilon_3\right)\rho d\rho,
\end{align*}
where
\begin{align*}
&z^{(N)}=\left(z_1^{(N)},z_2^{(N)},z_3^{(N)}\right), \quad\zeta=\left(\zeta_1,\zeta_2,\zeta_3\right),\\
&\Upsilon_1=\mathbf{k}+z^{(N)}+\zeta,\quad |\Upsilon_1|=1,\\
&\Upsilon_2=\Delta z^{(N)}+\Upsilon_{21},\\
&\Upsilon_{21}=\left(-\frac{z_1^{(N)}}{\rho^2}-2\frac{h_1}{\rho}\partial_\rho z_3^{(N)}, -\frac{z_2^{(N)}}{\rho^2},\kappa\left(1+z_3^{(N)}\right) +2\frac{h_1}{\rho}\partial_\rho z_1^{(N)} -2\frac{h_1h_3}{\rho^2}z_1^{(N)}\right),\\
&\Upsilon_3=\Delta \zeta+\Upsilon_{31},\\
&\Upsilon_{31}=\left(-\frac{\zeta_1}{\rho^2}-2\frac{h_1}{\rho}\partial_\rho \zeta_3, -\frac{\zeta_2}{\rho^2}, \kappa\zeta_3 +2\frac{h_1}{\rho}\partial_\rho \zeta_1 -2\frac{h_1h_3}{\rho^2}\zeta_1\right),
\end{align*}
which satisfy
\begin{equation}\label{EofUpsilon}
\begin{split}
&\left|\partial_{\rho}\Upsilon_1\right|\leq C\left(\left|\partial_{\rho}z^{(N)}\right| +\left|\partial_{\rho}\zeta\right|\right),\\
&\left|\Upsilon_{21}\right|\leq C\frac{1}{\rho^2}\left(\left|\partial_{\rho}z^{(N)}\right| +\left|z^{(N)}\right|\right),\\
&\left|\partial_{\rho}\Upsilon_{2}\right|\leq C\left(\left|\partial_{\rho}^3z^{(N)}\right| +\rho^{-3}|z^{(N)}|+\frac{1}{\rho^2}|\partial_{\rho}z^{(N)}|+|z^{(N)}|\right),\\
&\left|\Upsilon_{31}\right|\leq C\frac{1}{\rho^2}\left(|\zeta| +\left|\partial_{\rho}\zeta\right|\right).
\end{split}
\end{equation}
By Proposition \ref{pro2.1}, we get
\begin{equation}\label{EE6}
\begin{split}
\left|\mathcal{E}_6\right|=& 2\mathfrak{a}_2\int\left(\partial_{\rho}\left(\Upsilon_1\times\Upsilon_2\right), \zeta\times\partial_{\rho}\zeta\right)\rho d\rho -2\mathfrak{a}_2\int\left(\Upsilon_1\times\Upsilon_2, \zeta\times\Upsilon_{31}\right)\rho d\rho\\
\leq& Ct^{2\nu}\left(\|S\|_{H^1}^2+\|S\|_{H^1}^2\|\nabla S\|_{L^2}\right).
\end{split}
\end{equation}
For $\mathcal{E}_7$, since
\begin{align*}
\mathcal{E}_7=&-2 \mathfrak{a}_2\int\left(S\times\Delta U^{(N)},U^{(N)}\times\Delta S\right)dy\\
=&-2 \mathfrak{a}_2\int\left(\zeta\times\Upsilon_2, \left(\mathbf{k}+z^{(N)}\right)\times\Upsilon_3\right)\rho d\rho,
\end{align*}
we get
\begin{equation}\label{EE7}
|\mathcal{E}_7|\leq Ct^{2\nu}\left\|S\right\|_{H^1}^2.
\end{equation}
Similarly, we have the estimates:
\begin{equation}\label{E9-10}
\begin{split}
&|\mathcal{E}_9|\leq Ct^{2\nu}\left(\|S\|_{H^1}^2+\|S\|_{H^1}^2\|\nabla S\|_{L^2}\right),\\ &|\mathcal{E}_{10}|\leq Ct^{2\nu}\left\|S\right\|_{H^1}^2.
\end{split}
\end{equation}

Combining  \eqref{E1-8}, \eqref{EE6}, \eqref{EE7} and \eqref{E9-10}, we get
\begin{equation}\label{E3.10}
\left|\frac{d}{dt}J_1(t)\right|\leq C\frac{1}{t}\left(\|S\|_{H^1}^2+\|S\|_{H^1}^2\|\nabla S\|_{L^2}\right)+Ct^{2N-2\nu}.
\end{equation}

\subsubsection{Control the $L^2$ norm}
Now we consider the control of the following norm:
$$
J_0(t)=\int_{\mathbb{R}^2}|S|^2dy.
$$
By \eqref{E3.4}, we obtain
\begin{equation*}
\begin{split}
&t^{1+2\nu}\frac{d}{dt}\int|S|^2dy\\
=&-2\left(\frac{1}{2}+\nu\right)t^{2\nu} \int(S,S)dy +2\mathfrak{a}_1\int\left(U^{(N)}\times\Delta S,S\right)dy +2\int\left(R^{N},S\right)dy\\
&-2 \mathfrak{a}_2\int\left(U\times\left(U\times\Delta S\right), S\right)dy-2 \mathfrak{a}_2\int\left(U^{(N)}\times\left(S\times\Delta U^{(N)}\right), S\right)dy.
\end{split}
\end{equation*}
That is,
\begin{align*}
t^{1+2\nu}\frac{d}{dt}J_0(t)=\sum_{i=11}^{15}\mathcal{E}_i,
\end{align*}
where
\begin{align*}
&\mathcal{E}_{11}=-(1+2\nu)t^{2\nu}J_0(t),\\
&\mathcal{E}_{12}=2\mathfrak{a}_1\int\left(U^{(N)}\times\Delta S, S\right)dy,\\
&\mathcal{E}_{13}=2\int\left(R^{(N)},S\right)dy\\
&\mathcal{E}_{14}=-2 \mathfrak{a}_2\int\left(U\times\left(U\times\Delta S\right), S\right)dy =2 \mathfrak{a}_2\int\left(U\times\Delta S, U^{(N)}\times S\right)dy\\
&\mathcal{E}_{15}=2 \mathfrak{a}_2\int\left(U^{(N)}\times\left(S\times\Delta U^{(N)}\right), S\right)dy =2 \mathfrak{a}_2\int\left(S\times\Delta U^{(N)}, S \times U^{(N)}\right)dy.
\end{align*}
For $\mathcal{E}_{12}$, we recall \eqref{decUS} and decompose $U^{(N)}$ and $S$ in the basis $\{f_1,f_2,Q\}$, then $\mathcal{E}_{12}$ can be rewritten as
$$
\mathcal{E}_{12}=\sum_{i=16}^{18}\mathcal{E}_i,
$$
where
\begin{align*}
&\mathcal{E}_{16}=-4\mathfrak{a}_1\int_{\mathbb{R}^+}\frac{h_1}{\rho}\zeta_2\partial_\rho\zeta_3\rho d\rho,\\
&\mathcal{E}_{17}= -2\mathfrak{a}_1\int_{\mathbb{R}^+}\left(\partial_\rho z^{(N)}\times\partial_\rho\zeta,\zeta\right)\rho d\rho,\\
&\mathcal{E}_{18}=2\mathfrak{a}_1\int_{\mathbb{R}^+}\left(z^{(N)}\times l,\zeta\right)\rho d\rho.
\end{align*}
Here $l$ is given by
\begin{align*}
l= \left(-\frac{1}{\rho^2}\zeta_1 -\frac{2h_1}{\rho}\partial_\rho\zeta_3,\, -\frac{1}{\rho^2}\zeta_2,\, \kappa(\rho)\zeta_3+\frac{2h_1}{\rho}\partial_\rho\zeta_1 -\frac{2h_1h_3}{\rho^2}\partial_\rho\zeta_1\right),
\end{align*}
which satisfies
$$
|l|\leq C\frac{1}{\rho^2}\left(|\zeta|+\left|\partial_\rho\zeta\right|\right).
$$
Thus,
\begin{equation}\label{E3.11}
\left|\mathcal{E}_{18}\right|\leq Ct^{2\nu}\|S\|_{H^1}^2.
\end{equation}
For $\mathcal{E}_{16}$, since
\begin{equation}\label{E3.12}
2\left(\zeta,\mathbf{k}+z^{(N)}\right) +|\zeta|^2=0,
\end{equation}
we get
$$
\left|\partial_\rho\zeta_3\right|\leq C\left(\left|\partial_\rho z^{(N)}\right||\zeta| +\left|z^{(N)}\right|\left|\partial_\rho\zeta\right| +\left|\partial_\rho\zeta\right||\zeta|\right).
$$
Thus,
\begin{equation}\label{E3.13}
\left|\mathcal{E}_{16}\right|\leq C\left(t^{2\nu}\|S\|_{H^1}^2 +\|\nabla S\|_{L^2}^3\right).
\end{equation}
For $\mathcal{E}_{17}$, let $e_0=\mathbf{k}+z^{(N)}$ and $\zeta=\zeta^\bot+\mu e_0$, where  $\mu=(\zeta,e_0)$. By \eqref{E3.12}, we get
$$
|\mu|\leq C|\zeta|^2,\quad \left|\mu_\rho\right|\leq C|\zeta|\left|\partial_\rho\zeta\right|.
$$
Thus, $\mathcal{E}_{17}$ can be rewritten as
\begin{equation}\label{E3.14}
\mathcal{E}_{17}=-2\mathfrak{a}_1\int_{\mathbb{R}_+} \left(\partial_\rho\zeta^\bot\times\zeta^\bot, \partial_\rho e_0\right)\rho d\rho +O\left(\|S\|_{H^1}^2\|\nabla S\|_{L^2}\right).
\end{equation}
Let $\{e_1, e_2\}$ be a set of smooth orthogonal basis of the tangent space $T_{e_0}\mathbb{S}^2$ that satisfies $e_2=e_0\times e_1$, then $\left(\partial_\rho\zeta^\bot\times\zeta^\bot, \partial_\rho e_0\right)$ can be rewritten as
$$
\left(\partial_\rho\zeta^\bot\times\zeta^\bot, \partial_\rho e_0\right) =\left(\zeta^\bot, \partial_\rho e_0\right)\left[\left(\zeta^\bot, e_2\right)\left(\partial_\rho e_0, e_1\right) -\left(\zeta^\bot, e_1\right)\left(\partial_\rho e_0, e_2\right)\right],
$$
which gives
\begin{equation}\label{E3.15}
\left|\int_{\mathbb{R}_+} \left(\partial_\rho\zeta^\bot\times\zeta^\bot, \partial_\rho e_0\right)\rho d\rho\right|\leq \left\|\partial_\rho z^{(N)}\right\|_{L^\infty}^2 J_0(t)\leq t^{2\nu}J_0(t).
\end{equation}
Combining \eqref{E3.11}, \eqref{E3.13}, \eqref{E3.14} and \eqref{E3.15}, we get
\begin{equation}\label{EE12}
|\mathcal{E}_{12}|\leq C\left(t^{2\nu}\|S\|_{H^1}^2 +\|\nabla S\|_{L^2}^3\right)+2|\mathfrak{a}_1|t^{2\nu}J_0(t) +O\left(\|S\|_{H^1}^2\|\nabla S\|_{L^2}\right).
\end{equation}
For $\mathcal{E}_{13}$, by Proposition \ref{pro2.1}, we get
\begin{equation}\label{EE13}
|\mathcal{E}_{13}|\leq Ct^{N+\nu+\frac{1}{2}}\|\nabla S\|_{L^2}.
\end{equation}
For $\mathcal{E}_{14}$, combining
\begin{align*}
\mathcal{E}_{14}=&2 \mathfrak{a}_2\int\left(U\times\Delta S, U^{(N)}\times S\right)dy\\
=&2 \mathfrak{a}_2\int\left(\Upsilon_1\times\Upsilon_3, \left(\mathbf{k}+z^{(N)}\right)\times \zeta\right)\rho d\rho\\
=&-2 \mathfrak{a}_2\int\left(\partial_{\rho} z^{(N)}\times\partial_{\rho}\zeta, \left(\mathbf{k}+z^{(N)}\right)\times \zeta\right)\rho d\rho\\
&-2 \mathfrak{a}_2\int\left(\Upsilon_1\times\partial_{\rho}\zeta, \partial_{\rho}\left[\left(\mathbf{k}+z^{(N)}\right)\times \zeta\right]\right)\rho d\rho\\
&+2 \mathfrak{a}_2\int\left(\Upsilon_1\times\Upsilon_{31}, \left(\mathbf{k}+z^{(N)}\right)\times \zeta\right)\rho d\rho,
\end{align*}
and \eqref{EofUpsilon}, we get
\begin{equation}\label{EE14}
|\mathcal{E}_{14}|\leq Ct^{2\nu}\left(\|S\|_{H^1}^2+\|S\|_{H^1}^2\|\nabla S\|_{L^2}\right).
\end{equation}
For $\mathcal{E}_{15}$, similarly, by
\begin{align*}
\mathcal{E}_{15}=&2 \mathfrak{a}_2\int\left(S\times\Delta U^{(N)}, S \times U^{(N)}\right)dy\\
=&2 \mathfrak{a}_2\int\left(\zeta\times\Upsilon_2, \zeta \times \left(\mathbf{k}+z^{(N)}\right)\right)\rho d\rho\\
=&-2 \mathfrak{a}_2\int\left(\partial_{\rho}\zeta\times\partial_{\rho} z^{(N)}, \zeta \times \left(\mathbf{k}+z^{(N)}\right)\right)\rho d\rho\\
&-2 \mathfrak{a}_2\int\left(\zeta\times\partial_{\rho} z^{(N)}, \partial_{\rho}\left[\zeta \times \left(\mathbf{k}+z^{(N)}\right)\right]\right)\rho d\rho\\
&+2 \mathfrak{a}_2\int\left(\zeta\times\Upsilon_{21}, \zeta \times \left(\mathbf{k}+z^{(N)}\right)\right)\rho d\rho,
\end{align*}
we can obtain
\begin{equation}\label{EE15}
|\mathcal{E}_{15}|\leq C t^{2\nu}\|S\|_{H^1}^2 +2|\mathfrak{a}_2|t^{2\nu}J_0(t).
\end{equation}

Combining \eqref{EE12}, \eqref{EE13}, \eqref{EE14} and \eqref{EE15}, we get
\begin{equation}\label{E3.16}
\left|\frac{d}{dt}J_0(t)\right|\leq C\left(\frac{1}{t}\|S\|_{H^1}^2 +t^{-1-2\nu}\|S\|_{H^1}^2\|\nabla S\|_{L^2} +t^{2N-2\nu}\right).
\end{equation}

\subsubsection{Control of the weighted $L^2$ norm}
By \eqref{E3.4}, we can calculate $\frac{d}{dt}\|yS(t)\|_{L^2}^2$ by
\begin{align*}
&t^{1+2\nu}\frac{d}{dt}\left\||y|S(t)\right\|_{L^2}^2\\
=&-4\mathfrak{a}_1\int y_i\left(U^{(N)}\times\partial_i S, S\right)dy -2\mathfrak{a}_1\int|y|^2\left(\partial_i U^{(N)}\times\partial_i S,S\right)dy\\
&-2(1+2\nu)t^{2\nu}\left\||y|S(t)\right\|_{L^2}^2 +2\int|y|^2\left(R^{(N)},S\right)dy\\
&-2\mathfrak{a}_2\int\left(U\times\left(U\times\Delta S\right),|y|^2S\right)dy\\
&-2\mathfrak{a}_2\int\left(U^{(N)}\times\left(S\times\Delta U^{(N)}\right),|y|^2S\right)dy\\
=&-4\mathfrak{a}_1\int y_i\left(U^{(N)}\times\partial_i S, S\right)dy -2\mathfrak{a}_1\int|y|^2\left(\partial_i U^{(N)}\times\partial_i S,S\right)dy\\
&-2(1+2\nu)t^{2\nu}\left\||y|S(t)\right\|_{L^2}^2 +2\int|y|^2\left(R^{(N)},S\right)dy\\
&-2\mathfrak{a}_2\int|y|^2\left(\partial_iU^{(N)}\times\partial_i S,U^{(N)}\times S\right)dy -2\mathfrak{a}_2\int|y|^2\left(U^{(N)}\times\partial_i S,\partial_iU^{(N)}\times S\right)dy\\
&-4\mathfrak{a}_2\int y_i\left(U^{(N)}\times\partial_i S,U^{(N)}\times S\right)dy -2\mathfrak{a}_2\int|y|^2\left|U^{(N)}\times\partial_i S\right|^2dy\\
&-2\mathfrak{a}_2\int|y|^2\left(\partial_i S\times\partial_iU^{(N)},U^{(N)}\times S\right)dy +2\mathfrak{a}_2\int|y|^2\left|S\times\partial_iU^{(N)}\right|^2dy\\
&+4\mathfrak{a}_2\int y_i\left|S \times\partial_i U^{(N)}\right|^2dy -2\mathfrak{a}_2\int |y|^2\left(S \times\partial_i U^{(N)}, \partial_i U^{(N)}\times \partial_i S\right)dy,
\end{align*}
where $\partial_j$ denotes $\partial_{y_j}$. Here and below we use the convention of implicit summation over repeated indices.

Thus, we obtain
\begin{equation}\label{E3.17}
\left|\frac{d}{dt}\left\||y|S(t)\right\|_{L^2}^2\right|\leq C\left(\left\||y|S(t)\right\|_{L^2}^2+ t^{-4\nu}\|S\|_{H^1}^2 +t^{2N-4\nu}\right).
\end{equation}

\subsubsection{Control of the higher order derivatives}
In addition to the assumption \eqref{E3.5}, we also assume
\begin{equation}\label{E3.18}
\|S(t)\|_{H^3}+\left\||y|S(t)\right\|_{L^2}^2\leq t^{\frac{2N}{5}}.
\end{equation}
We next obtain the control of the $\dot{H}^3$ norm of solutions by estimating $\|\nabla S_t\|_{L^2}$. More precisely, we consider the functional
$$
J_3(t)=t^{2+4\nu}\int\left|\nabla s_t(x,t)\right|^2dx +t^{1+2\nu}\int\kappa\left(t^{-\frac{1}{2}-\nu}x\right)\cdot\left|s_t(x,t)\right|^2dx,
$$
where $s(x,t)$ is defined by
$$
s(x,t)=e^{\alpha(t)R}S\left(\lambda(t)x,t\right).
$$
Let $s_t(x,t)=e^{\alpha(t)R}\lambda^2(t)g(\lambda(t)x,t)$, then $J_3$ can be expressed by a functional of $g$:
$$
J_3(t)=\int\left|\nabla g(y,t)\right|^2dy +\int\kappa(\rho)|g(y,t)|^2dy.
$$
Now we calculate $\frac{d}{dt}J_3(t)$. Note that $g(y,t)$ satisfies
\begin{align}\label{E3.19}
&t^{1+2\nu}g_t+\alpha_0t^{2\nu}Rg -\left(\nu+\frac{1}{2}\right)t^{2\nu}\left(2+y\cdot\nabla\right)g\notag\\
=&\mathfrak{a}_1\left(S+U^{(N)}\right)\times\Delta g +\mathfrak{a}_1g\times\left(\Delta U^{N}+\Delta S\right)\notag\\
&+\mathfrak{a}_1\left(U^{(N)}\times\Delta U^{(N)}-R^{(N)}\right)\times\Delta S\notag\\
&+\mathfrak{a}_1S\times\Delta\left(U^{(N)}\times\Delta U^{(N)}-R^{(N)}\right)\notag\\
&-\mathfrak{a}_2\left(U^{(N)}\times\Delta U^{(N)}-R^{(N)}\right)\times \left[\left(S+U^{(N)}\right)\times\Delta S \right]\notag\\
&-\mathfrak{a}_2g\times \left[\left(S+U^{(N)}\right)\times\Delta S \right] -\mathfrak{a}_2\left(S+U^{(N)}\right)\times \left(g\times\Delta S\right)\notag\\
&-\mathfrak{a}_2\left(S+U^{(N)}\right)\times \left[\left(U^{(N)}\times\Delta U^{(N)}-R^{(N)}\right)\times\Delta S\right]\notag\\
&-\mathfrak{a}_2\left(S+U^{(N)}\right)\times\left[\left(S+U^{(N)}\right)\times\Delta g\right]\\ &-\mathfrak{a}_2g\times\left[\left(S+U^{(N)}\right)\times\Delta U^{(N)}\right]\notag\\
&-\mathfrak{a}_2S\times\left[\left(U^{(N)}\times\Delta U^{(N)}-R^{(N)}\right)\times\Delta U^{(N)}\right] -\mathfrak{a}_2S\times\left(g\times\Delta U^{(N)}\right)\notag\\
&-\mathfrak{a}_2S\times\left[\left(S+U^{(N)}\right)\times\Delta\left(U^{(N)}\times\Delta U^{(N)}-R^{(N)}\right)\right]\notag\\
&-\mathfrak{a}_2\left(U^{(N)}\times\Delta U^{(N)}-R^{(N)}\right)\times\left(S+\Delta U^{(N)}\right)\notag\\ &-\mathfrak{a}_2U^{(N)}\times \left(g\times\Delta U^{(N)}\right)\notag\\
&-\mathfrak{a}_2U^{(N)}\times\left[S\times\Delta\left(U^{(N)}\times\Delta U^{(N)}-R^{(N)}\right)\right] +t^{2+4\nu}r_t^{(N)}.\notag
\end{align}
Thus,
\begin{align}\label{E3.20}
t^{1+2\nu}\frac{d}{dt}J_3(t) =&(2+4\nu)t^{2\nu}\|\nabla g\|_{L^2}^2\\ &+\left(\frac{1}{2}+\nu\right)t^{2\nu}\int\left(2\kappa-\rho\kappa'\right) |g|^2dy +\sum_{i=16}^{32}\mathcal{E}_i,\notag
\end{align}
where
\begin{align*}
\mathcal{E}_{16}=&-2\mathfrak{a}_1\int\left(g\times\Delta\chi^{(N)},\Delta g\right)dy +2\mathfrak{a}_1\int\kappa\left(\chi^{(N)}\times\Delta g,g\right)dy,\\
\mathcal{E}_{17}=&-2\mathfrak{a}_1\int\left(\left(U^{(N)}\times\Delta U^{(N)} -R^{(N)}\right) \times\Delta S, \Delta g\right)dy\\
&+2\mathfrak{a}_1\int\left(\Delta\left(U^{(N)}\times\Delta U^{(N)}-R^{(N)}\right)\times S,\Delta g\right)dy\\
&+2\mathfrak{a}_1\int\kappa\left(\left(U^{(N)}\times\Delta U^{(N)}-R^{(N)}\right)\times\Delta S, g\right)dy\\
&-2\mathfrak{a}_1\int\kappa\left(\Delta\left(U^{(N)}\times\Delta U^{(N)}-R^{(N)}\right)\times S, g\right)dy,\\
\mathcal{E}_{18}=&-2\mathfrak{a}_1\int\left(g\times\Delta S,\Delta g\right)dy,\\
\mathcal{E}_{19}=&2\mathfrak{a}_1\int\kappa\left(S\times\Delta g, g\right)dy,\\
\mathcal{E}_{20}=&-2t^{2+4\nu}\int(r_t,\Delta g)dy +2t^{2+4\nu}\int\kappa(r_t,g)dy,\\
\mathcal{E}_{21}=&2\mathfrak{a}_2\int\left(\left(U^{(N)}\times\Delta U^{(N)}-R^{(N)}\right)\times \left[\left(S+U^{(N)}\right)\times\Delta S \right], \Delta g\right)dy\\
&-2\mathfrak{a}_2\int\kappa\left(\left(U^{(N)}\times\Delta U^{(N)}-R^{(N)}\right)\times \left[\left(S+U^{(N)}\right)\times\Delta S \right], g\right)dy,\\
\mathcal{E}_{22}=&2\mathfrak{a}_2\int\left(g\times \left[\left(S+U^{(N)}\right)\times\Delta S \right], \Delta g\right)dy,\\
\mathcal{E}_{23}=&2\mathfrak{a}_2\int\left(\left(S+U^{(N)}\right)\times \left(g\times\Delta S\right), \Delta g\right)dy\\
&-2\mathfrak{a}_2\int\kappa\left(\left(S+U^{(N)}\right)\times \left(g\times\Delta S\right), g\right)dy,\\
\mathcal{E}_{24}=&2\mathfrak{a}_2\int\left(\left(S+U^{(N)}\right)\times \left[\left(U^{(N)}\times\Delta U^{(N)}-R^{(N)}\right)\times\Delta S\right], \Delta g\right)dy\\
&-2\mathfrak{a}_2\int\kappa\left(\left(S+U^{(N)}\right)\times \left[\left(U^{(N)}\times\Delta U^{(N)}-R^{(N)}\right)\times\Delta S\right], g\right)dy,\\
\mathcal{E}_{25}=&2\mathfrak{a}_2\int\left(\left(S+U^{(N)}\right)\times\left[\left(S+U^{(N)}\right) \times\Delta g\right], \Delta g\right)dy\\
&-2\mathfrak{a}_2\int\kappa\left(\left(S+U^{(N)}\right)\times\left[\left(S+U^{(N)}\right)\times\Delta g\right], g\right)dy,\\
\mathcal{E}_{26}=&2\mathfrak{a}_2\int\left(g\times\left[\left(S+U^{(N)}\right)\times\Delta U^{(N)}\right], \Delta g\right)dy,\\
\mathcal{E}_{27}=&2\mathfrak{a}_2\int\left(S\times\left[\left(U^{(N)}\times\Delta U^{(N)}-R^{(N)}\right)\times\Delta U^{(N)}\right], \Delta g\right)dy\\
&-2\mathfrak{a}_2\int\kappa\left(S\times\left[\left(U^{(N)}\times\Delta U^{(N)}-R^{(N)}\right)\times\Delta U^{(N)}\right], g\right)dy,\\
\mathcal{E}_{28}=&2\mathfrak{a}_2\int\left(S\times\left(g\times\Delta U^{(N)}\right), \Delta g\right)dy -2\mathfrak{a}_2\int\kappa\left(S\times\left(g\times\Delta U^{(N)}\right), g\right)dy,\\
\mathcal{E}_{29}=&2\mathfrak{a}_2\int\left(S\times\left[\left(S+U^{(N)}\right) \times\Delta\left(U^{(N)}\times\Delta U^{(N)}-R^{(N)}\right)\right], \Delta g\right)dy\\
&-2\mathfrak{a}_2\int\kappa\left(S\times\left[\left(S+U^{(N)}\right) \times\Delta\left(U^{(N)}\times\Delta U^{(N)}-R^{(N)}\right)\right], g\right)dy,\\
\mathcal{E}_{30}=&2\mathfrak{a}_2\int\left(\left(U^{(N)}\times\Delta U^{(N)}-R^{(N)}\right)\times\left(S+\Delta U^{(N)}\right), \Delta g\right)dy\\
&-2\mathfrak{a}_2\int\kappa\left(\left(U^{(N)}\times\Delta U^{(N)}-R^{(N)}\right)\times\left(S+\Delta U^{(N)}\right), g\right)dy,\\
\mathcal{E}_{31}=&2\mathfrak{a}_2\int\left(U^{(N)}\times \left(g\times\Delta U^{(N)}\right), \Delta g\right)dy\\
&-2\mathfrak{a}_2\int\kappa\left(U^{(N)}\times \left(g\times\Delta U^{(N)}\right), g\right)dy,\\
\mathcal{E}_{32}=&2\mathfrak{a}_2\int\left(U^{(N)}\times\left[S\times\Delta\left(U^{(N)}\times\Delta U^{(N)}-R^{(N)}\right)\right], \Delta g\right)dy\\
&-2\mathfrak{a}_2\int\kappa\left(U^{(N)}\times\left[S\times\Delta\left(U^{(N)}\times\Delta U^{(N)}-R^{(N)}\right)\right], g\right)dy.
\end{align*}

For $\mathcal{E}_j$, $j=16,19,20$, if $N$ sufficiently large, then there exists $t_0=t_0(N)>0$ such that for $t\leq t_0$, the following estimates hold:
\begin{equation}\label{E3.21}
\begin{split}
\left|\mathcal{E}_{16}\right|\leq &Ct^{2\nu}\|g\|_{H^1}^2,\\
\left|\mathcal{E}_{19}\right|\leq &C\|g\|_{H^1}^2\|S\|_{H^3}\leq Ct^{2\nu}\|g\|_{H^1}^2,\\
\left|\mathcal{E}_{20}\right|\leq &C\left(t^{2\nu}\|g\|_{H^1}^2+t^{2N+3+4\nu}\right).
\end{split}
\end{equation}

For $\mathcal{E}_{17}$, we have
\begin{align*}
\left|\mathcal{E}_{17}\right|\leq &C\left(\left\|\Delta\chi^{(N)}\right\|_{W^{2,\infty}} +\left\|R^{(N)}\right\|_{H^3}\right) \|g\|_{H^1}\|S\|_{H^3}\\
&+C\left\|\langle y\rangle^{-1}\nabla\Delta^2\chi^{(N)}\right\|_{L^\infty} \|\nabla g\|_{L^2} \|\langle y\rangle S\|_{L^2}.
\end{align*}
Thus,
\begin{equation}\label{E3.22}
\left|\mathcal{E}_{17}\right|\leq Ct^{2\nu}\left(\|g\|_{H^1}\|S\|_{H^3} +\|\nabla g\|_{L^2} \|\langle y\rangle S\|_{L^2}\right).
\end{equation}
Note that
\begin{equation}\label{E3.23}
g=\left(U^{(N)}+S\right)\times\Delta S+ S\times\Delta U^{(N)}+R^{(N)},
\end{equation}
and by the bootstrap hypothesis \eqref{E3.18}, we get
\begin{equation}\label{E3.24}
\begin{split}
&\|g\|_{L^2}\leq C\left(\|S\|_{H^2} +\left\|R^{(N)}\right\|_{L^2}\right),\\
&\left\|\nabla g\right\|_{L^2}\leq C\left(\|S\|_{H^3} +\left\|\nabla R^{(N)}\right\|_{L^2}\right).
\end{split}
\end{equation}
Therefore, combining \eqref{E3.21} and \eqref{E3.22}, we get
\begin{equation}\label{E3.25}
\begin{split}
&\left|\mathcal{E}_{16}\right|+\left|\mathcal{E}_{17}\right| +\left|\mathcal{E}_{19}\right|+\left|\mathcal{E}_{20}\right|\\
\leq& Ct^{2\nu}\left[\|S\|_{H^3}^2 +\left(\|S\|_{H^3}+t^{N+1+2\nu}\right)\left\|\langle y\rangle S\right\|_{L^2}\right] +Ct^{2N+1+4\nu}.
\end{split}
\end{equation}

For $\mathcal{E}_{18}$, note that
\begin{equation}\label{E3.26}
\begin{split}
g\times\Delta S=&\left(U^{(N)}+S,\Delta S\right)\Delta S-|\Delta S|^2\left(U^{(N)}+S\right)\\ &\quad\quad\quad \quad\quad\quad \quad\quad\quad+\left(S\times\Delta U^{(N)}+R^{(N)}\right)\times\Delta S,\\
\Delta g=&\left(U^{(N)}+S\right)\times\Delta^2S+Y,
\end{split}
\end{equation}
where
\begin{align*}
Y=2\left(\partial_j U^{(N)}+\partial_j S\right)\times\Delta\partial_jS +S\times\Delta^2 U^{(N)}+2\partial_jS\times\Delta\partial_jU^{(N)}+\Delta R^{(N)}.
\end{align*}
Thus, $\mathcal{E}_{18}$ can be rewritten as
$$
\mathcal{E}_{18}=\mathcal{E}_{18,1}+\mathcal{E}_{18,2}+\mathcal{E}_{18,3},
$$
where
\begin{align*}
&\mathcal{E}_{18,1}=-2\mathfrak{a}_1\int\left(U^{(N)}+S,\Delta S\right)\left(S,\Delta g\right)dy,\\
&\mathcal{E}_{18,2}=2\mathfrak{a}_1\int|\Delta S|^2\left(U^{N}+S,\Delta g\right)dy =2\mathfrak{a}_1\int|\Delta S|^2\left(U^{(N)}+S,Y\right)dy,\\
&\mathcal{E}_{18,3}=-2\mathfrak{a}_1\int\left(\left(S\times\Delta U^{(N)}+R^{(N)}\right)\times\Delta S, \Delta g\right)dy.
\end{align*}
For $\mathcal{E}_{18,1}$, note that
$$
-\left(U^{(N)}+S,\Delta S\right)=\left(\Delta U^{(N)}, S\right) +2\left(\partial_j U^{(N)},\partial_j S\right) +\left(\partial_jS,\partial_j S\right),
$$
thus,
\begin{align*}
\mathcal{E}_{18,1}=&-2\mathfrak{a}_1\int\left[\left(\Delta U^{(N)}, S\right) +2\left(\partial_j U^{(N)},\partial_j S\right) +\left(\partial_jS,\partial_j S\right)\right]\left(\Delta \partial_kS,\partial_k g\right)dy\\
&-2\mathfrak{a}_1\int\partial_k\left[\left(\Delta U^{(N)}, S\right) +2\left(\partial_j U^{(N)},\partial_j S\right) +\left(\partial_jS,\partial_j S\right)\right]\left(\Delta S,\partial_k g\right)dy,
\end{align*}
which gives
\begin{equation}\label{E3.27}
\left|\mathcal{E}_{18,1}\right|\leq C\|S\|_{H^3}^2\|g\|_{H^1}\leq Ct^{2\nu}\|S\|_{H^3}^2.
\end{equation}
For $\mathcal{E}_{18,2}$, by \eqref{E3.26}, we get
$$
\|Y\|_{L^2}\leq C\left(\|S\|_{H^3}+t^N\right).
$$
Thus,
\begin{equation}\label{E3.28}
\left|\mathcal{E}_{18,2}\right|\leq Ct^{2\nu}\|S\|_{H^3}^2.
\end{equation}
Finally, $\mathcal{E}_{18,3}$ satisfies
\begin{equation}\label{E3.29}
\left|\mathcal{E}_{18,3}\right|\leq C\|g\|_{H^1}\left(\|S\|_{H^3}^2+t^N\|S\|_{H^3}\right) \leq Ct^{2\nu}\|S\|_{H^3}^2 +Ct^{3N}.
\end{equation}
Combining \eqref{E3.27}, \eqref{E3.28} and \eqref{E3.29}, we get
\begin{equation}\label{E3.30}
\left|\mathcal{E}_{18}\right|\leq C\left(t^{2\nu}\|S\|_{H^3}^2 +t^{3N}\right),
\end{equation}

For $\mathcal{E}_{21}$, note that
$$
\left|\left(S+U^{(N)}\right)\times\Delta S\right|^2 =\left|\Delta S\right|^2- \left|\left(S+U^{(N)},\Delta S\right)\right|^2,
$$
\begin{align*}
&\left|\partial_j\left[\left(S+U^{(N)}\right)\times\Delta S\right]\right|^2\\ =&\left|\partial_jS+\partial_jU^{(N)}\right|^2\left|\Delta S\right|^2 +\left|\Delta\partial_j S\right|^2-\left|\left(\partial_jS+\partial_jU^{(N)},\Delta S\right)\right|^2 \\
& -\left|\left(S+U^{(N)},\Delta\partial_j S\right)\right|^2 -2\left(\partial_jS+\partial_jU^{(N)}, \Delta\partial_j S\right)\left(S+U^{(N)}, \Delta S\right),
\end{align*}
and
\begin{equation}\label{SUDSa}
\begin{split}
\left(S+U^{(N)},\Delta S\right) =&-\left(S+U^{(N)},\Delta U^{(N)}\right)- \left|\partial_j U^{(N)}+\partial_jS\right|^2,\\
\left(S+U^{(N)},\Delta \partial_jS\right) =&-\left(\Delta S+\Delta U^{(N)},\partial_jS\right)- \Delta\left(\partial_jU^{(N)},S\right)\\
&-2\left(\partial_kU^{(N)}+\partial_kS,\partial_{jk}^2S\right).
\end{split}
\end{equation}
This combined with \eqref{E3.18} gives
\begin{equation}\label{SUDS}
\begin{split}
&\left\|\left(S+U^{(N)}\right)\times\Delta S\right\|_{H^1} \leq C\|S\|_{H^2},\\
&\left\|\partial_j\left[\left(S+U^{(N)}\right)\times\Delta S\right]\right\|_{H^1} \leq C\|S\|_{H^2}.
\end{split}
\end{equation}
Thus,
\begin{equation}\label{EEE21}
\left|\mathcal{E}_{21}\right|\leq Ct^{2\nu}\|g\|_{H^1}\|S\|_{H^3}.
\end{equation}

For $\mathcal{E}_{22}$, since
\begin{equation}
\begin{split}
g\times\left[\left(U^{(N)}+S\right)\times\Delta S\right]=\left(S\times\Delta U^{(N)}+R^{(N)}\right) \times\left[\left(U^{(N)}+S\right)\times\Delta S\right],
\end{split}
\end{equation}
we obtain that $\mathcal{E}_{22}$ can be written as
$$
\mathcal{E}_{22}=2\mathfrak{a}_2\int\left(\left(S\times\Delta U^{(N)}+R^{(N)}\right) \times\left[\left(U^{(N)}+S\right)\times\Delta S\right], \Delta g\right)dy.
$$
By \eqref{SUDS}, we can obtain the estimate of $\mathcal{E}_{22}$:
\begin{equation}\label{E3.291}
\left|\mathcal{E}_{22}\right|\leq C\|g\|_{H^1}\left(\|S\|_{H^3}^2+t^N\|S\|_{H^3}\right) \leq Ct^{2\nu}\|S\|_{H^3}^2 +Ct^{3N}.
\end{equation}

For $\mathcal{E}_{23}$, by \eqref{E3.26}, we rewrite $\mathcal{E}_{23}$ as
$$
\mathcal{E}_{23}=\mathcal{E}_{23,1}+\mathcal{E}_{23,2}+\mathcal{E}_{23,3},
$$
where
\begin{align*}
&\mathcal{E}_{23,1}=2\mathfrak{a}_2\int\left(\left(S+U^{(N)}\right)\times \Delta S, \Delta g\right)\left(U^{(N)}+S,\Delta S\right)dy,\\
&\mathcal{E}_{23,2}=-2\mathfrak{a}_2\int\left(\left(S+U^{(N)}\right)\times \left[\left(S\times\Delta U^{(N)}+R^{(N)}\right)\times\Delta S\right], \Delta g\right)dy,\\
&\mathcal{E}_{23,3}=-2\mathfrak{a}_2\int\kappa\left(\left(S+U^{(N)}\right)\times \left(g\times\Delta S\right), g\right)dy.
\end{align*}
For $\mathcal{E}_{23,1}$, we have
\begin{align*}
\mathcal{E}_{23,1}=&-2\mathfrak{a}_2\int\left(\left(S+U^{(N)}\right)\times \Delta S, \Delta g\right)\\
&\qquad\qquad \cdot\left[\left(\Delta U^{(N)}, S\right) +2\left(\partial_j U^{(N)},\partial_j S\right) +\left(\partial_jS,\partial_j S\right)\right]dy\\
=&2\mathfrak{a}_2\int\left(\partial_k \left[\left(S+U^{(N)}\right)\times \Delta S\right], \partial_k g\right)\\
&\qquad\qquad \cdot\left[\left(\Delta U^{(N)}, S\right) +2\left(\partial_j U^{(N)},\partial_j S\right) +\left(\partial_jS,\partial_j S\right)\right]dy\\
&+2\mathfrak{a}_2\int\left(\left(S+U^{(N)}\right)\times \Delta S, \partial_k g\right)\\
&\qquad\qquad \cdot\partial_k\left[\left(\Delta U^{(N)}, S\right) +2\left(\partial_j U^{(N)},\partial_j S\right) +\left(\partial_jS,\partial_j S\right)\right]dy.
\end{align*}
Thus
\begin{equation}\label{EEE231}
\left|\mathcal{E}_{23,1}\right|\leq C\|S\|_{H^3}^2\|g\|_{H^1}\leq Ct^{2\nu}\|S\|_{H^3}^2.
\end{equation}
For $\mathcal{E}_{23,2}$, by \eqref{E3.26}, we get
\begin{equation}\label{EEE232}
\left|\mathcal{E}_{23,2}\right|\leq Ct^{2\nu}\|S\|_{H^3}^2.
\end{equation}
Finally, $\mathcal{E}_{23,3}$ can be estimated as
\begin{equation}\label{EEE233}
\left|\mathcal{E}_{23,3}\right|\leq C\|g\|_{H^1}\left(\|S\|_{H^3}^2+t^N\|S\|_{H^3}\right) \leq Ct^{2\nu}\|S\|_{H^3}^2 +Ct^{3N}.
\end{equation}
Combining \eqref{EEE231}, \eqref{EEE232} and \eqref{EEE233}, we obtain
\begin{equation}\label{EEE23}
\left|\mathcal{E}_{23}\right|\leq C\left(t^{2\nu}\|S\|_{H^3}^2 +t^{3N}\right),
\end{equation}

For $\mathcal{E}_{i}$, $i=24,25,\cdots,28,31$, note that
\begin{align*}
\mathcal{E}_{25}=&-2\mathfrak{a}_2\int\kappa\left(\left(S+U^{(N)}\right)\times \left[\left(S+U^{(N)}\right)\times\Delta g\right], g\right)dy,\\
\mathcal{E}_{31}\leq&-2\mathfrak{a}_2\int\kappa\left(U^{(N)}\times \left(g\times\Delta U^{(N)}\right), g\right)dy,
\end{align*}
we get
\begin{equation}\label{EEE242831}
\begin{split}
&\left|\mathcal{E}_{24}\right|\leq C\|g\|_{H^1}\|S\|_{H^3}^2,\\
&\left|\mathcal{E}_{25}, \mathcal{E}_{26}\right|\leq C\|g\|_{H^1}^2\|S\|_{H^3}\leq Ct^{2\nu}\|g\|_{H^1}^2,\\
&\left|\mathcal{E}_{27}\right|\leq Ct^{2\nu} \|g\|_{H^1}\|S\|_{H^3}+Ct^{3N},\\
&\left|\mathcal{E}_{28}\right|\leq Ct^{2\nu} \|g\|_{H^1}^2,\\
&\left|\mathcal{E}_{31}\right| \leq Ct^{2\nu}\|g\|_{L^2}^2.
\end{split}
\end{equation}

For $\mathcal{E}_i$, $i=29,30,32$, we get
\begin{align*}
\left|\mathcal{E}_{29}, \mathcal{E}_{30}, \mathcal{E}_{32}\right| \leq &C\left(\left\|\Delta\chi^{(N)}\right\|_{W^{2,\infty}} +\left\|R^{(N)}\right\|_{H^3}\right) \|g\|_{H^1}\|S\|_{H^3}\\
&+C\left\|\langle y\rangle^{-1}\nabla\Delta^2\chi^{(N)}\right\|_{L^\infty} \|\nabla g\|_{L^2} \|\langle y\rangle S\|_{L^2}.
\end{align*}
Thus,
\begin{equation}\label{EEE293032}
\left|\mathcal{E}_{29}, \mathcal{E}_{30}, \mathcal{E}_{32}\right| \leq Ct^{2\nu}\left(\|g\|_{H^1}\|S\|_{H^3} +\|\nabla g\|_{L^2} \|\langle y\rangle S\|_{L^2}\right).
\end{equation}

Combining \eqref{E3.25}, \eqref{E3.30}, \eqref{EEE21}, \eqref{E3.291}, \eqref{EEE23}, \eqref{EEE242831} and \eqref{EEE293032}, we get
\begin{equation}\label{E3.31}
\left|\frac{d}{dt}J_3(t)\right|\leq C\frac{1}{t} \left[\|S\|_{H^3}^2+ \left(\|S\|_{H^3} +t^{N+1+2\nu}\right) \left\||y|S\right\|_{L^2}\right] +Ct^{2N+2\nu}.
\end{equation}

\subsubsection{The proof Proposition \ref{pro3.1}}
To prove Proposition \ref{pro3.1}, it is sufficient to prove the bootstrap hypotheses \eqref{E3.5} and \eqref{E3.18} imply \eqref{E3.2} and \eqref{E3.3}.

According to the bootstrap hypothesis \eqref{E3.18} and the estimates \eqref{E3.10} and \eqref{E3.16}, we obtain that for any $N$ sufficiently large and $t_0$ sufficiently small,
\begin{equation}\label{E3.32}
\sum_{i=0}^1\left|\frac{d}{dt}J_i(t)\right|\leq C\frac{1}{t} \|S\|_{H^1}^2 +Ct^{2N-2\nu},\quad \forall t\leq t_0.
\end{equation}

Note that for any $c_0>0$ sufficiently large, we have
$$
\|S\|_{H^1}^2\leq J_1+c_0J_0.
$$
Let
$$
J(t)=J_1(t)+c_0J_0(t),
$$
Then \eqref{E3.32} can be written as
\begin{equation}\label{E3.33}
\left|\frac{d}{dt}J(t)\right|\leq C\frac{1}{t} J(t) +Ct^{2N-2\nu}.
\end{equation}
Integrating the two sides of \eqref{E3.33}, where the zero initial condition is satisfied at $t_1$, we obtain that for sufficiently large $N$,
\begin{equation}\label{E3.34}
J(t)\leq \frac{C}{N}t^{2N+1-2\nu},\quad \forall t\in[t_1,t_0].
\end{equation}
Thus,
\begin{equation}\label{E3.35}
\|S\|_{H^1}^2\leq \frac{C}{N}t^{2N+1-2\nu},\quad \forall t\in[t_1,t_0],
\end{equation}

For $\left\||y|S(t)\right\|_{L^2}$, by \eqref{E3.17} and \eqref{E3.35}, we get
\begin{equation}\label{E3.36}
\left|\frac{d}{dt}\left\||y|S(t)\right\|_{L^2}^2\right| \leq C\frac{1}{t}\left(\left\||y|S(t)\right\|_{L^2}^2+ t^{2N+1-6\nu}\right).
\end{equation}
Integrating both sides of \eqref{E3.36}, we obtain that for $N$ sufficiently large,
\begin{equation}\label{E3.37}
\left\||y|S(t)\right\|_{L^2}^2 \leq \frac{C}{N} t^{2N+1-6\nu}, \quad \forall t\in[t_1,t_0],
\end{equation}
thus,
\begin{equation}\label{E3.38}
\left\||x|s(t)\right\|_{L^2}^2 \leq t^{\frac{N}{2}}, \quad \forall t\in[t_1,t_0].
\end{equation}
Next, we consider $\|\nabla\Delta s(t)\|_{L^2(\mathbb{R}^2)}$. By \eqref{E3.23} and \eqref{E3.18}, we obtain that for any $j=1,2$,
\begin{equation}\label{E3.39}
\left\|\partial_jg-\left(U^{(N)}+S\right) \times\Delta\partial_j S\right\|_{L^2} \leq C\left(\|S\|_{H^2(\mathbb{R}^2)}+t^{N+1+2\nu}\right).
\end{equation}
Note that $\left|U^{(N)}+S\right|=1$. Thus,
\begin{align*}
\left|\left(U^{(N)}+S\right)\times\Delta\partial_j S\right|^2 =\left|\Delta\partial_j S\right|^2-\left(U^{(N)}+S,\Delta\partial_jS\right)^2.
\end{align*}
This combined with \eqref{E3.18} and \eqref{SUDSa} gives
\begin{equation}\label{E3.40}
\left\|\Delta\partial_jg\right\|_{L^2}^2 -\left\|\left(U^{(N)}+S\right)\times\Delta\partial_j S\right\|_{L^2}^2 \leq C\|S\|_{H^2(\mathbb{R}^2)}^2.
\end{equation}

Now we consider the functional $\tilde{J}_3(t)=J_3+c_1J_0(t)$. By \eqref{E3.24}, \eqref{E3.39} and \eqref{E3.40}, we obtain that for $c_1>0$ sufficiently large, there exists $c_2>0$ such that
\begin{equation}\label{E3.41}
c_2\|S\|_{H^3}^2 -Ct^{2N+1+2\nu} \leq \tilde{J}_3(t)\leq C\left(\|S\|_{H^3(\mathbb{R}^2)}^2+t^{2N+1+2\nu}\right).
\end{equation}

By \eqref{E3.31}, \eqref{E3.32} and \eqref{E3.37}, we get
\begin{equation}\label{E3.42}
\begin{split}
\left|\frac{d}{dt}\tilde{J}_3(t)\right|\leq &C\left[\frac{1}{t}\left(\|S\|_{H^3{(\mathbb{R}^2)}}^2 +\left\||y|S\right\|_{L^2(\mathbb{R}^2)}\right)+t^{2N-2\nu}\right]\\
\leq &C\frac{1}{t}\tilde{J}_3(t)+ Ct^{2N-6\nu}.
\end{split}
\end{equation}
Integrating the two sides of \eqref{E3.42} with respect to $t$ from $t_1$ to $t$, and noting that
$$
\tilde{J}_3(t_1)=t_1^{2+4\nu} \int\left|\nabla r^{(N)}(x,t_1)\right|^2dx +t_1^{1+2\nu}\int\kappa\left(t^{-\frac{1}{2}+\nu}x\right)\left|r^{(N)}(x,t_1)\right|^2,
$$
we get
$$
\left|\tilde{J}_3(t_1)\right|\leq Ct_1^{2N+1+2\nu}.
$$
Thus,
$$
\tilde{J}_3(t)\leq C t^{2N+1-6\nu},\quad \forall t\in[t_1,t_0].
$$
This combined with \eqref{E3.41} gives
$$
\|S\|_{H^3(\mathbb{R}^2)}^2\leq C t^{2N+1-6\nu},\quad \forall t\in[t_1,t_0],
$$
thus,
$$
\|s\|_{H^3(\mathbb{R}^2)}\leq  t^{\frac{N}{2}},\quad \forall t\in[t_1,t_0].
$$
This completes the proof of Proposition \ref{pro3.1}.
\end{proof}

\subsection{Proof of the main theorem}
Now, we start to prove the main theorem of this paper. Fix $N$ such that Proposition \ref{pro3.1} holds. Select the sequence $\{t^j\}$, $0<t^j<t_0$, satisfying $t^j\rightarrow 0$ as $j\rightarrow\infty$时. Let $u_j(x,t)$ be a solution of the following problem:
\begin{equation}\label{E3.43}
\begin{cases}
&\partial_tu_j= \mathfrak{a}_1u_j\times\Delta u_j-\mathfrak{a}_2u_j\times(u_j\times\Delta u_j),\quad t\geq t^j,\\
&u_j|_{t=t^j}=u^{(N)}(t^j).
\end{cases}
\end{equation}
By Proposition \ref{pro3.1}, for any $j$, there exists $u_j-u^{(N)}\in C([t^j,t_0],H^3)$ satisfying
\begin{equation}\label{E3.44}
\left\|u_j(t)-u^{(N)}(t)\right\|_{H^3}+ \left\|\langle x\rangle\left(u_j(t)-u^{(N)}(t)\right)\right\|_{L^2} \leq2t^{\frac{N}{2}},\quad \forall t\in\left[t^j,t_0\right].
\end{equation}
Thus, the sequence $u_j(t_0)-u^{(N)}(t_0)$ is compactness in $H^2$. This ensures that we can select a subsequence and pass the limit such that $u_j(t_0)-u^{(N)}(t_0)$ converges to some 1-equivariant function $w\in H^3$ in $H^2$, where $\|w\|_{H^3}\leq\delta^{2\nu}$ and $\left|u^{(N)}(t_0)+w\right|=1$.

For the Cauchy problem:
\begin{equation}\label{E3.45}
\begin{cases}
&\partial_tu=\mathfrak{a}_1u\times\Delta u-\mathfrak{a}_2u\times(u\times\Delta u),\quad t\geq t_0,\\
&u|_{t=t_0}=u^{(N)}(t_0)+w,
\end{cases}
\end{equation}
by the classical local well-posedness theory, \eqref{E3.45} exists a unique solution $u\in C((t^*,t_0], \dot{H}^1\cap\dot{H}^3)$ for some $0\leq t^*<t_0$. According to the $H^1$ continuity of the Landau--Lifshitz flow, we have $u_j\rightarrow u$ in $C((t^*,t_0], \dot{H}^1)$. This combined with \eqref{E3.44} gives
\begin{equation}\label{E3.46}
\left\|u(t)-u^{(N)}(t)\right\|_{H^3} \leq2t^{\frac{N}{2}},\quad \forall t\in\left(t^*,t_0\right].
\end{equation}
Thus, $t^*=0$. This combined with Proposition \ref{pro2.1} gives Theorem \ref{MT}.

\section*{Acknowledgements}
This work was supported by National Natural Science Foundation of China (Grant Nos. 12231016 and 12071391) and Guangdong Basic and Applied Basic Research Foundation (Grant No. 2022A1515010860).

\bibliographystyle{amsplain}

\end{document}